\documentclass[11pt]{article}
\usepackage{amssymb,amsmath,amsthm}
\usepackage{mathrsfs}
\usepackage{graphicx}
\usepackage[margin=10pt,font=small,labelfont=bf,
            labelsep=endash]{caption}
\usepackage{dsfont}
\usepackage{color}
\numberwithin{equation}{section}
\newtheorem{thm}{Theorem}[section]

\newtheorem{prop}[thm]{Proposition}
\newtheorem{lem}[thm]{Lemma}
\newtheorem{cor}[thm]{Corollary}
\theoremstyle{definition}
\newtheorem{rem}[thm]{Remark}
\newtheorem{rems}[thm]{Remarks}
\let\oldproofname=\proofname
\renewcommand{\proofname}{\rm\bf{\oldproofname}}

\newcommand{\N}{\mathbb{N}}

\newcommand{\R}{\mathbb{R}}
\renewcommand{\S}{\mathbb{S}}
\newcommand{\C}{\mathbb{C}}

\newcommand{\cB}{\mathcal{B}}

\newcommand{\cD}{\mathcal{D}}
\newcommand{\cE}{\mathcal{E}}
\newcommand{\cF}{\mathcal{F}}

\newcommand{\cK}{\mathcal{K}}
\newcommand{\cL}{\mathcal{L}}
\newcommand{\cM}{\mathcal{M}}
\newcommand{\cN}{\mathcal{N}}
\newcommand{\cO}{\mathcal{O}}

\renewcommand{\Re}{\mathop{\mathrm{Re}}}
\renewcommand{\Im}{\mathop{\mathrm{Im}}}
\renewcommand{\div}{\mathop{\mathrm{div}}}
\newcommand{\dd}{\,{\rm d}}
\newcommand{\D}{{\rm d}}
\newcommand{\TS}{\textstyle}

\newcommand{\QED}{\mbox{}\hfill$\Box$}
\renewcommand{\:}{\thinspace :}
\newcommand{\1}{\mathds{1}}
\newcommand{\loc}{{\rm loc}}
\newcommand{\disc}{{\rm disc}}
\newcommand{\ess}{{\rm ess}}
\newcommand{\Langle}{\bigl\langle}
\newcommand{\Rangle}{\bigr\rangle}
\renewcommand{\phi}{\varphi}
\newcommand{\EE}{E^{(2)}}
\newcommand{\WW}{W^{(2)}}

\begin{document}

\title{Asymptotic self-similarity in diffusion equations
 with nonconstant radial limits at infinity}

\author{
{\bf Thierry Gallay, Romain Joly, and Genevi\`eve Raugel}
}

\date{May 27, 2020}
\maketitle

\begin{abstract}
We study the long-time behavior of localized solutions to linear or
semilinear parabolic equations in the whole space $\R^n$, where
$n \ge 2$, assuming that the diffusion matrix depends on the space
variable $x$ and has a finite limit along any ray as $|x| \to \infty$.
Under suitable smallness conditions in the nonlinear case, we
prove convergence to a self-similar solution whose profile is entirely
determined by the asymptotic diffusion matrix. Examples are given
which show that the profile can be a rather general Gaussian-like
function, and that the approach to the self-similar solution can be
arbitrarily slow depending on the continuity and coercivity properties
of the asymptotic matrix. The proof of our results relies on appropriate
energy estimates for the diffusion equation in self-similar
variables. The new ingredient consists in estimating not only the
difference $w$ between the solution and the self-similar profile, but
also an antiderivative $W$ obtained by solving a linear elliptic
problem which involves $w$ as a source term. Hence, a good part of our
analysis is devoted to the study of linear elliptic equations
whose coefficients are homogeneous of degree zero.
\end{abstract}

\section{Introduction}\label{sec1}

We consider semilinear parabolic equations of the form
\begin{equation}\label{diffeq}
  \partial_t u(x,t) \,=\, \div \bigl(A(x)\nabla u(x,t)\bigr) + 
  N(u(x,t))\,, \qquad x \in \R^n\,,\quad t > 0\,,
\end{equation}
which describe the evolution of a scalar quantity $u(x,t) \in \R$ under
the action of inhomogeneous diffusion and nonlinear self-interaction.
We assume that the diffusion matrix $A(x)$ in \eqref{diffeq} is
symmetric, Lipschitz continuous as a function of $x \in \R^n$, and
satisfies the following uniform ellipticity condition\: there
exist positive constants $\lambda_1,\lambda_2$ such that
\begin{equation}\label{Aelliptic}
  \lambda_1 |\xi|^2 \,\le\, \bigl(A(x)\xi,\xi) 
  \,\le\, \lambda_2 |\xi|^2\,, \qquad \hbox{for all } 
  x\in \R^n \hbox{ and all }\xi \in \R^n\,,
\end{equation}
where $(\cdot,\cdot)$ denotes the Euclidean scalar product in
$\R^n$. As for the nonlinearity, we suppose that $N$ is globally
Lipschitz, that $N(0) = 0$, and that $N(u) = \cO(|u|^\sigma)$ as $u
\to 0$ for some $\sigma > 1+2/n$. Our goal is to investigate the long-time
behavior of all solutions of \eqref{diffeq} starting from sufficiently
small and localized initial data.

Even in the linear case where $N = 0$, it is necessary to make further
assumptions on the diffusion matrix $A(x)$ to obtain accurate results
on the long-time behavior of solutions of \eqref{diffeq}. In fact,
two classical situations are well understood\: the {\em asymptotically
  flat} case, and the {\em periodic} case. More precisely, if $A(x)$
converges to the identity matrix as $|x| \to \infty$, it is possible
to show that all solutions of the diffusion equation $\partial_t u =
\div(A(x)\nabla u)$ in $L^1(\R^n)$ behave asymptotically like the
solutions of the heat equation $\partial_t u = \Delta u$ with the same
initial data, see e.g. \cite{DZ1}. On the other hand, if $A(x)$ is a
periodic function of $x$ with respect to a lattice of $\R^n$, the
relevant asymptotic equation is $\partial_t u = \div(\bar A\,\nabla
u)$, where $\bar A \in \cM_n(\R)$ is a {\em homogenized matrix}
which is determined by solving an elliptic problem in a cell of
the lattice \cite{DZ2}. These results can be extended to a class of
semilinear equations as well \cite{DC,DZ1,DZ2}.

In this paper, we consider a different situation which is apparently
less studied in the literature\: we assume that the diffusion matrix
$A(x)$ has {\em radial limits} at infinity in all directions. This means
that, for all $x \in \R^n$, the following limit exists\:
\begin{equation}\label{Alimits}
  A_\infty(x) \,:=\, \lim_{r \to +\infty} A(rx)\,.
\end{equation}
It is clear that the limiting matrix $A_\infty(x)$ is symmetric,
homogeneous of degree zero with respect to $x \in \R^n$, and uniformly
elliptic in the sense of \eqref{Aelliptic}. We also suppose that
the restriction of $A_\infty$ to the unit sphere $\S^{n-1} \subset
\R^n$ is Lipschitz continuous, and that the limit in \eqref{Alimits}
is reached uniformly on $\S^{n-1}$ at some rate $\nu > 0$\:
\begin{equation}\label{Aunif}
  \sup_{x \in \R^n} |x|^\nu \,\|A(x) - A_\infty(x)\| \,<\, \infty\,.
\end{equation}

Following \cite{Wa,GR}, to investigate the long-time behavior of solutions
to \eqref{diffeq}, we introduce forward {\em self-similar variables} defined
by $y = x/\sqrt{1+t}$ and $\tau = \log(1+t)$. More precisely, we look for
solutions of \eqref{diffeq} in the form
\begin{equation}\label{self-similar}
  u(x,t) \,=\, \frac{1}{(1+t)^{n/2}}\, v\Bigl(\frac{x}{\sqrt{1+t}}\,,
  \,\log(1+t)\Bigr)\,, \qquad x \in \R^n\,, \quad t \ge 0\,.  
\end{equation}
Note that the change of variables \eqref{self-similar} reduces to
identity at initial time, so that $u(x,0) = v(x,0)$. The new function
$v(y,\tau)$ satisfies the rescaled equation
\begin{equation}\label{veq}
  \partial_\tau v \,=\, \div\Bigl(A\bigl(ye^{\tau/2}\bigr)\nabla v\Bigr)
  + \frac{1}{2}\,y\cdot\nabla v + \frac{n}{2}\,v + \cN(\tau,v)\,,
  \qquad y \in \R^n\,,\quad \tau > 0\,,
\end{equation}
where
\begin{equation}\label{cNdef}
  \cN(\tau,v) \,=\, e^{(1 + \frac{n}{2})\tau}\,N\bigl(e^{-n\tau/2}
  v\bigr)\,.
\end{equation}
Equation \eqref{veq} is non-autonomous, but has (at least formally)
a well-defined limit as $\tau \to +\infty$. Indeed, using \eqref{Alimits}
and the assumption that $N(u) = \cO(|u|^\sigma)$ as $u \to 0$ for some 
$\sigma > 1+2/n$, we arrive at the limiting equation
\begin{equation}\label{veqlimit}
  \partial_\tau v \,=\, \div\bigl(A_\infty(y)\nabla v\bigr)
  + \frac{1}{2}\,y\cdot\nabla v + \frac{n}{2}\,v\,, 
  \qquad y \in \R^n\,,\quad \tau > 0\,.
\end{equation}
In what follows we denote by $L$ the differential operator in the
right-hand side of \eqref{veqlimit}. 

Our main results show that, under appropriate assumptions, the
solutions of \eqref{veq} indeed converge to solutions of \eqref{veqlimit} 
as $\tau \to \infty$, so that the long-time
asymptotics are determined by the linear equation \eqref{veqlimit}. 
We first observe that the limiting equation has a unique 
steady state\:

\begin{prop}\label{phiprop}
There exists a unique solution $\phi \in H^1(\R^n) \cap L^1(\R^n)$ 
of the elliptic equation
\begin{equation}\label{phieq}
  L \phi(y) \,\equiv\, 
  \div \bigl(A_\infty(y)\nabla \phi(y)\bigr) + \frac12 \,y \cdot \nabla
  \,\phi(y) + \frac{n}{2}\,\phi(y) \,=\, 0\,, \qquad y \in \R^n\,,
\end{equation}
satisfying the normalization condition $\int_{\R^n} \phi(y) \dd y = 1$. 
Moreover $\phi$ is H\"older continuous, and there exists a constant
$C \ge 1$ such that 
\begin{equation}\label{phibounds}
  C^{-1}\,e^{-C|y|^2} \,\le\, \phi(y) \,\le\, C\,e^{-|y|^2/C}\,,
  \qquad \hbox{for all } y \in \R^n\,.
\end{equation}
\end{prop}

\begin{rem}\label{idrem}
If we suppose that $A_\infty(y) = \1$ (the identity matrix), or more
generally that $A_\infty(y)y = y$ for all $y \in \R^n$, the ``principal
eigenfunction'' $\phi$ defined in Proposition~\ref{phiprop} is given
by the explicit formula $\phi(y) = (4\pi)^{-n/2}\,e^{-|y|^2/4}$. In
contrast, we show in Remark~\ref{principal} below that, if $B(y)$ 
is a symmetric matrix that is homogeneous of degree zero and uniformly
elliptic, the Gaussian-like function $\phi(y) = \exp\bigl(-\frac14
(B(y)y,y)\bigr)$ satisfies \eqref{phieq} for some appropriate choice
of the limiting matrix $A_\infty$, provided the oscillations of
$B(y)$ are not too rapid. This indicates that the profile $\phi$ given
by Proposition~\ref{phiprop} can be a pretty general function
satisfying the Gaussian bounds \eqref{phibounds}.
\end{rem}

We next consider solutions of \eqref{veq} in the weighted $L^2$ space 
\begin{equation}\label{L2mdef}
  L^2(m) \,=\, \Bigl\{v \in L^2_\loc(\R^n)\,\Big|\, \|v\|_{L^2(m)} < \infty
  \Bigr\}\,, \quad \|v\|_{L^2(m)}^2 \,=\,
  \int_{\R^n} (1+|y|^2)^m |v(y)|^2\dd y\,,
\end{equation}
which was used in a similar context in \cite{GWa}. The parameter $m \in \R$
specifies the behavior of the solutions at infinity. In particular, we 
observe that $L^2(m) \hookrightarrow L^1(\R^n)$ when $m > n/2$, as a
consequence of H\"older's inequality.

We are now ready to state our main result in the linear case where $\cN = 0$. 

\begin{thm}\label{main1} {\em (Asymptotics in the linear case)}\\
Assume that $n \ge 2$ and that the diffusion matrix $A(x)$ satisfies 
hypotheses \eqref{Aelliptic}--\eqref{Aunif}. For all $m > n/2$ and all initial
data $v_0 \in L^2(m)$, the rescaled equation \eqref{veq} with $\cN = 0$ has 
a unique global solution $v \in C^0([0,+\infty),L^2(m))$ such that 
$v(0) = v_0$. Moreover, for any $\mu$ satisfying
\begin{equation}\label{mudef1}
  0 \,<\, \mu \,<\, \frac12\,\min\Bigl(m - \frac{n}{2}\,,\,\nu\,,\,
  \beta\Bigr)\,,
\end{equation}
where $\nu > 0$ is as in \eqref{Aunif} and $\beta \in (0,1]$ is the
exponent in \eqref{Ggrad} below, there exists a positive constant
$C$ (independent of $v_0$) such that
\begin{equation}\label{mainconv1}
  \|v(\cdot,\tau) - \alpha \phi\|_{L^2(m)} \,\le\, C\,\|v_0\|_{L^2(m)}
  \,e^{-\mu \tau}\,, \qquad \hbox{for all } \tau \ge 0\,,
\end{equation}
where $\alpha = \int_{\R^n} v_0(y) \dd y$ and $\phi$ is given by 
Proposition~\ref{phiprop}. 
\end{thm}

\begin{rem}\label{original}
In terms of the original variables, the convergence result \eqref{mainconv1}
implies in particular that, in the linear case $N = 0$, the solution
$u(x,t)$ of \eqref{diffeq} with initial data $u_0 \in L^2(m)$ satisfies
\begin{equation}\label{originalconv}
  \int_{\R^n} \Bigl|u(x,t) - \frac{\alpha}{(1+t)^{n/2}}\,\phi\Bigl(
  \frac{x}{\sqrt{1+t}}\Bigr)\Bigr| \dd x \,=\, \cO(t^{-\mu})\,,
  \qquad \hbox{as} \quad t \to +\infty\,,
\end{equation}
where $\alpha =  \int_{\R^n} u_0(x) \dd x$. Using parabolic regularity,
it is possible to prove convergence in higher $L^p$ norms too, as
in \cite{DZ1}.  
\end{rem}

\begin{rem}\label{higher}
Theorem~\ref{main1} holds true in all space dimensions $n \ge 1$, but
the proof we propose only works for $n \ge 2$ and depends on $n$ in a
nontrivial way. In fact, as we shall see in Section~\ref{sec4} below,
the number of energy functionals we need increases with $n$, so that
our method becomes cumbersome in high dimensions. For simplicity we
concentrate on the most relevant cases $n = 2$ and $n = 3$, for which
we provide a complete proof, but we also give a pretty detailed sketch
of the argument when $4 \le n \le 7$, see Section~\ref{ssec45}. On the
other hand, the one-dimensional case, which is substantially simpler
for several reasons, is completely solved in our previous work
\cite{GR}, where damped hyperbolic equations are also considered. In
many respects, the present paper can be viewed as a (rather
nontrivial) extension of the method of \cite{GR} to higher dimensions.
\end{rem}

Before considering semilinear equations, we comment on the formula
\eqref{mudef1} for the convergence rate $\mu$, which is quite
instructive. We first recall that, for any measurable matrix $A(x)$
satisfying the ellipticity conditions \eqref{Aelliptic}, the solutions
of the linear equation $\partial_t u = \div \bigl(A\nabla u\bigr)$
with localized initial data satisfy $\|u(\cdot,t)\|_{L^\infty} =
\cO(t^{-n/2})$ as $t \to +\infty$, see for instance \cite{FS}. The
purpose of Theorem~\ref{main1} is to exhibit the leading-order term in
the asymptotic expansion of $u(x,t)$, and to estimate the rate $\mu$
at which the leading term is approached by the solutions. As can be
seen from the simple example of the heat equation, where $A = \1$, the
convergence rate $\mu$ depends on how fast the initial data decay as
$|x| \to \infty$. More precisely, it is known in that example that
Theorem~\ref{main1} holds for any $\mu \le 1/2$ such that $2\mu < m -
\frac{n}{2}$ \cite{GWa}. This result is sharp and the constraints on
$\mu$ are determined by the spectral properties of the differential
operator $L$ in \eqref{veqlimit}, considered as acting on the weighted
space $L^2(m)$. If $m > n/2$, so that $L^2(m) \hookrightarrow
L^1(\R^n)$, the origin $\lambda = 0$ is a simple isolated eigenvalue,
with Gaussian eigenfunction $\phi$ as in Remark~\ref{idrem}. The
convergence rate $\mu$ is determined by the {\em spectral gap} between
the origin and the rest of the spectrum of $L$, see Figure~\ref{fig1} in
Section~\ref{sec3}.

In more general situations, the convergence rate $\mu$ obviously
depends on how fast the limits in \eqref{Alimits} are reached. This
effect can be studied using the techniques of \cite{DZ1} if we assume
that $A(x) = \1 + B(x)$, where $\|B(x)\| = \cO(|x|^{-\nu})$ as $|x|
\to \infty$. In that case, the solutions of the linear equation
$\partial_t u = \div \bigl(A\nabla u\bigr)$ in $L^2(m)$ behave
asymptotically like the solutions of the heat equation $\partial_t u =
\Delta u$ with the same initial data, but the convergence rate in
\eqref{mainconv1} or \eqref{originalconv} is further constrained by
the relation $\mu \le \nu/2$, which appears to be sharp. As can be
expected, we thus have $\mu \to 0$ as $\nu \to 0$.

Finally, it is important to realize that the convergence rate $\mu$
also depends on the properties of the limiting matrix $A_\infty(x)$
itself, and cannot be arbitrarily large even if $A = A_\infty$ and
$m \gg n/2$. We have already seen that $\mu \le 1/2$ when $A_\infty = \1$,
due to the presence of an isolated eigenvalue $\lambda = -1/2$ in
the spectrum of $L$ if $m > 1 + n/2$, see Figure~\ref{fig1}. For
a more general matrix $A_\infty(x)$, the principal eigenvalue of 
the corresponding operator $L$ is fixed at the origin, as asserted
by Proposition~\ref{phiprop}, but the next eigenvalue can be
pretty arbitrary, and this determines the width of the spectral gap.
In Section~\ref{ssec32}, we study an instructive example for which
\begin{equation}\label{MSexample}
  A_\infty(x) \,=\, b\,\1 + (1-b) \frac{x \otimes x}{|x|^2}\,,
\end{equation}
where $b > 0$ is a free parameter. In that case, we can compute
explicitly all eigenvalues and eigenfunctions of the linear operator
$L$ in \eqref{veqlimit}, and we observe that the spectral gap shrinks
to zero as $b \to 0$, see Figure~\ref{fig2}.

The example \eqref{MSexample} is already considered in classical
papers by Meyers \cite{Me} and Serrin \cite{Se}, where uniqueness and
regularity properties are studied for the solutions of the linear
elliptic equation $-\div\bigl(A_\infty(x)\nabla u\bigr) = f$ in
$\R^n$. It turns out that this equation plays a crucial role in our
analysis because, as we shall see in Section~\ref{sec4}, the
convergence result \eqref{mainconv1} is obtained using energy
estimates not only for the difference $w = v - \alpha \phi$,
but also for the ``antiderivative'' $W$ defined by $-\div\bigl(
A_\infty(x)\nabla W\bigr) = w$. It is important to keep in mind that
the matrix $A_\infty(x)$, being homogeneous of degree zero, is not
smooth at the origin unless it is constant. So we do not expect that
the solutions of the elliptic equation above are smooth, even if $f$
is, but the celebrated De Giorgi-Nash theory asserts that all weak
solutions in $H^1_\loc(\R^n)$ are at least H\"older continuous with
exponent $\beta$, for some $\beta \in (0,1)$. This exponent is the
third quantity that appears in the formula \eqref{mudef1} for the
convergence rate.  Consequently, Theorem~\ref{main1} draws an original
connection between the regularity properties of the elliptic problem
and the long-time behavior of the solutions of the evolution equation.

To study the elliptic problem, we consider the associated Green function
$G(x,y)$, which is uniquely defined at least if $n \ge 3$. For the
reasons mentioned above, that function is H\"older continuous with
exponent $\beta$, but not more regular unless $A_\infty$ is constant.
However, using the assumption that $A_\infty$ is homogeneous of
degree zero and Lipschitz outside the origin, it is possible to
establish the following gradient estimate
\begin{equation}\label{Ggrad}
  |\nabla_x G(x,y)| \,\le\, C\biggl(\frac{1}{|x-y|^{n-1}} +  
  \frac{1}{|x|^{1-\beta}|x-y|^{n-2+\beta}}\biggr)\,, \qquad
  x \neq y\,, \quad x \neq 0\,,
\end{equation}
where the second term in the right-hand side describes the precise
nature of the singularity at the origin. As is well known, the Green
function of the Laplace operator satisfies \eqref{Ggrad} with $\beta =
1$, but for nonconstant homogeneous matrices $A_\infty(x)$ we have
$\beta < 1$ in general. Estimate \eqref{Ggrad} is apparently new and
plays an important role in our analysis of the elliptic problem, hence
in the proof of Theorem~\ref{main1}.

Although we only considered linear equations so far, the techniques we
use in the proof of Theorem~\ref{main1} are genuinely nonlinear, and
were originally developed to handle semilinear problems, see
\cite{Wa,GR}. To illustrate the scope of our method, we also treat the
full equation \eqref{diffeq} with a nonlinearity $N$ that is
``irrelevant'' for the long-time asymptotics of small and localized
solutions, according to the terminology introduced in \cite{BKL}. For
simplicity, we make here rather strong assumptions on $N$, which
could be relaxed at expense of using additional energy functionals
in the proof. We suppose that there exist two constants $C > 0$ and
$\sigma > 1 + 2/n$ such that
\begin{equation}\label{Nprop}
  \bigl|N(u)\bigr| \,\le\, C |u|^\sigma \quad \hbox{and}\quad
  \bigl|N(u) - N(\tilde u)\bigr| \,\le\, C |u-\tilde u|\,,
  \qquad \hbox{for all } u, \tilde u \in \R\,.
\end{equation}

Our second main result is the following\:

\begin{thm}\label{main2} {\em (Asymptotics in the semilinear case)}\\
Assume that $n \ge 2$ and the diffusion matrix $A(x)$ satisfies hypotheses 
\eqref{Aelliptic}--\eqref{Aunif}, and that conditions
\eqref{Nprop} are fulfilled by the nonlinearity $N$. Given any $m > n/2$,
there exist a positive constant $\epsilon_0$ such that, for all initial
data $v_0 \in L^2(m)$ with $\|v_0\|_{L^2(m)} \le \epsilon_0$, the rescaled
equation \eqref{veq} has a unique global solution $v \in C^0([0,+\infty),
L^2(m))$ such that $v(0) = v_0$. Moreover, there exists some $\alpha_*
\in \R$ and, for all $\mu$ satisfying
\begin{equation}\label{mudef2}
  0 \,<\, \mu \,<\, \frac12\,\min\Bigl(m - \frac{n}{2}\,,\,\nu\,,\,
  \beta\,,\,2\eta\Bigr)\,, \qquad\hbox{where}\quad \eta \,=\,
  \frac{n}{2}\bigl(\sigma - 1\bigr) - 1\,,
\end{equation}
there exists a positive constant $C$ (independent of $v_0$) such that 
\begin{equation}\label{mainconv2}
  \|v(\cdot,\tau) - \alpha_* \phi\|_{L^2(m)} \,\le\, C\,\|v_0\|_{L^2(m)}
  \,e^{-\mu \tau}\,, \qquad \hbox{for all } \tau \ge 0\,,
\end{equation}
where $\phi$ is given by Proposition~\ref{phiprop}. 
\end{thm}

\begin{rem}\label{changes}
The integral of $u$ is not preserved under the nonlinear evolution
defined by \eqref{diffeq}, and this explains why there is no formula
for the asymptotic mass $\alpha_*$ in Theorem~\ref{main2}. However,
the proof shows that $\alpha_* = \int_{\R^n}v_0\dd y + 
\cO(\|v_0\|_{L^2(m)}^\sigma)$, where $\sigma$ is as in \eqref{Nprop}.
It is important to observe that the convergence rate $\mu$ in 
\eqref{mudef2} is also affected by the nonlinearity, through the value 
of the parameter $\sigma$. In particular $\mu$ converges to zero as $\sigma$ 
approaches from above the critical value $1+2/n$, and no convergence at 
all is expected if $\sigma \le 1 + 2/n$. 
\end{rem}

\begin{rem}\label{lowdim}
As in the linear case, our strategy to prove Theorem~\ref{main2} becomes
complicated in large space dimensions. For simplicity we provide a complete 
proof only if $n = 2$, or if $n = 3$ and $\mu < 1/4$. The other cases can be 
treated using the hierarchy of energy functionals introduced in 
Section~\ref{ssec45}. 
\end{rem}

\smallskip
The rest of this paper is organized as follows. In Section~\ref{sec2},
we study in some detail the elliptic equation $-\div\bigl(A_\infty
\nabla u\bigr) = f$ under the assumption that the matrix $A_\infty(x)$
is homogeneous of degree zero and uniformly elliptic. In particular,
we derive estimates for the associated Green function, and we apply
them to bound the solution $u$ in terms of the data $f$ in weighted
$L^2$ spaces. In this process we use a general result on integral
operators with homogeneous kernels, which is essentially due to
Karapetiants and Samko \cite{KS}. In Section~\ref{sec3}, we
investigate the spectral properties of the linear operator defined by
the right-hand side of \eqref{veqlimit}; in particular, we prove
Proposition~\ref{phiprop} and we establish a few additional properties
of the principal eigenfunction $\phi$. We also study in detail the
particular case where the matrix $A_\infty$ is given by
\eqref{MSexample}. Section~\ref{sec4} is devoted to the proof of
Theorem~\ref{main1}, using weighted energy estimates for the
perturbation $w = v - \alpha\phi$. As was already mentioned, the main
original idea is to introduce the ``antiderivative'' $W$, which is
defined as the solution of the elliptic equation $-\div\bigl(A_\infty
\nabla W \bigr) = w$. It turns out that weighted $L^2$ estimates for
both $W$ and $w$ are sufficient to establish the convergence result
 \eqref{mainconv1} if $n = 2$, or if $n = 3$ and $\mu < 1/4$, 
whereas additional energy functionals are needed in the other cases. 
The same strategy works in the nonlinear case too, under suitable
assumptions on the function $N$, and the details are worked out in
Section~\ref{sec5}. The final Section~\ref{sec6} is an appendix where
a few auxiliary results are collected for easy reference.

\medskip
\noindent{\bf Acknowledgements.}
This project started more than 15 years ago, but was left aside for a
long time due to other priorities. Paradoxically, the untimely demise
of Genevi\`eve Raugel in spring 2019 gave a new impetus to the
subject. The authors are indebted to Marius Paicu for his active
participation at the early stage of this project, and to Emmanuel Russ
for constant help on many technical questions.  All three authors were
supported by the project ISDEEC ANR-16-CE40-0013 of the French
Ministry of Higher Education, Research and Innovation.

%%%%%%%%%%%%%%%%%%%%%%%%%%%%%%%%%%%%%%%%%%%%%%%%%%%%%%%%%%%%%%%%%%%%%

\section{The diffusion operator with homogeneous 
coefficients}\label{sec2}

In this section, we study the elliptic operator $H$ on $L^2(\R^n)$
formally defined by 
\begin{equation}\label{Hdef}
  Hu \,=\, -\div\bigl(A_\infty(x)\nabla u\bigr)\,,
  \qquad u \in L^2(\R^n)\,,
\end{equation}
where the matrix-valued coefficient $A_\infty(x)$ satisfies the
following assumptions\:

\medskip\noindent
1) The $n \times n$ matrix $A_\infty(x)$ is symmetric for all $x \in \R^n$,
and the operator $H$ is uniformly elliptic\\ \null\hskip 14pt
in the sense of \eqref{Aelliptic}; 

\smallskip\noindent
2) The map $A_\infty : \R^n \to \cM_n(\R)$ is homogeneous of degree zero\:
$A_\infty(\lambda x) = A_\infty(x)$ for all $x\in\R^n$\\ \null\hskip 14pt
and all $\lambda > 0$;

\smallskip\noindent
3) The restriction of $A_\infty$ to the unit sphere $\S^{n-1}
\subset \R^n$ is a Lipschitz continuous function. 

\medskip

Elliptic operators of the form \eqref{Hdef} are of course well known,
and were extensively studied in the literature, see for instance
\cite{Da,GT}. For the reader's convenience we recall here a few basic
properties, paying special attention to the homogeneity assumption 2),
which will play an important role in our analysis. As a consequence of
homogeneity, the function $x \mapsto A_\infty(x)$ is necessarily
discontinuous at $x = 0$, unless it is identically constant. Moreover,
in view of 2) and 3), there exists a constant $C > 0$ such that
$\|A_\infty(x)\| \le C$ for all $x \in \R^n$ and
\begin{equation}\label{gradA}
  \|\nabla A_\infty(x)\| \,\le\, \frac{C}{|x|}\,, \qquad
  \hbox{for all } x \in \R^n \setminus \{0\}\,.
\end{equation}
  
\subsection{Definition and domain}\label{ssec21}

To give a rigorous definition of the operator $H$, the easiest way is
to consider the corresponding quadratic form and to use the classical
representation theorem, see e.g. \cite[Section~VI.2]{Ka}. Let 
$\cB$ be the bilinear form on $L^2(\R^n)$ defined by $D(\cB) = 
H^1(\R^n)$ and
\[
  \cB(u_1,u_2) \,=\, \int_{\R^n} \bigl(A_\infty(x)\nabla u_1(x),
  \nabla u_2(x)\bigr)\dd x\,, \qquad u_1, u_2 \in D(\cB)\,.
\]
Under our assumptions on the matrix $A_\infty(x)$, it is easily
verified that the form $\cB$ is symmetric, closed, and nonnegative.
Applying the representation theorem, we thus obtain\:

\begin{prop}\label{prop_domain}
There exists a (unique) nonnegative selfadjoint operator $H : D(H) \to
L^2(\R^n)$ such that $D(H) \subset D(\cB)$ and $\cB(u_1,u_2) = (Hu_1,u_2)$
for all $u_1 \in D(H)$ and all $u_2 \in D(\cB)$. In addition $D(H) =
\{u \in H^1(\R^n)\,|\, \div(A_\infty\nabla u) \in L^2(\R^n)\}$ where
the divergence is understood in the sense of distributions.
\end{prop}

If $H$ has constant coefficients, namely if the matrix $A_\infty$ does
not depend on $x$, it is clear that $D(H) = H^2(\R^n)$. However, this
is not true in the general case, as can be seen from the example of
the Meyers-Serrin matrix \eqref{MSexample} where $D(H)$ contains
functions $u$ that are not $H^2$ in a neighborhood of the origin, see
Section~\ref{ssec32}. As a matter of fact, it does not seem obvious to
determine exactly the domain $D(H)$ under our assumptions on the
diffusion matrix $A_\infty$, but the following (elementary)
observations can nevertheless be made.

\begin{rems} (On the domain of $H$) \\[1mm]
{\bf 1.} Since $A_\infty$ is Lipschitz outside the origin, the elliptic 
regularity theory \cite[Section~8.4]{GT} asserts that $D(H) \subset H^1(\R^n)
\cap H^2(\R^n \setminus B_r)$ for any $r > 0$, where $B_r = \{x \in \R^n
\,|\, |x| \le r\}$. \\[1mm]
{\bf 2.} If $n \ge 3$, then $D(H) \supset H^2(\R^n)$. Indeed, if $u\in
H^2(\R^n)$, we have by Leibniz's rule
\[
  Hu \,=\, -\sum_{i,j=1}^n \Bigl(A_\infty(x)_{ij}\partial^2_{x_ix_j}u
  + \partial_{x_i}(A_\infty(x)_{ij})\partial_{x_j}u \Bigr)\,.
\]
The first term in the right-hand side obviously belongs to $L^2(\R^n)$,
and so does the second one due to estimate \eqref{gradA} and Hardy's
inequality
\begin{equation}\label{HardyL2}
  \Bigl\|\frac {v}{|x|}\Bigr\|_{L^2(\R^n)} \,\le\, \frac{2}{n-2}
  \,\|\nabla v\|_{L^2(\R^n)}\,, \qquad v \in H^1(\R^n)\,, \quad
  n \ge 3\,,
\end{equation}
see e.g. \cite[Section~2.1]{RS}. Thus $Hu \in L^2(\R^n)$, hence $u \in
D(H)$. \\[1mm]
{\bf 3.} If $n \ge 3$ and $A_\infty(x) = \1 + \epsilon B(x)$, where
$B$ is homogeneous of degree zero and Lipschitz continuous on the
sphere $\S^{n-1}$, then $D(H) = H^2(\R^n)$ for all sufficiently
small $\epsilon \in \R$. Indeed, in that case, the argument above shows
that $H$ is a small perturbation of $-\Delta$ in $\cL(H^2(\R^n),L^2(\R^n))$,
the space of bounded linear maps from $H^2(\R^n)$ into $L^2(\R^n)$.
Since ${\bf 1} -\Delta \in \cL(H^2(\R^n),L^2(\R^n))$ is invertible, the
same property remains true for ${\bf 1} + H$ if $\epsilon$ is sufficiently
small, and this implies that $D(H) = H^2(\R^n)$. 
\end{rems}

\subsection{Semigroup and fundamental solution}\label{ssec22}

We next consider the evolution equation $\partial_t u + Hu = 0$, 
namely the linear diffusion equation
\begin{equation}\label{uhom}
  \partial_t u(x,t) \,=\, \div\bigl(A_\infty(x)\nabla u(x,t)\bigr)\,,
  \qquad x \in \R^n\,,\quad t > 0\,,
\end{equation}
which is the analogue of \eqref{veqlimit} in the original variables. 
Since the operator $H$ is selfadjoint and nonnegative, it is well known 
that $-H$ generates an analytic semigroup $e^{-tH}$ in $L^2(\R^n)$ which 
satisfies the contraction property $\|e^{-tH}u\|_{L^2} \le \|u\|_{L^2}$ for 
all $t \ge 0$, see e.g. \cite[Chapter~1]{Pa}. In particular, the Cauchy
problem for equation \eqref{uhom} is well posed for all initial data $u_0 
\in L^2(\R^n)$, the solution being $u(t) = e^{-tH}u_0$ for all $t \ge 0$. 

On the other hand, using the fact that the matrix $A_\infty$ satisfies
the uniform ellipticity condition \eqref{Aelliptic}, one can show that
the semigroup generated by $-H$ is hypercontractive
\cite[Section~2]{Da}, which means that $e^{-Ht}$ is a bounded operator
from $L^2(\R^n)$ to $L^\infty(\R^n)$ for any $t > 0$, and also from
$L^1(\R^n)$ to $L^2(\R^n)$ by duality. By the semigroup property, it
follows that $e^{-Ht}$ is also a bounded operator from $L^1(\R^n)$ to
$L^\infty(\R^n)$, and this implies that there exists a unique integral
kernel $\Gamma(x,y,t)$ such that, for any $u \in L^1(\R^n)$ or
$L^2(\R^n)$, 
\begin{equation}\label{fundsol}
  \bigl(e^{-tH} u\bigr)(x) \,=\, \int_{\R^n} \Gamma(x,y,t) u(y)\dd y\,,
  \qquad x\in \R^n\,, \quad t > 0\,,
\end{equation}
see Remark~\ref{kernelrem} below. The kernel $\Gamma(x,y,t)$ is
usually called the {\em fundamental solution} of the
parabolic equation \eqref{uhom}.

From the pioneering work of De Giorgi \cite{DG} and Nash \cite{Na}, we
know that $\Gamma$ is a H\"older continuous function of its three
arguments, and the strong maximum principle \cite[Section~8.7]{GT}
implies that $\Gamma$ is strictly positive. The following additional
properties will be used later on\:

\medskip\noindent
a) Since $H$ is selfadjoint, we have $\Gamma(x,y,t) = 
\Gamma(y,x,t)$  for all $x,y \in \R^n$ and all $t > 0$. 

\smallskip\noindent
b) For all $x,y \in \R^n$ and all $t > 0$, the following identities
hold
\begin{equation}\label{IntGamma}
  \int_{\R^n}\Gamma(x,y,t)\dd x \,=\, \int_{\R^n}\Gamma(x,y,t)\dd y 
  \,=\, 1\,.
\end{equation}
c) There exists a constant $C > 1$ such that, for all $x,y \in \R^n$ 
and all $t > 0$, 
\begin{equation}\label{Gammabounds}
  \frac{1}{C t^{n/2}}\,e^{-C|x-y|^2/t} \,\le\, \Gamma(x,y,t) \,\le\, 
  \frac{C}{t^{n/2}}\,e^{-|x-y|^2/(Ct)}\,.
\end{equation}
Such Gaussian bounds were first established by Aronson \cite{Ar1,Ar2}, 
see also \cite[Chap.~3]{Da}. 

\smallskip\noindent
d) Since $A_\infty$ is homogeneous of degree zero, we have
\begin{equation}\label{Gammascale}
  \lambda^n\,\Gamma(\lambda x, \lambda y, \lambda^2 t) \,=\, 
  \Gamma(x,y,t)\,, 
\end{equation}
for all $x,y \in \R^n$ and all $t > 0$. 

\begin{rem}\label{kernelrem}
That an integral kernel can be associated to any bounded linear 
operator from $L^p(\Omega)$ to $L^q(\Omega)$ with $q > p$ is 
a ``classical'' result, which is however rather difficult 
to locate precisely in the literature. According to \cite{Za},
this result is due to Dunford in the particular case where 
$\Omega = [0,1]$, and to Buhvalov \cite{Bu} in more general 
situations.
\end{rem}

\subsection{The Green function in dimension $n\ge 3$}\label{ssec23}

We next consider the elliptic equation $Hu = f$, namely
\begin{equation}\label{uell}
  -\div\bigl(A_\infty(x)\nabla u(x)\bigr) \,=\, f(x)\,, 
  \qquad x \in \R^n\,,
\end{equation}
where $f : \R^n \to \R$ is given and $u : \R^n \to \R$ is the
unknown function. If $n \ge 3$ and $f$ is, for instance, a continuous
function with compact support, it is well known that equation
\eqref{uell} has a unique solution $u$ that vanishes at infinity.
In fact, uniqueness is a consequence of the maximum principle
for the uniformly elliptic operator $H$, see \cite[Chapter~3]{GT},
and existence follows from the integral representation
\begin{equation}\label{urep}
  u(x) \,=\, \int_{\R^n} G(x,y) f(y) \dd y\,, \qquad 
  x \in \R^n\,,
\end{equation}
where $G(x,y)$ is the {\em Green function} defined by 
\begin{equation}\label{Greendef}
  G(x,y) \,=\, \int_0^\infty \Gamma(x,y,t)\dd t \,>\, 0\,, \qquad
  \hbox{for all } x,y, \in \R^n\,,\quad x \neq y\,.
\end{equation}
The following elementary properties are direct consequences of the
corresponding assertions for the fundamental solution $\Gamma$\:

\medskip\noindent
a) The Green function $G$ is symmetric\: $G(x,y) = G(y,x)$
for all $x \neq y$.

\smallskip\noindent
b) There exists a constant $C > 1$ such that
\begin{equation}\label{Gbounds}
  \frac{C^{-1}}{|x-y|^{n-2}} \,\le\, G(x,y) \,\le\, 
  \frac{C}{|x-y|^{n-2}}\,, \qquad \hbox{for all } x \neq y\,.
\end{equation}

\smallskip\noindent
c) The Green function is homogeneous of degree $2-n$\: 
$\lambda^{n-2}\,G(\lambda x,\lambda y) = G(x,y)$ for all $x \neq y$
\\ \null\hskip 14pt and all $\lambda > 0$. 

\smallskip\noindent
d) For any $y \in \R^n$ and any test function $v \in C^\infty_c(\R^n)$,
we have
\begin{equation}\label{GDirac}
  \int_{\R^n} \bigl(A_\infty(x)\nabla_x G(x,y)\,,\nabla v(x)\bigr)\dd x
  \,=\, v(y)\,.
\end{equation}

\smallskip
The last property implies that $-\div_x\bigl(A_\infty(x)\nabla_x
G(x,y) \bigr) = \delta(x-y)$ in the sense of distributions, so that
$G(x,y)$ can be considered as the fundamental solution of the elliptic
equation \eqref{uell}. The main statement in this section is the
following proposition, which gives accurate H\"older and gradient
estimates for $G$ under our assumptions on the diffusion matrix $A_\infty$.

\begin{prop}\label{GHolder} Assume that $n \ge 3$, and let $G$ be the 
Green function associated with the elliptic problem \eqref{uell}, where 
the diffusion matrix is symmetric, uniformly elliptic, and homogeneous
of degree zero. There exist constants $C > 0$ and $\beta \in (0,1)$
such that
\begin{equation}\label{Gbd1}
  |G(x_1,y) - G(x_2,y)| \,\le\, C|x_1-x_2|^\beta\biggl(\frac{1}{
  |x_1-y|^{n-2+\beta}} +  \frac{1}{|x_2-y|^{n-2+\beta}}\biggr)\,,
\end{equation}
for all $x_1,x_2,y \in \R^n$ with $x_1 \neq y$ and $x_2 \neq y$. 
Moreover
\begin{equation}\label{Gbd2}
  |\nabla_x G(x,y)| \,\le\, C\biggl(\frac{1}{|x-y|^{n-1}} +  
  \frac{1}{|x|^{1-\beta}|x-y|^{n-2+\beta}}\biggr)\,,
\end{equation}
for all $x,y \in \R^n$ with $x \neq y$ and $x \neq 0$. 
\end{prop}

\begin{proof}
The H\"older estimate \eqref{Gbd1} is explicitly stated in
\cite[Theorem~1.9]{GWi}, but in that classical reference the elliptic
equation \eqref{uell} is considered in a bounded domain
$\Omega \subset \R^n$ with homogeneous Dirichlet conditions at the
boundary $\partial \Omega$.  The more recent work \cite{HK} studies
a class of strongly elliptic systems that includes the scalar equation
\eqref{uell}. In the whole space $\R^n$, the following estimate 
is stated in \cite[Section~3.6]{HK}\: there exist $C > 0$ and 
$0 < \beta < 1$ such that 
\begin{equation}\label{Gbd3}
  |G(x_1,y) - G(x_2,y)| \,\le\, C|x_1-x_2|^\beta |x_1-y|^{2-n-\beta}
  \,, \quad \hbox{if}\quad |x_1 - x_2| < |x_1 - y|/2\,.
\end{equation}
Exchanging the roles of $x_1$ and $x_2$, we deduce
\begin{equation}\label{Gbd4}
  |G(x_1,y) - G(x_2,y)| \,\le\, C|x_1-x_2|^\beta |x_2-y|^{2-n-\beta}
  \,, \quad \hbox{if}\quad |x_1 - x_2| < |x_2 - y|/2\,.
\end{equation}
In the intermediate region where $x_j \neq y$ and $|x_1 - x_2| \ge 
|x_j - y|/2$ for $j = 1,2$, we have by \eqref{Gbounds}
\[
  |G(x_j,y)| \,\le\, C|x_j-y|^{2-n} \,\le\, C|x_1-x_2|^\beta |x_j-y|^{2-n-\beta}\,,
  \qquad j = 1,2\,,
\]
hence 
\begin{equation}\label{Gbd5}
  |G(x_1,y) - G(x_2,y)| \,\le\, G(x_1,y) + G(x_2,y) \,\le\, 
  C \biggl(\frac{|x_1-x_2|^\beta}{|x_1-y|^{n-2+\beta}} + 
  \frac{|x_1-x_2|^\beta}{|x_2-y|^{n-2+\beta}}\biggr)\,.
\end{equation}
Combining \eqref{Gbd3}--\eqref{Gbd5}, we obtain \eqref{Gbd1} in all 
cases. 

We now prove the gradient estimate \eqref{Gbd2}, which takes into 
account the fact that the diffusion matrix in \eqref{uell} is 
homogeneous of degree zero. We use the following auxiliary result. 

\begin{lem}\label{auxlem} {\bf \cite{GWi}}
Assume that $u$ is a bounded solution of the elliptic equation 
$Hu = 0$ in the domain $\Omega = \{x \in \R^n\,|\, |x-x_0| < r\}$, 
where $x_0 \in \R^n$, $x_0 \neq 0$, and $0 < r \le |x_0|/2$. 
Then 
\begin{equation}\label{nablau}
  |\nabla u(x_0)| \,\le\, \frac{C}{r}\,\sup_{x \in \Omega}|u(x)|\,,
\end{equation}
where $C > 0$ depends only on $n$, on $\lambda_1, \lambda_2$ in
\eqref{Aelliptic}, and on the constant in \eqref{gradA}. 
\end{lem}

Estimate \eqref{nablau} follows immediately from Lemma~3.1 in 
\cite{GWi} and its proof, if we use the fact that the matrix 
$A_\infty(x)$ in \eqref{uell} satisfies the Lipschitz estimate
\[
   \|A_\infty(x) - A_\infty(y)\| \,\le\, \frac{C}{|x_0|}\,|x-y|\,,
   \qquad \hbox{for all } x,y \in \Omega\,.
\]
We now come back to the proof of estimate \eqref{Gbd2}. 
Fix $x_0 \in \R^n$, $x_0 \neq 0$, and take $y \in \R^n$, 
$y \neq x_0$. If $|x_0| \le |x_0 - y|/2$, we apply Lemma~\ref{auxlem}
with $r = |x_0|/2$ and $u(x) = G(x,y) - G(x_0,y)$. We know from 
\eqref{Gbd1} that $|u(x)| \le C |x-x_0|^\beta |x_0-y|^{2-n-\beta}$ for 
$x \in \Omega = B(x_0,r)$, and we deduce from \eqref{nablau} that 
\begin{equation}\label{Gbd6}
  |\nabla u(x_0)| \,=\, |\nabla G(x_0,y)| \,\le\,
  \frac{C}{|x_0|^{1-\beta} |x_0 - y|^{n-2+\beta}}\,.
\end{equation}
In the converse case where $|x_0| > |x_0 - y|/2$, we apply 
Lemma~\ref{auxlem} with $r = |x_0-y|/4$ and $u(x) = G(x,y)$.
As $|u(x)| \le C |x-y|^{2-n}$, we deduce from  \eqref{nablau} 
that 
\begin{equation}\label{Gbd7}
  |\nabla u(x_0)| \,=\, |\nabla G(x_0,y)| \,\le\,
  \frac{C}{|x_0 - y|^{n-1}}\,.
\end{equation}
Combining \eqref{Gbd6}, \eqref{Gbd7}, we obtain estimate 
\eqref{Gbd2} in all cases. The proof of Proposition~\ref{GHolder}
is now complete. 
\end{proof}

\subsection{The Green functions in dimension $n=2$}\label{ssec24}

In the two-dimensional case, the integral in \eqref{Greendef} does not
converge anymore, and it is no longer possible to solve the elliptic
problem \eqref{uell} using a positive Green function that decays to
zero at infinity. However, as is shown in the Appendix of \cite{KL},
see also \cite{DK,TKB}, it is still possible to define a Green function
$G(x,y)$ with the following properties\:

\medskip\noindent
i) $G$ is symmetric\: $G(x,y) = G(y,x)$ for all $x,y \in \R^2$
with $x \neq y$.

\smallskip\noindent
ii) $G$ is H\"older continuous for $x \neq y$, and there exists a constant
$C > 0$ such that
\begin{equation}\label{Gdef1}
  |G(x,y)| \,\le\, C\Bigl(1 + \bigl|\log |x-y|\bigr|
  \Bigr)\,, \qquad x \neq y\,.
\end{equation}

\noindent
iii) For any $f \in C^0_c(\R^2)$ such that $\int_{\R^2}f(y)\dd y = 0$,
the unique solution of the elliptic equation\\ \null\hskip 16pt
\eqref{uell} such that $u(x) \to 0$ as $|x| \to \infty$ is given by
\begin{equation}\label{Gdef2}
  u(x) \,=\, \int_{\R^2} G(x,y) f(y)\dd y\,, \qquad x \in \R^2\,.
\end{equation}
iv) Equality \eqref{GDirac} with $n = 2$ holds for all $y \in \R^2$ and all
test functions $v \in C^\infty_c(\R^2)$. 

\smallskip
The Green function with these properties is unique up to an additive
constant. In the particular case where $A_\infty = \1$, we have the
explicit expression $G(x,y) = -(2\pi)^{-1}\log |x-y|$. As is clear
from that example, the Green function is not homogeneous. However,
using the fact that $A_\infty(x)$ is homogeneous of degree zero, it is
easy to verify that, if $G(x,y)$ is a Green function, so is $G(\lambda
x,\lambda y)$ for any $\lambda > 0$. Thus $G(\lambda x,\lambda y) -
G(x,y)$ must be equal to a constant $c(\lambda)$, which depends
continuously only on $\lambda$. As $c(\lambda_1\lambda_2) =
c(\lambda_1) + c(\lambda_2)$ for all $\lambda_1,\lambda_2 > 0$ by
construction, we conclude that there exists a (positive) real number
$c_0$ such that
\begin{equation}\label{Gdef3}
  G(\lambda x,\lambda y) \,=\, G(x,y) \,+\, c_0 \log\frac{1}{\lambda}\,,
\end{equation}  
for all $x \neq y$ and all $\lambda > 0$.

The analogue of Proposition~\ref{GHolder} in the present case is\:

\begin{prop}\label{GHolder2d} Assume that $n = 2$, and let $G$ be a 
Green function associated with the elliptic problem \eqref{uell}, where 
the diffusion matrix is symmetric, uniformly elliptic, and homogeneous
of degree zero. There exist constants $C > 0$ and $\beta \in (0,1)$
such that estimates \eqref{Gbd1}, \eqref{Gbd2} hold with $n = 2$. 
\end{prop}

\begin{proof}
For a class of elliptic systems that include the scalar equation
\eqref{uell}, a Green function in the whole plane $\R^2$ is constructed
in \cite[Section~6]{TKB}, and is shown to satisfy the H\"older estimate
\[
  |G(x_1,y) - G(x_2,y)| \,\le\, C\,\frac{|x_1-x_2|^\beta}{|x_1-y|^\beta}
  \,, \quad \hbox{if}\quad |x_1 - x_2| < |x_1 - y|/2\,,
\]
which is the exact analogue of \eqref{Gbd3} when $n = 2$. Exchanging
the roles $x_1$ and $x_2$, we also have 
\[
  |G(x_1,y) - G(x_2,y)| \,\le\, C\,\frac{|x_1-x_2|^\beta}{|x_2-y|^\beta}
  \,, \quad \hbox{if}\quad |x_1 - x_2| < |x_2 - y|/2\,.
\]
In the intermediate region where $x_j \neq y$ and $|x_1 - x_2| \ge 
|x_j - y|/2$ for $j = 1,2$, we use the fact that the function
$(x_1,x_2,y) \mapsto G(x_1,y) - G(x_2,y)$ is homogeneous of degree zero,
as a consequence of \eqref{Gdef3}. We can thus assume that $|x_1 - x_2| = 1$,
and using \eqref{Gdef1} we easily find
\begin{equation}\label{GHolder2dest}
  |G(x_1,y) - G(x_2,y)| \,\le\, C \biggl(\frac{|x_1-x_2|^\beta}{|x_1-y|^\beta} + 
  \frac{|x_1-x_2|^\beta}{|x_2-y|^\beta}\biggr)\,,
\end{equation}
which completes the proof of \eqref{Gbd1} when $n = 2$. 

To establish the gradient estimate \eqref{Gbd2} for $n = 2$, we use
again Lemma~\ref{auxlem}, which is valid in all space
dimensions. Proceeding as in the proof of Proposition~\ref{GHolder}, we
fix $x_0 \in \R^2$, $x_0 \neq 0$, and take $y \in \R^2$, $y \neq
x_0$. If $|x_0| \le |x_0 - y|/2$, we apply Lemma~\ref{auxlem} with $r
= |x_0|/2$ and $u(x) = G(x,y) - G(x_0,y)$. From \eqref{GHolder2dest} we
know that $|u(x)| \le C |x-x_0|^\beta |x_0-y|^{-\beta}$ for $x \in \Omega =
B(x_0,r)$, and we deduce from \eqref{nablau} that
\[
  |\nabla u(x_0)| \,=\, |\nabla G(x_0,y)| \,\le\,
  \frac{C}{|x_0|^{1-\beta} |x_0 - y|^\beta}\,.
\]
In the converse case where $|x_0| > |x_0 - y|/2$, we apply 
Lemma~\ref{auxlem} with $r = |x_0-y|/4$ and again $u(x) = G(x,y) - G(x_0,y)$.
As $|u(x)| \le C$ by \eqref{GHolder2dest}, we deduce from \eqref{nablau} 
that 
\[
  |\nabla u(x_0)| \,=\, |\nabla G(x_0,y)| \,\le\,
  \frac{C}{|x_0 - y|}\,.
\]
This completes the proof of estimate \eqref{Gbd2} in the
two-dimensional case. 
\end{proof}

\subsection{Weighted estimates for the elliptic equation}\label{ssec25}

The aim of this section is to derive estimates on the integral
operator $K$ formally defined by
\begin{equation}\label{defK}
  K[f](x) \,=\, \int_{\R^n} G(x,y)f(y)\dd y\,, \qquad x \in \R^n\,,
\end{equation}
where $G$ is the Green function introduced in Section~\ref{ssec23}
or \ref{ssec24}. In the two-dimensional case, the Green function
is only defined up to an additive constant, but we always assume
that $f$ is integrable and $\int_{\R^2} f(y)\dd y = 0$, so that
there is no ambiguity in definition \eqref{defK}. 

If $n \ge 3$, we know from \eqref{Gbounds} that $G(x,y)
\le C|x-y|^{2-n}$ for all $x \neq y$. Using the classical
Hardy-Littlewood-Sobolev inequality \cite{LL}, we deduce the
useful estimate
\begin{equation}\label{KHLS}
  \bigl\|K[f]\bigr\|_{L^q(\R^n)} \,\le\, C\|f\|_{L^p(\R^n)}\,, \qquad
  \hbox{if}\quad 1 < p < \frac{n}{2} \quad \hbox{and}\quad  \frac1q \,=\,
  \frac 1p - \frac2n\,.
\end{equation}
However, the bound \eqref{KHLS} is not sufficient for our purposes,
first because the case $n = 2$ is excluded, and also because we need
estimates in the weighted spaces. These improved bounds will be obtained
using the following general result, which concerns integral operators
of the form
\begin{equation}\label{Kopdef}
  \cK[f](x) \,=\, \int_{\R^n} k(x,y) f(y)\dd y\,,
  \qquad x \in \R^n\,,
\end{equation}
where the integral kernel $k(x,y)$ satisfies the following assumptions\:

\medskip\noindent
1) The measurable function $k : \R^n \times \R^n \to \R$ is homogeneous
of degree $-d$, where $d \in (0,n]$\:
\begin{equation}\label{kprop1}
  k(\lambda x,\lambda y) \,=\, \lambda^{-d} \,k(x,y)\,, \qquad x,y \in \R^n\,,
  \quad \lambda > 0\,.
\end{equation}
2) The function $k$ is invariant under simultaneous rotations of both
arguments\:  
\begin{equation}\label{kprop2}
  k(Sx,Sy) \,=\, k(x,y)\,, \qquad x,y \in \R^n\,, \quad S \in SO(n)\,.
\end{equation}
3) There exists $p \in [1,+\infty]$ with $(n{-}d)p \le n$ such that,
for $x \in \S^{n-1} \subset \R^n$, 
\begin{equation}\label{kprop3}
  \kappa_1 \,:=\, \int_{\R^n} |k(x,y)|^{n/d} \,|y|^{-n^2/(dq)}\dd y
  \,<\, \infty\,, \qquad \hbox{where} \quad 1 + \frac1q \,=\, \frac1p +
  \frac{d}{n}\,.
\end{equation}
As a consequence of \eqref{kprop2}, the quantity $\kappa_1$ does not
depend on the choice of $x \in \S^{n-1}$. 

\begin{prop}\label{propkernel}
Assume that the integral kernel $k(x,y)$ satisfies assumptions
\eqref{kprop1}--\eqref{kprop3} above. Then the operator $\cK$ defined by
\eqref{Kopdef} is bounded from $L^p(\R^n)$ to $L^q(\R^n)$ and  
\begin{equation}\label{kprop4}
  \big\|\cK[f] \big\|_{L^q(\R^n)} \,\le\, \kappa_1^{d/n} \,\|f\|_{L^p(\R^n)}\,,
  \qquad \hbox{for all } f \in L^p(\R^n)\,.
\end{equation}
\end{prop}

\begin{rem}\label{history}
Proposition~\ref{propkernel} can be seen as a clever, but relatively
straightforward generalization of the classical Young inequality for
convolution operators. In the particular case where $d = n$, so that
$q = p$, the result is apparently due to L.~G.~Mikhailov,
N.~K.~Karapetiants, and S.~G.~Samko, see \cite[Section~6]{KS} and
\cite{LPSW}. For the reader's convenience, we give a proof of the
general case in Section~\ref{ssec61}.  As is explained in \cite{LPSW},
many classical inequalities, including Hilbert's inequality and
various forms of Hardy's inequality, can be deduced from
Proposition~\ref{propkernel} by an appropriate choice of the integral
kernel $k$. We add to this list the Stein-Weiss inequality \cite{Lie},
which corresponds to the kernel
\[
  k(x,y) \,=\, \frac{1}{|x|^a}\,\frac{1}{|x-y|^\lambda}\,\frac{1}{|y|^b}\,,
  \qquad x \neq y\,,
\]
where $0 < \lambda < n$, $d := a + b + \lambda \in [\lambda,n]$, and
$a < n/q$, $b < n(1-1/p)$ with $p,q$ as in \eqref{kprop3}. As is
easily verified, we can apply Proposition~\ref{propkernel} to that
example under the additional assumption that $a + b > 0$. In
particular the limiting case $a = b = 0$, which corresponds to
the classical HLS inequality, cannot be obtained in this way. 
\end{rem}

As a first application of Proposition~\ref{propkernel}, we establish
the following estimate for the linear operator \eqref{defK} in the
weighted spaces $L^2(m)$ defined in \eqref{L2mdef}.

\begin{prop}\label{Kernelprop0}
If $n\ge 3$ and if $m \ge 0$ satisfies $2 - n/2 < m < n/2$, the
operator $K$ defined by \eqref{defK} is bounded from $L^2(m)$ to
$L^2(m{-}2)$. Specifically, if $f \in L^2(m)$ and $u = K[f]$,
we have the homogeneous estimate
\begin{equation}\label{uweight0}
  \int_{\R^n} |x|^{2m-4}\,|u(x)|^2\dd x \,\le\, C
  \int_{\R^n}|x|^{2m}\,|f(x)|^2\dd x \,<\, \infty\,,
\end{equation}
for some constant $C > 0$ independent of $f$. 
\end{prop}

\begin{proof}
If $f \in L^2(m)$ and $u = K[f]$ we have, in view of \eqref{defK} and
\eqref{Gbounds}, 
\begin{equation}\label{ueq0}
  |x|^{m-2} \,|u(x)| \,\le\, C \int_{\R^n} k(x,y)\, |y|^m |f(y)| \dd y\,,
  \quad \hbox{where}\quad  k(x,y) \,=\, \frac{|x|^{m-2}}{|x{-}y|^{n-2}|y|^m}\,.
\end{equation}
The integral kernel $k(x,y)$ in \eqref{ueq0} is homogeneous of degree
$-n$ and invariant under rotations, in the sense of \eqref{kprop2}.
Moreover, for any $x \in \S^{n-1} \subset \R^n$, we have
\begin{equation}\label{newkappa1}
  \kappa_1 \,=\, \int_{\R^n} k(x,y) \,|y|^{-n/2}\dd y \,<\, \infty\,.
\end{equation}
Indeed, the integral in \eqref{newkappa1} converges near the origin
because $m + n/2 < n$, and near infinity because $n-2 + m + n/2 >
n$. Moreover, the singularity at $y = x$ is always integrable. So,
applying Proposition~\ref{propkernel} with $d = n$ and $p = q = 2$, we
obtain the estimate \eqref{uweight0}. If $m \le 2$, this immediately
implies that $K$ is bounded from $L^2(m)$ to $L^2(m{-}2)$.  If $m >
2$, which is only possible when $n \ge 5$, it remains to bound the
$L^2$ norm of $u$ on the unit ball $B = B(0,1) \subset \R^n$, which is
not controlled by \eqref{uweight0} since $2m-4 > 0$. This is easily
done using the HLS inequality \eqref{KHLS}, which shows that
$\|u\|_{L^2(B)} \le C \|u\|_{L^{2n/(n-4)}(\R^n)} \le C\|f\|_{L^2(\R^n)}$.
\end{proof}

\begin{rem}\label{rightinverse}  
By a similar argument, using estimate \eqref{Gbd2}, one can
show that the function $u = K[f]$ in Proposition~\ref{Kernelprop0}
satisfies $\nabla u \in L^2(m{-}1)$ and $\nabla u(x) = \int \nabla_x
G(x,y)f(y)\dd y$. Thus, if multiply equality \eqref{GDirac} by
$f(y)$ and integrate over $y \in \R^n$, we obtain the relation
$\int \bigl(A_\infty(x)\nabla u(x),\nabla v(x) \bigr)\dd x = \int
v(x) f(x)\dd x$, which is valid for all $v \in C^\infty_c(\R^n)$.
This implies that $-\div(A_\infty\nabla u) = f$ in the sense of distributions
on $\R^n$, namely $HK[f] = f$ where $H$ is defined in \eqref{Hdef}. 
\end{rem}

The assumption that $m < n/2$ is essential in
Proposition~\ref{Kernelprop0}, even in the particular case where
$A_\infty = \1$. As we now show, it is possible to establish estimate
\eqref{uweight0} for larger values of $m$, if we assume that the
function $f \in L^2(m)$ has {\em zero mean}. At this point, we recall
that $L^2(m) \hookrightarrow L^1(\R^n)$ precisely when $m > n/2$. For
technical reasons that will become clear in the proof of
Theorem~\ref{main2}, we formulate our next result in the more
general framework of weighted $L^p$ spaces, with $p \in [1,2]$.
Those spaces are defined in close analogy with \eqref{L2mdef}\:
\begin{equation}\label{Lpmdef}
  L^p(m) \,=\, \Bigl\{f \in L^p_\loc(\R^n)\,\Big|\, \|f\|_{L^p(m)} < \infty
  \Bigr\}\,, \quad \|f\|_{L^p(m)}^p \,=\, \int_{\R^n} (1+|y|)^{mp}
  |v(y)|^p\dd y\,.
\end{equation}
If $m > n(1-\frac1p)$, we have $L^p(m) \hookrightarrow L^1(\R^n)$ by
H\"older's inequality, and in that case we denote by $L^p_0(m)$ the
closed subspace of $L^p(m)$ defined by
\begin{equation}\label{L20mdef}
  L^p_0(m) \,=\, \Bigl\{f \in L^p(m)\,\Big|\, \int_{\R^n} f(x)\dd x = 0
 \Bigr\}\,, \qquad m \,>\, n\bigl(1 - {\TS \frac1p}\bigr)\,.
\end{equation}

\begin{prop}\label{Kernelprop1}
Let $n\ge 2$ and let $\beta \in (0,1)$ be as in \eqref{Gbd1}. For any
$m \in (n/2,n/2+\beta)$ and any $p \in [1,2]$ such that $p > 2n/(n{+}4)$,
the operator $K$ defined by \eqref{defK} is bounded from $L^p_0(m{-}s)$ to
$L^2(m{-}2)$, where $s = n/p - n/2$. Specifically, if $f \in L^p_0(m{-}s)$
and $u = K[f]$, we have the homogeneous estimate
\begin{equation}\label{uweight}
  \int_{\R^n} |x|^{2m-4}\,|u(x)|^2\dd x \,\le\, C \left(\int_{\R^n}|x|^{p(m-s)}
  \,|f(x)|^p\dd x\right)^{2/p} \,<\, \infty\,,
\end{equation}
for some constant $C > 0$ independent of $f$. 
\end{prop}

\begin{rem}\label{rem_p}
If $p = 2$, so that $s = 0$, estimate \eqref{uweight} reduces to
\eqref{uweight0}, and Proposition~\ref{Kernelprop1} thus shows that $K$ is
bounded from $L^2_0(m)$ to $L^2(m{-}2)$ if $n/2 < m < n/2 + \beta$. We
believe that the upper bound on $m$ is sharp.  In the particular case
were $A_\infty = \1$, so that $\beta = 1$, estimate \eqref{uweight} is
not valid for $m > n/2 + 1$ unless one assumes that not only the
integral but also the first order moments of $f$ vanish. In the proof
of Theorem~\ref{main2} below, where $n = 2$ or $3$,
Proposition~\ref{Kernelprop1} will also be used with $p = 1$ and
$s = n/2$. 
\end{rem}
  
\begin{rem}\label{rem_mrange}
If $n = 2$, or if $n = 3$ and $\beta \le 1/2$, we necessarily have $m <
2$ in Proposition~\ref{Kernelprop1}, so that $2m-4 < 0$. In that case,
if $f$ satisfies the assumptions of Proposition~\ref{Kernelprop1}, the
solution $u$ of the elliptic equation \eqref{uell} may not belong to
$L^2(\R^n)$, because $u(x)$ decays too slowly as $|x| \to
+\infty$. Explicit examples of this phenomenon can be constructed
using the Meyers-Serrin matrix \eqref{MSexample}, see Section~\ref{ssec63}.
\end{rem}
  
\begin{proof}
Our assumptions on the parameters $m$ and $p$ obviously imply that $s
\in [0,n/2]$, $s < 2$, and $m-s > n(1-\frac1p)$, so that $L^p(m{-}s)
\hookrightarrow L^1(\R^n)$. If $f \in L^p_0(m{-}s)$ and $u = K[f]$, we
thus have the representation formula
\[
  u(x) \,=\, \int_{\R^n} \Bigl(G(x,y)-G(x,0)\Bigr)f(y)\dd y\,,
  \qquad x \in \R^n\,,
\]
which is equivalent to \eqref{defK} since $\int_{\R^n} f(x)\dd x = 0$.
We recall that the above integral uniquely defines $u$ even if $n = 2$
because $G$ is unique up to a constant in that case. We also note
that, in any dimension $n\ge 2$, the difference $G(x,y)-G(x,0)$ is
homogeneous of degree $2-n$, see Sections \ref{ssec23} and
\ref{ssec24}.  The general idea is to bound that difference using
estimate \eqref{Gbd1} when $|y|$ is small compared to $|x|$, and
estimate \eqref{Gbounds} or \eqref{Gdef1} when $|y| \ge |x|/2$. We
thus introduce a smooth cut-off function $\chi : \R_+ \to [0,1]$
satisfying $\chi(r) = 1$ when $r \in [0,1/2]$ and $\chi(r) = 0$ when
$r \ge 3/4$. We observe that $|u(x)| \le u_1(x) + u_2(x)$ where
\begin{align*}
  u_1(x) \,&=\, \int_{\R^n} \Bigl|G(x,y)-G(x,0)\Bigr|\,\chi\Bigl({\TS
  \frac{|y|}{|x|}}\Bigr)|f(y)| \dd y\,, \\
  u_2(x) \,&=\, \int_{\R^n} \Bigl|G(x,y)-G(x,0)\Bigr|\,
  \Bigl(1 - \chi\Bigl({\TS \frac{|y|}{|x|}}\Bigr)\Bigr)|f(y)| \dd y\,.
\end{align*}
We shall prove that, for $j = 1,2$, the following estimate holds\:
\begin{equation}\label{ujeq}
  |x|^{m-2} \,u_j(x) \,\le\, C \int_{\R^n} k_j(x,y)\, |y|^{m-s} |f(y)| \dd y\,,
\end{equation}
where $k_j(x,y)$ is an integral kernel which fulfills the assumptions of
Proposition~\ref{propkernel} with $d = n - s$ and $p = 2n/(n{+}2s)$.
This will imply that both $u_1$ and $u_2$ satisfy estimate \eqref{uweight}
with $q = 2$, which gives the desired conclusion. 

We start with $u_1$. Using \eqref{Gbd1} to bound the difference
$G(x,y)-G(x,0) \equiv G(y,x)-G(0,x)$, we obtain estimate \eqref{ujeq}
for $j = 1$ where
\[
  k_1(x,y) \,=\, \frac{|x|^{m-2}}{|y|^{m-s}}\,\Bigl(\frac{|y|^\beta}{|x-y|^{n-2+\beta}}
  + \frac{|y|^\beta}{|x|^{n-2+\beta}}\Bigr)\chi\Bigl(\frac{|y|}{|x|}\Bigr)\,.
\]
The kernel $k_1(x,y)$ is obviously homogeneous of degree $-d = s-n$ and invariant
under rotations. Moreover, if $|x| = 1$, we have $\chi(|y|/|x|) =
\chi(|y|) = 0$ when $|y| \ge 3/4$, so that condition \eqref{kprop3} becomes
\[
  \int_{\R^n} k_1(x,y)^{n/d}\,|y|^{-n^2/(2d)} \dd y \,\equiv\, \int_{|y| \le 3/4}
  \Bigl(k_1(x,y)\,|y|^{-n/2}\Bigr)^{n/d} \dd y \,<\, \infty\,, \quad
  \hbox{when } |x|=1\,.
\]
The only singularity of the integrand is at the origin where $k_1(x,y)\,|y|^{-n/2}
\sim |y|^{\beta+s-m-n/2}$, and the assumption that $m < n/2 + \beta$ ensures that
$(n/d)\bigl(m + n/2 - \beta - s\bigr) < n$. So we can apply
Proposition~\ref{propkernel} and conclude that the function $u_1$
satisfies estimate \eqref{uweight} with $q = 2$. 

To estimate $u_2$ if $n\ge 3$, we use \eqref{Gbounds} and we obtain 
estimate \eqref{ujeq} for $j = 2$, where
\[
  k_2(x,y) \,=\,
  \frac{|x|^{m-2}}{|y|^{m-s}}\,\Bigl(\frac{1}{|x-y|^{n-2}}+\frac{1}{|x|^{n-2}}
  \Bigr)\,\Bigl(1 - \chi\Bigl({\TS \frac{|y|}{|x|}}\Bigr)\Bigr)\,,
  \qquad n \ge 3\,.
\]
If $n = 2$, the difference $G(x,y) - G(x,0)$ is homogeneous of degree
zero, and it follows that $G(x,y)-G(x,0)=G(x/|x|,y/|x|)-G(x/|x|,0)$. Using
\eqref{Gdef1}, we thus obtain estimate \eqref{ujeq} for $j = 2$, where
\[
  k_2(x,y) \,=\, \frac{|x|^{m-2}}{|y|^{m-s}}\,\biggl(1+\Bigl|\log \frac{|x-y|}{|x|}
  \Bigr|\biggr) \,\Bigl(1 - \chi\Bigl({\TS \frac{|y|}{|x|}}\Bigr)\Bigr)\,,
  \qquad n = 2\,.
\]
In any case, the kernel $k_2(x,y)$ is homogeneous of degree $-d = s-n$,
invariant under rotations, and if $|x| = 1$ we have
\[
  \int_{\R^n} k_2(x,y)^{n/d}\,|y|^{-n^2/(2d)} \dd y \,\equiv\,
  \int_{|y| \ge 1/2} \Bigl(k_2(x,y)\,|y|^{-n/2}\Bigr)^{n/d} \dd y \,<\, \infty\,.
\]
Indeed, the singularity at $y = x$ is integrable provided $(n/d)(n-2)
< n$, which is the case because we assumed that $s < 2$, and the
convergence of the integral at infinity is guaranteed since $m > n/2$.
Applying Proposition~\ref{propkernel} again, we conclude that $u_2$
also satisfies estimate \eqref{uweight} with $q = 2$. This completes
the proof of \eqref{uweight}.

It is now easy to conclude the proof of
Proposition~\ref{Kernelprop1}.  If $m \le 2$, estimate
\eqref{uweight} implies of course that $u \in L^2(m{-}2)$ and
$\|u\|_{L^2(m-2)} \le C\|f\|_{L^p(m-s)}$. If $m > 2$, which is
possible only when $n \ge 3$, it remains to bound the $L^2$ norm of
$u$ on the unit ball $B = B(0,1) \subset \R^n$. If $p > 1$, which is
automatic when $n \ge 4$, this follows from the HLS inequality
\eqref{KHLS}, which implies that $\|u\|_{L^q(\R^n)} \le
C\|f\|_{L^p(\R^n)}$ for $q = np/(n{-}2p) > 2$.  In the particular case
where $p = 1$ and $n = 3$, we can obtain the bound $\|u\|_{L^q(B)} \le
C\|f\|_{L^1(\R^n)}$ for all $q < 3$ using definition \eqref{defK},
estimate \eqref{Gbounds}, and H\"older's inequality. 
\end{proof}

We also need to estimate the function $u = K[f]$ in the particular case
where $f = \div g$ for some vector field $g : \R^n \to \R^n$. In that situation,
if we integrate by parts formally in \eqref{defK}, we obtain the relation
$u = (K\circ\div)[g]$, where the new operator $K\circ\div$ is defined by
\begin{equation}\label{Kcircdef}
  \bigl(K \circ \div\bigr)[g](x) \,=\,  -\int_{\R^n} \nabla_y G(x,y)\cdot
  g(y) \dd y\,, \qquad x \in \R^n\,.
\end{equation}
We first prove that this operator is well defined on $L^2(m{-}1)$ if $m > 2 - n/2$
and $m \ge 1$, and we next give conditions on $g$ that ensure that
$(K \circ \div)[g] = K[\div g]$. 

\begin{prop}\label{Kernelprop2}
Let $n\ge 2$ and let $\beta \in (0,1)$ be as in \eqref{Gbd1}. For
any $m \in (2-n/2,n/2+\beta)$ such that $m \ge 1$, the operator
$K\circ\div$ defined by \eqref{Kcircdef} is bounded from $L^2(m{-}1)^n$ to
$L^2(m{-}2)$. Specifically, if $g \in L^2(m{-}1)^n$ and $u = (K\circ\div)[g]$,
we have the homogeneous estimate
\begin{equation}\label{uweight2}
  \int_{\R^n} |x|^{2m-4}\,|u(x)|^2\dd x \,\le\, C \int_{\R^n}|x|^{2m-2}
  \,|g(x)|^2\dd x \,<\, \infty\,,
\end{equation}
for some constant $C > 0$ independent of $g$. 
\end{prop}

\begin{proof}
Let $g \in L^2(m{-}1)^n$ and $u = (K\circ\div)[g]$. We estimate the integral
kernel $\nabla_y G(x,y)$ in \eqref{Kcircdef} using the bound \eqref{Gbd2} and
keeping in mind that $\nabla_y G(x,y) = \nabla_z G(z,x)\big|_{z=y}$ by symmetry.
This gives
\begin{equation}\label{uueq}
  |x|^{m-2} \,|u(x)| \,\le\, C \int_{\R^n} k(x,y)\, |y|^{m-1} |g(y)| \dd y\,,
\end{equation}
where
\[
  k(x,y) \,=\, \frac{|x|^{m-2}}{|y|^{m-1}}\,\Bigl(\frac{1}{|x-y|^{n-1}}
  + \frac{1}{|y|^{1-\beta}|x-y|^{n-2+\beta}}\Bigr)\,.
\]
The kernel $k$ is homogeneous of degree $-n$ and invariant under rotations.
To apply Proposition~\ref{propkernel} with $p = q = 2$, we need to verify
that, for any $x \in \S^{n-1} \subset \R^n$, 
\[
  \int_{\R^n} k(x,y)\,|y|^{-n/2} \dd y \,<\, \infty\,.
\]
The integral converges for small $|y|$ if and only if $m - \beta + n/2 < n$,
namely $m < n/2 + \beta$. At infinity, the integrability condition is
$m + n -2 + n/2 > n$, namely $m > 2 - n/2$. Thus, applying
Proposition~\ref{propkernel}, we deduce \eqref{uweight2} from
\eqref{uueq}. 

To show that $u \in L^2(m{-}2)$, it remains to control the $L^2$ norm
of $u$ when $m > 2$. In that case, we simply observe that $2 \in
(2-n/2,n/2+\beta)$, and applying the argument above (with $m = 2$) we
obtain the bound $\|u\|_{L^2(\R^n)} \le C \|g\|_{L^2(1)} \le C
\|g\|_{L^2(m-1)}$.  This concludes the proof.
\end{proof}

\begin{cor}\label{2Kcor}
If $g \in L^2(m{-}1)^n$ for some $m > n/2$ and if
$f = \div g \in L^2(m)$, then $f \in L^2_0(m)$ and $K[f] =
(K\circ\div)[g]$.
\end{cor}

\begin{proof}
As $m > n/2$, we have $L^2(m{-}1) \hookrightarrow L^p(\R^n)$ for some
$p < n/(n{-}1)$, by H\"older's inequality. Thus, applying
Lemma~\ref{intlem} below, we see that $\int_{\R^n}f \dd x  = 0$
if $f$ is as in the statement. To show that $K[f] = (K\circ\div)[g]$,
we have to justify the integration by parts leading to \eqref{Kcircdef}. As
in Section~\ref{ssec62}, we denote $\chi_k(x) = \chi(x/k)$, where
$\chi : \R^n \to [0,1]$ is a smooth cut-off function satisfying
$\chi(x) = 1$ for $|x| \le 1$ and $\chi(x) = 0$ for $|x| \ge 2$. 
We start from the identity
\[
  \int_{\R^n} \chi_k(y)\Bigl(G(x,y) \div g(y) + \nabla_y G(x,y)\cdot
  g(y)\Bigr) \dd y \,=\, -\int_{\R^n} G(x,y)\,g(y)\cdot\nabla\chi_k(y)
  \dd y\,,
\]
which holds for all $k \in \N^*$ and almost all $x \in \R^n$. If $m \in
(n/2,n/2+\beta)$, the left hand-side has a limit in $L^2(m{-}2)$ as $k
\to +\infty$, in view of Propositions~\ref{Kernelprop1} and \ref{Kernelprop2}.
To prove the desired result, it is thus sufficient to show that the
right-hand side converges to zero in the sense of distributions. 
Integrating against a test function $\psi \in C^\infty_c(\R^n)$ and denoting
$\Psi(y) = \int_{\R^n}G(x,y) \psi(x)\dd x$, we have to show that
\[
  \lim_{k \to +\infty} \int_{\R^n} \Psi(y)\,g(y)\cdot\nabla\chi_k(y)
  \dd y \,\equiv\, \lim_{k \to +\infty} \frac1k \int_{k \le |y| \le 2k}
  \Psi(y) \,g(y)\cdot\nabla\chi(y/k) \dd y \,=\, 0\,.
\]
This in turn is an easy consequence of H\"older's inequality, if
we use the facts that $g \in L^2(m{-}1)$ for some $m > 2-n/2$, and
$|\Psi(y)| \le C(1{+}|y|)^{2-n}$ if $n \ge 3$ or $|\Psi(y)|
\le C\log(2{+}|y|)$ if $n = 2$.   
\end{proof}

\begin{rem}\label{Krem}
As a final comment, we mention that, if $f \in L^2_0(m)$ for some $m \in
(n/2,n/2+1)$, there exists $g \in L^2(m{-}1)^n$ such that $\div g = f$,
see Lemma~\ref{divlem}. Thus $K[f] = (K\circ\div)[g]$ by
Corollary~\ref{2Kcor}, and estimate \eqref{uweight0} can be deduced
from estimate \eqref{uweight2} if $m < n/2 + \beta$. 
\end{rem}

%%%%%%%%%%%%%%%%%%%%%%%%%%%%%%%%%%%%%%%%%%%%%%%%%%%%%%%%%%%%%%%%%%%%%

\section{The diffusion operator in self-similar 
variables}\label{sec3}

In this section we study the generator $L$ of the evolution equation 
\eqref{veqlimit}, considered as an operator in the weighted space
$L^2(m) \subset L^2(\R^n)$ for some $m \ge 0$. This operator is defined
by
\begin{equation}\label{Ldef}
  L u \,=\, \div(A_\infty(x) \nabla u) + \frac{1}{2}\,x \cdot \nabla u
  + \frac{n}{2}\,u\,, \qquad u \in D(L)\,,
\end{equation}
where $D(L) \subset L^2(m)$ is the maximal domain
\[
  D(L) \,=\, \big\{u \in L^2(m) \cap H^1(\R^n)\,\big|\,
  \div(A_\infty(x) \nabla u) + {\TS \frac12} x \cdot \nabla u
  \in L^2(m)\big\}\,.
\]

\subsection{The constant coefficient case}\label{ssec31}
  
In the particular case where $A_\infty = \1$, the operator $L$ is
studied in detail in \cite[Appendix~A]{GWa}. It is shown there that
the spectrum of $L$ in $L^2(m)$ consists of two different parts\:

\smallskip\noindent
a) a countable sequence of discrete eigenvalues\: $\sigma_\disc =
\{-k/2\,|\, k = 0,1,2,\dots\}$;

\smallskip\noindent
b) a half-plane of essential spectrum\: $\sigma_\ess = \{z \in
\C\,|\, \Re(z) \le \frac{n}{4} - \frac{m}{2}\}$.

\smallskip\noindent
The spectrum $\sigma = \sigma_\disc \cup \sigma_\ess$ is represented
in Figure~\ref{fig1} for a typical choice of the parameters $n,m$.
It is worth noting that the discrete spectrum $\sigma_\disc$
does not depend on $m$. In fact, conjugating the operator $L$
with the Gaussian weight $e^{-|x|^2/8}$, we obtain the useful
relation
\begin{equation}\label{Lconj1}
  e^{|x|^2/8}\,L\, e^{-|x|^2/8} \,=\, \Delta - \frac{|x|^2}{16} +
  \frac{n}{4}\,,
\end{equation}
where the right-hand side is the harmonic operator in $\R^n$,
normalized so that its spectrum in $L^2(\R^n)$ is precisely the
sequence $\sigma_\disc$. This shows that the eigenfunctions of $L$
associated with the discrete spectrum $\sigma_\disc$ have Gaussian
decay at infinity, hence belong to $L^2(m)$ for any $m \ge
0$. Moreover we have $L \phi = 0$, where
\begin{equation}\label{phi1}
  \phi(x) \,=\, \frac{1}{(4\pi)^{n/2}}\, e^{-|x|^2/4}\,,
  \qquad x \in \R^n\,,
\end{equation}
and differentiating $k$ times the principal eigenfunction $\phi$
we obtain the $k$-th order Hermite functions that span the
kernel of $L + k/2$ if $m$ is sufficiently large, namely 
$m > k + n/2$.

\begin{figure}[ht]
  \begin{center}
  \begin{picture}(200,130)% width and height of the picture
  \put(0,10){\includegraphics[width=1.00\textwidth]{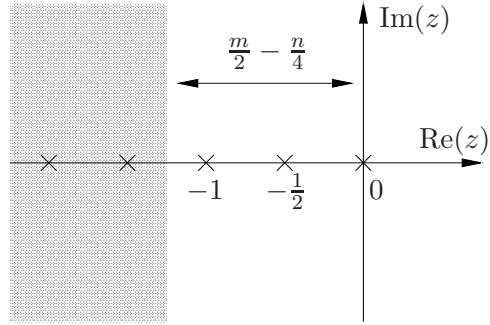}}
  \put(139,50){$0$}
  \put(100,50){$\textstyle -\frac{1}{2}$}
  \put(70,50){$-1$}
  \put(85,103){$\textstyle \frac{m}{2}-\frac{n}{4}$}
  \put(158,69){$\Re(z)$}
  \put(143,115){$\Im(z)$}
  \end{picture}
  \caption{When $A_\infty = \1$ the spectrum of the operator $L$ in the space
  $L^2(m)$ consists of a sequence of eigenvalues $0, -1/2, -1, \dots $ and
  of essential spectrum filling the half-space $\{z \in \C\,|\,
  \Re(z) \le \frac{n}{4} - \frac{m}{2}\}$. For any $k \in \N$, 
  the eigenvalue $-k/2$ is isolated if $m > k+n/2$.\label{fig1}}
  \end{center}
\end{figure}

On the other hand, the essential spectrum $\sigma_\ess$ has a
completely different origin, which is revealed by applying the
Fourier transform so that $L$ becomes a first-order differential
operator acting on the Sobolev space $H^m(\R^n)$, see
\cite[Appendix~A]{GWa}. Using this observation, on can show that each
complex point $z \notin \sigma_\disc$ is an eigenvalue of $L$ of
infinite multiplicity (if $n \ge 2$), with eigenfunctions that decay
slowly, like $|x|^{2\Re(z)-n}$, as $|x| \to \infty$. In particular,
these eigenvalues belong to $L^2(m)$ if and only if $\Re(z) <
\frac{n}{4} - \frac{m}{2}$, which explains why the essential spectrum
$\sigma_\ess$, unlike $\sigma_\disc$, is sensitive to the value of
$m$.

To summarize, in the case where $A_\infty = \1$ the operator $L$ has
$k+1$ isolated eigenvalues if the parameter $m$ is large enough so
that $m > k + n/2$, see Figure~1. In particular, if $m > n/2$, the
zero eigenvalue is simple and isolated, and the rest of the spectrum is
contained in the half-plane $\{z \in \C\,|\, \Re(z) \le -\mu\}$,
where $\mu = \min(1/2,m/2 - n/4)$. Note that the assumption
$m > n/2$ ensures that $L^2(m) \hookrightarrow L^1(\R^n)$.

\subsection{A nontrivial example\: the Meyers-Serrin operator}\label{ssec32}

We next study in detail the instructive example where the limiting
matrix $A_\infty$ is given by \eqref{MSexample}. It turns out that,
in that case too, the eigenvalues and eigenfunctions of the linear
operator \eqref{Ldef} can be computed explicitly, and exhibit a
nontrivial behavior when the parameter $b > 0$ is varied. In
what follows we denote
\begin{equation}\label{Abdef}
  A_b(x) \,=\, b\,\1 + (1-b) \frac{x \otimes x}{|x|^2}\,,
  \qquad x \in \R^n\setminus\{0\}\,,
\end{equation}
where $\1$ is the identity matrix and $(x\otimes x)_{ij} = x_ix_j$.
Elliptic equations with a diffusion matrix of the form \eqref{Abdef}
were considered by Meyers and Serrin nearly sixty years ago.  If the
parameter $b > 0$ is small enough, they turn out to be useful to illustrate
the optimality of general results concerning the interior regularity of
solutions \cite[Section~5]{Me} or the local uniqueness \cite{Se,BB}.

As is clear from definition \eqref{Abdef}, we have $A_b(x)x = x$ and 
$A_b(x)y = by$ for any $y \in \R^n$ that is orthogonal to $x$. If 
$b \neq 1$, the eigenvalues of $A_b(x)$ are thus $1$ (multiplicity $1$) 
and $b$ (multiplicity $n-1$). For the evolution equation $\partial_t u = 
\div(A_b(x)\nabla u)$, this means that diffusion in the radial direction 
is unaffected by the value of $b$, whereas the diffusion rate 
is increased ($b > 1$) or decreased ($b < 1$) in the transverse 
directions. 

We now consider the rescaled diffusion operator $L_b$ defined 
by 
\begin{equation}\label{Lb1}
  L_b u \,=\, \div(A_b \nabla u) + \frac{1}{2}x \cdot \nabla u
  + \frac{n}{2}u\,, \qquad x \in \R^n\,.
\end{equation}
Since
\[
  \div\Bigl(\frac{x \otimes x}{|x|^2}\,\nabla u\Bigr) \,=\, 
  \div\Bigl(\frac{x}{|x|^2}\,x\cdot \nabla u\Bigr) \,=\, 
  \frac{1}{|x|^2}\Bigl((x\cdot\nabla)^2 u + (n{-}2)\,
  x\cdot\nabla u \Bigr)\,,
\]
we obtain the alternative form
\begin{equation}\label{Lb2}
  L_b u \,=\, b \Delta u +  \frac{1{-}b}{|x|^2}\Bigl((x\cdot\nabla)^2 u 
  + (n{-}2)\,x\cdot\nabla u \Bigr) + \frac{1}{2}x \cdot \nabla u
  + \frac{n}{2}u\,.
\end{equation}
As is clear from \eqref{Lb2}, the operator $L_b$ is invariant under
rotations around the origin, and this makes it possible to compute
its eigenvalues and eigenvectors by the classical method of 
``separation of variables''. 

Indeed, let $p : \R^n \to \R$ be a {\em harmonic polynomial} 
that is {\em homogeneous of degree} $\ell \in \N$. We look for 
eigenfunctions of $L_b$ of the form 
\begin{equation}\label{Ansatz1}
  u(x) \,=\, p(x)\,\phi(|x|)\,, \qquad x \in \R^n\,,
\end{equation}
where $\phi : \R_+ \to \R$. As $\Delta p = 0$ and $x \cdot \nabla p = 
\ell p$, we easily find
\[
  \Delta u(x) \,=\, p(x)\Bigl(\phi''(r) + \frac{n-1+2\ell}{r}\,\phi'(r)
  \Bigr)\,, \qquad \hbox{where } r = |x|\,.
\]
Similarly
\[
  x\cdot\nabla u = p\bigl(r\phi' + \ell \phi\bigr)\,, \qquad  
  (x\cdot\nabla)^2 u = p\bigl(r^2\phi'' + (2\ell+1)r \phi' + 
  \ell^2 \phi\bigr)\,,
\]
hence   
\[
  \div\Bigl(\frac{x}{|x|^2}\,x\cdot \nabla u\Bigr) \,=\, 
  p\Bigl(\phi'' + \frac{n-1+2\ell}{r}\,\phi' + \frac{\ell(n-2+\ell)}{r^2}
  \,\phi\Bigr)\,.
\]
It follows that $(L_bu)(x) = p(x) (L_{b,\ell}\,\phi)(|x|)$, where
\begin{equation}\label{Lbldef}
  L_{b,\ell}\,\phi \,=\, \phi'' + \frac{n-1+2\ell}{r}\,\phi' + (1-b)
  \frac{\ell(n-2+\ell)}{r^2}\,\phi + \frac{r}{2}\,\phi' 
  + \frac{n+\ell}{2}\,\phi\,. 
\end{equation} 

In a second step, we look for eigenfunctions of the radial operator 
$L_{b,\ell}$ of the following form
\begin{equation}\label{Ansatz2}
  \phi(r) \,=\, r^{\gamma} e^{-r^2/4} \psi(r^2/4)\,, \qquad r > 0\,,
\end{equation}
where $\gamma \in \R$ is a parameter that will be determined below. 
A direct computation shows that
\begin{align*}
  \phi'(r) \,&=\, r^\gamma e^{-r^2/4} \biggl(\frac{r}{2}\psi'\Bigl(\frac{r^2}{4}
  \Bigr) + \Bigl(\frac{\gamma}{r} - \frac{r}{2}\Bigr)\psi\Bigl(\frac{r^2}{4}
  \Bigr)\biggr)\,, \\
  \phi''(r) \,&=\,  r^\gamma e^{-r^2/4} \biggl(\frac{r^2}{4}\psi'' + 
  \Bigl(\gamma + \frac12 - \frac{r^2}{2}\Bigr)\psi' + 
  \Bigl(\frac{\gamma^2{-}\gamma}{r^2} - \gamma - \frac12 + \frac{r^2}{4}
  \Bigr)\psi\biggr)\,, 
\end{align*}
and it follows that $(L_{b,\ell}\,\phi)(r) = r^\gamma e^{-r^2/4} (L_{b,\ell,\gamma} 
\psi)(r^2/4)$, where the differential operator $L_{b,\ell,\gamma}$ acts 
on the variable $y = r^2/4 \in \R_+$ and is defined in the following 
way. Setting
\begin{equation}\label{alphadef}
  \alpha \,=\, \frac{n}{2} -1 + \gamma + \ell\,, \qquad
  \delta \,=\, \gamma^2 + \gamma(n-2+2\ell) + (1-b)\ell(n-2+\ell)\,,
\end{equation}
we have the explicit expression
\begin{equation}\label{Lblgdef}
  \bigl(L_{b,\ell,\gamma}\,\psi\bigr)(y) \,=\, y\psi''(y) + (\alpha + 1 - y)
  \psi'(y) + \Bigl(\frac{\delta}{4y} - \frac{\gamma+\ell}{2}\Bigr)\psi(y)\,, 
  \qquad y > 0\,.
\end{equation} 
To find eigenfunctions, it is necessary to choose the parameter $\gamma$ 
in such a way that the quantity $\delta$ defined in \eqref{alphadef} 
vanishes. This leads to
\begin{equation}\label{gammadef}
  \gamma \,=\, \frac12\Bigl(-(n-2+2\ell) + \sqrt{(n-2)^2 + 4b\ell
  (n-2+\ell)}\Bigr)
\end{equation}
Note that $\gamma = 0$ if either $b = 1$ (trivial case) or $\ell = 0$
(radially symmetric solutions). In the general case, we always have 
$\gamma + \ell \ge 0$, which means that $p(x)|x|^\gamma$ is bounded 
near the origin. 

\begin{rem}\label{Serrin}
Taking the other sign in front of the square root in \eqref{gammadef} 
would give more singular solutions of the eigenvalue equation, for which 
the gradient is not square integrable near the origin; these are 
examples of the ``pathological solutions'' considered by Serrin 
\cite{Se}.
\end{rem}

The eigenfunctions of the operator $L_{b,\ell,\gamma}$ are easy 
to determine when $\gamma$ is chosen so that $\delta = 0$, because
for any $k \in \N$ the differential equation
\[
  y\psi''(y) + (\alpha + 1 - y) \psi'(y) + k \psi(y) \,=\, 0\,,
  \qquad y > 0\,,
\] 
has a solution of the form $\psi(y) = \mathrm{L}_k^{(\alpha)}(y)$, 
where $\mathrm{L}_k^{(\alpha)}$ is the $k^{\mathrm{th}}$ (generalized)
Laguerre polynomial with parameter $\alpha$, see \cite[Section~22]{AS}. 
In particular, for $k = 0,1,2$, we have 
\[
  \mathrm{L}_0^{(\alpha)}(y) \,=\, 1\,, \qquad 
  \mathrm{L}_1^{(\alpha)}(y) \,=\, -y + \alpha + 1\,, \qquad 
  \mathrm{L}_2^{(\alpha)}(y) \,=\, \frac{y^2}{2} - (\alpha+2)y + 
  \frac{(\alpha+1)(\alpha+2)}{2}\,.
\]
Summarizing, the calculations above lead to the following statement.  

\begin{prop}\label{Lbspec}
Fix $b > 0$, $\ell \in \N$, $k \in \N$, and let
\begin{equation}\label{spec1}
  \alpha \,=\, \frac12 \sqrt{(n-2)^2 + 4b\ell(n-2+\ell)}\Bigr)\,, 
  \qquad \gamma \,=\, -\frac{n}{2} + 1 -\ell + \alpha\,.
\end{equation}
If $p : \R^n \to \R$ is a harmonic polynomial that is homogeneous 
of degree $\ell$ and if 
\begin{equation}\label{spec2}
  u(x) \,=\, p(x) |x|^\gamma e^{-|x|^2/4}\,\mathrm{L}_k^{(\alpha)}(|x|^2/4)\,,
  \qquad x \in \R^n\,,
\end{equation}
where $\mathrm{L}_k^{(\alpha)}$ is the $k^{\mathrm{th}}$ Laguerre polynomial 
with parameter $\alpha$, then $u$ is an eigenfunctions of the 
differential operator $L_b$ defined in \eqref{Lb2} in the sense that
\begin{equation}\label{spec3}
  L_b u \,=\, \lambda u\,, \qquad \hbox{where} \quad  
  \lambda = - \frac{\gamma+\ell}{2} -k\,. 
\end{equation}
\end{prop}

\begin{figure}[ht]
  \begin{center}
  \begin{picture}(480,220)% width and height of the picture
  \put(0,0){\includegraphics[width=1.00\textwidth]{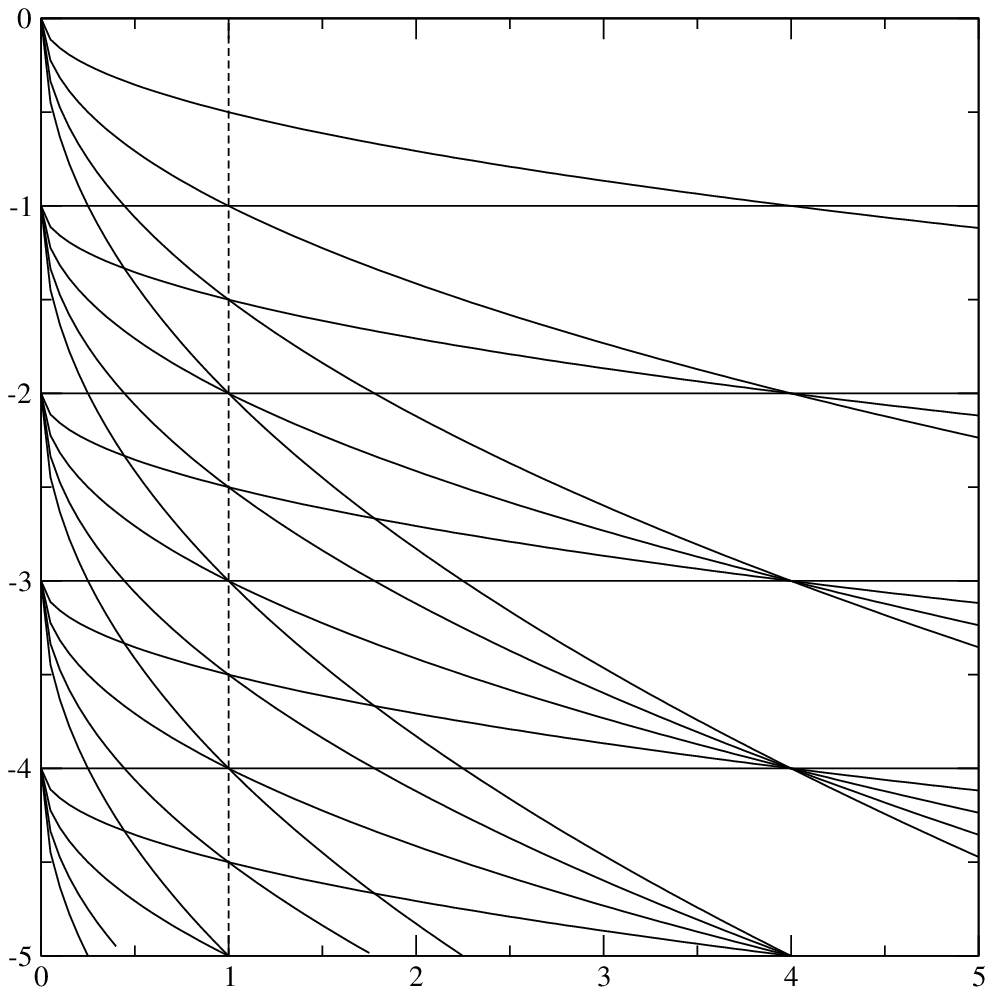}}
  \put(240,0){\includegraphics[width=1.00\textwidth]{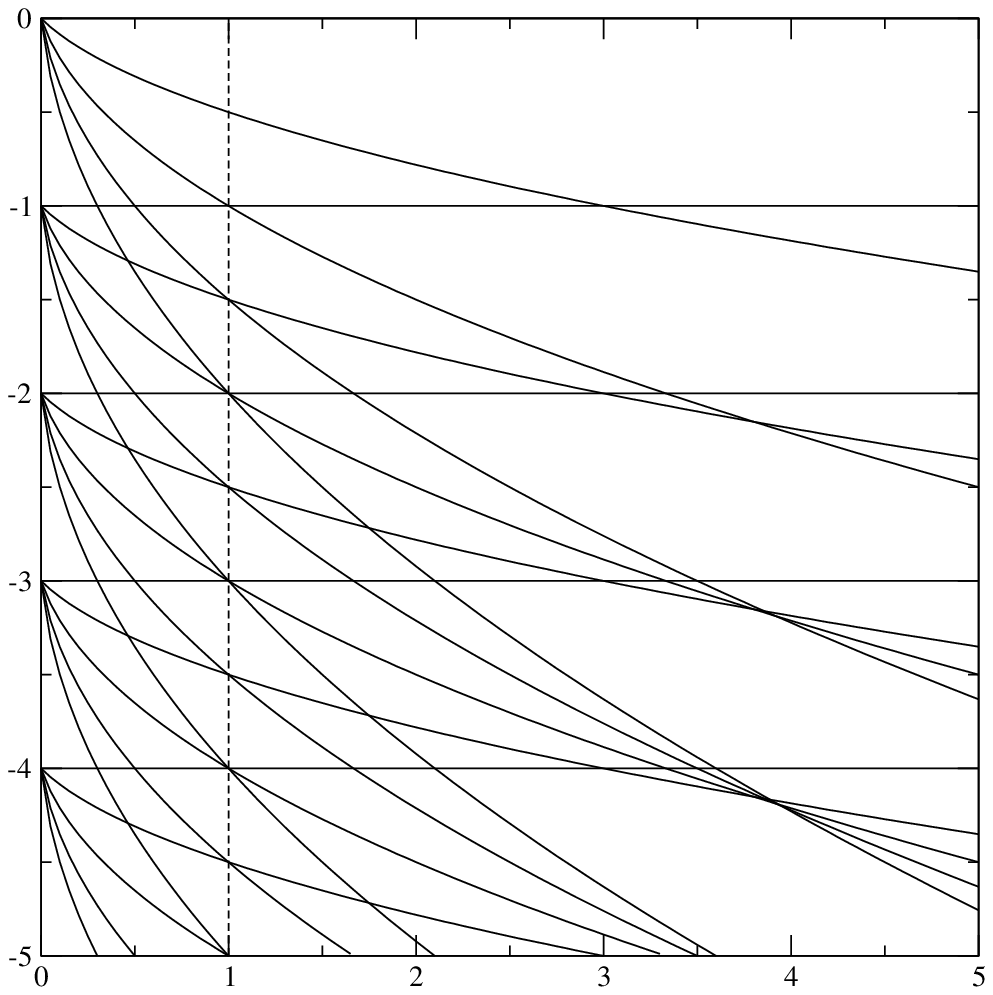}}
  \put(140,190){$n = 2$}
  \put(380,190){$n = 3$}
  \end{picture}
  \caption{The eigenvalues $\lambda = \lambda(b,\ell,k)$ of the
  linear operator \eqref{Lb1} are represented as a function of
  $b \in [0,5]$, for $\ell,k = 0,1,2,3,4$ and $n = 2$ (left) or
  $n = 3$ (right). The horizontal lines are eigenvalues
  corresponding to radially symmetric eigenfunctions ($\ell = 0$). 
  The vertical dashed line highlights the constant coefficient case
  $b =1$, where $\lambda = -\ell/2 -k$.\label{fig2}}
  \end{center}
\end{figure}

\begin{rem}\label{selfadjoint}
For all values of the parameter $b > 0$, the operator $L_b$ is selfadjoint 
in the weighted $L^2$ space
\[
  X \,=\, \bigl\{u \in L^2(\R^n)\,\big|\, e^{|x|^2/8}\,u \in 
  L^2(\R^n)\bigr\}\,.
\]
Indeed, if $v = e^{|x|^2/8}u$, it a direct calculation shows that
$\cL_b v = e^{|x|^2/8} L_b u$ where 
\begin{equation}\label{cLb}
  \cL_b v \,=\, \div(A_b \nabla v) - \frac{|x|^2}{16}v + 
  \frac{n}{4}v\,, \qquad x \in \R^n\,.
\end{equation}
The operator $\cL_b$ is obviously symmetric in $L^2(\R^n)$, and 
becomes selfadjoint when defined on its maximal domain; moreover
$\cL_b$ has compact resolvent, hence purely discrete spectrum.
By conjugation, the same properties hold for the operator $L_b$ 
in the weighted space $X$. In view of \eqref{spec2}, all
eigenfunctions given by Proposition~\ref{Lbspec} belong to 
$X$, and the method of separation of variables ensures that
the corresponding eigenfunctions can be chosen so as to 
form an orthogonal basis of $X$. We conclude that all 
eigenvalues of $L_b$ in $X$ are given by expressions 
\eqref{spec1}, \eqref{spec3}. The first few of them are
represented in Figure~\ref{fig2}, for $n = 2$ and $n = 3$. 
\end{rem}

\begin{rem}\label{betarem}
The eigenfunction of $L_b$ given by \eqref{spec2} satisfies $u(x) \sim
|x|^{\ell+\gamma}$ as $x \to 0$. In view of \eqref{spec1}, the exponent
$\ell+\gamma$ vanishes if $\ell = 0$ and is an increasing
function of $\ell \in \N$. On the other hand, using estimate
\eqref{Gbd1} and the fact that $u$ solves the elliptic equation
\eqref{uell} with $A_\infty = A_b$ and $f(x) = \frac12 x \cdot \nabla u
+ (\frac{n}{2}-\lambda)u$, it is not difficult to verify that
$|u(x) - u(0)| \le C |x|^\beta$ as $|x| \to 0$. This shows that
$\beta \le \ell + \gamma$ for any $\ell \ge 1$, and taking $\ell = 1$
we obtain
\begin{equation}\label{betaupper}
  0 \,<\, \beta \,\le\, - \frac{n}{2} + 1 + \frac12
  \sqrt{(n-2)^2 + 4b(n-1)}\Bigr)\,.
\end{equation}
The right-hand side of \eqref{betaupper} is an increasing function
of $b$ which converges to 0 as $b \to 0$ and to $1$ as $b \to 1$.
We conjecture that the upper bound \eqref{betaupper} is optimal
for $b \in (0,1)$. 
\end{rem}

\subsection{Properties of the principal eigenfunction\: the general case}
\label{ssec33}

After considering two particular examples, we now return to the
general case where the matrix $A_\infty(x)$ satisfies the assumptions
listed at the beginning of Section~\ref{sec2}. Much less is known on
the operator $L$ in that situation, but it is still possible to prove
that the kernel of $L$ in the space $L^2(m)$ is one-dimensional if $m
> n/2$, so that $L^2(m) \hookrightarrow L^1(\R^n)$. We claim that
the kernel of $L$ is spanned by the function $\phi : \R^n \to \R_+$
defined by
\begin{equation}\label{phi2}
  \phi(x) \,=\, \Gamma(x,0,1)\,, \qquad x \in \R^n\,,
\end{equation}
where $\Gamma(x,y,t)$ is the fundamental solution of \eqref{uhom}.  We
already know that $\phi$ is H\"older continuous, and the estimates
\eqref{Gammabounds} imply that $\phi$ satisfies the Gaussian bounds
\eqref{phibounds}.  Moreover the normalization condition
$\int_{\R^n}\phi(x)\dd x = 1$ follows from \eqref{IntGamma}. Finally,
we observe that the definition \eqref{phi2} reduces to \eqref{phi1} in
the particular case where $A_\infty = \1$.

\begin{lem}\label{philem}
If $\phi$ defined by \eqref{phi2}, then $\phi \in D(L)$ and 
$L\phi = 0$. 
\end{lem}

\begin{proof}
In view of \eqref{phibounds}, we have $\phi \in L^2(m)$ for any $m \ge
0$. Moreover, the definition \eqref{phi2} implies that $\phi =
e^{-H/2}\psi$ where $\psi(x) = \Gamma(x,0,1/2)$. As $\psi \in
L^2(\R^n)$, we thus have $\phi \in D(H) \subset H^1(\R^n)$. To prove
that $L\phi = 0$, we start from identity \eqref{Gammascale} with
$(y,t) = (0,1)$, and we set $\lambda = \sqrt{t}$ where $t > 0$
is a new parameter. This gives the useful relation
\begin{equation}\label{phi3}
  \phi(x) \,=\, t^{n/2}\,\Gamma\bigl(x \sqrt{t},0,t\bigr)\,, \qquad
  x \in \R^n\,, \quad t > 0\,.
\end{equation}
The idea is now to differentiate both sides of \eqref{phi3} with respect
to $t$, at point $t = 1$. Using the fact that, by definition,
the fundamental solution $(x,t) \mapsto \Gamma(x,y,t)$ is a solution
of the evolution equation \eqref{uhom} for any fixed $y \in \R^n$,
we obtain after straightforward calculations\:
\begin{equation}\label{phi4}
  0 \,=\, \frac{n}{2}\,\phi(x) + \frac{1}{2}\,x \cdot \nabla
  \phi(x) + \div(A_\infty(x)\nabla \phi(x)) \,\equiv\,
  \bigl(L\phi\bigr)(x)\,, \qquad x \in \R^n\,.
\end{equation}
This shows that $\phi \in D(L)$ and $L\phi = 0$. 
\end{proof}

To complete the proof of Proposition~\ref{phiprop}, it remains
to verify that the kernel of $L$ in the space of integrable
functions is one-dimensional. 

\begin{lem}\label{philem2}
If $\psi \in H^1(\R^n) \cap L^1(\R^n)$ satisfies $L\psi = 0$
and $\int_{\R^n}\psi \dd x = 1$, then $\psi = \phi$. 
\end{lem}

\begin{proof}
If $\psi$ is as in the statement, we define
\[ 
  u(x,t) \,=\, \frac{1}{t^{n/2}}\,\psi\Bigl(\frac{x}{\sqrt{t}}
  \Bigr)\,, \qquad x \in \R^n\,,\quad t > 0\,.
\]
We claim that, after modifying $\psi$ on a negligible set if needed,
we have the relation
\begin{equation}\label{phi5}
  \psi(x) \,=\, \int_{\R^n} \Gamma\bigl(x,y,1-t\bigr) u(y,t)\dd y
  \,\equiv\, \int_{\R^n} \Gamma\bigl(x,y\sqrt{t},1-t\bigr) \psi(y)\dd y\,,
\end{equation}
for all $x \in \R^n$ and all $t \in (0,1)$. Indeed, if we differentiate
with respect to time the last member of \eqref{phi5}, considered
as a distribution on $\R^n$, we obtain as in Lemma~\ref{philem}
\[
  \frac{\D}{\D t} \int_{\R^n} \Gamma\bigl(x,y\sqrt{t},1-t\bigr)\psi(y)\dd y
  \,=\, -\frac{1}{t}\int_{\R^n} \Gamma\bigl(x,y\sqrt{t},1-t\bigr)
  (L\psi)(y)\dd y \,=\, 0\,.
\]
Thus the first integral in \eqref{phi5} is independent of time,
and converges to $\psi(x)$ in $L^1(\R^n)$ as $t \to 1$, in
view of the properties \eqref{IntGamma}, \eqref{Gammabounds} of the
fundamental solution $\Gamma$. This proves \eqref{phi5}. 

We next take the limit $t \to 0$ in the second member of \eqref{phi5}, 
for a fixed $x \in \R^n$. As $\Gamma$ is H\"older continuous and satisfies
\eqref{Gammabounds}, it is clear that
\[
  \int_{\R^n} \Bigl(\Gamma(x,y,1-t) - \Gamma(x,y,1)\Bigr)
   u(y,t)\dd y \,\xrightarrow[~t \to 0~]{}\, 0\,.
\]
Moreover $u(\cdot,t) \rightharpoonup \delta_0$ (the Dirac measure
at the origin) as $t \to 0$, so that 
\[
  \int_{\R^n} \Gamma(x,y,1) u(y,t)\dd y \,\xrightarrow[~t \to 0~]{}\,
  \Gamma(x,0,1) \,=\, \phi(x)\,.
\]
We conclude that $\psi(x) = \phi(x)$ for (almost) all $x \in \R^n$. 
\end{proof}

The following properties of the derivatives of $\phi$ will be
useful. 

\begin{prop}\label{gradphi}
If $\phi$ is defined by \eqref{phi2}, then $|\nabla \phi| \in
L^2(m)$ for all $m \in \N$. In addition we have $\nabla \phi
\in L^q(\R^n)$ for $2 \le q < n/(1-\beta)$, where $\beta$ is as
in Proposition~\ref{GHolder} or \ref{GHolder2d}. 
\end{prop}

\begin{proof}
Let $\chi : \R^n \to [0,1]$ be a smooth and compactly supported
function such that $\chi(x) = 1$ if $|x| \le 1$. We also assume
that $\chi$ is radially symmetric and satisfies $x \cdot \nabla
\chi(x) \le 0$ for all $x \in \R^n$. Given any $m \in \N$, we
introduce for each $k \in \N^*$ the truncated weight function
\[
  p_k(x) \,=\, |x|^{2m} \chi(x/k)\,, \qquad x \in \R^n\,.
\]
We now multiply both sides of \eqref{phi4} by $p_k \phi$ and
integrate the resulting equality over $x \in \R^n$. After
integrating by parts, we obtain the relation
\[
  \int_{\R^n} p_k \bigl(\nabla \phi, A_\infty \nabla \phi\bigr)
  \dd x +  \int_{\R^n} \phi \bigl(\nabla p_k, A_\infty \nabla \phi\bigr)
  \dd x \,=\, \frac{n}{4} \int_{\R^n} p_k \phi^2 \dd x -
  \frac14 \int_{\R^n} (x\cdot\nabla p_k) \phi^2 \dd x\,,
\]
and using the ellipticity assumption \eqref{Aelliptic} we deduce
that
\begin{equation}\label{phi6}
  \lambda_1 \int_{\R^n} p_k |\nabla \phi|^2 \dd x \,\le\,
  \lambda_2 \int_{\R^n} \phi \,|\nabla p_k| |\nabla \phi| \dd x
  + \frac{n}{4} \int_{\R^n} p_k \phi^2 \dd x - \frac14
  \int_{\R^n} (x\cdot\nabla p_k) \phi^2 \dd x\,.
\end{equation}
As $\phi$ satisfies the Gaussian bound \eqref{phibounds}, we have
$\int p_k \phi^2 \dd x \to \int |x|^{2m} \phi^2 \dd x$ as $k \to \infty$.
To control the other terms in the right-hand side of \eqref{phi6},
we observe that
\[
  \nabla p_k(x) \,=\, 2m x |x|^{2m-2} \chi(x/k) + \frac{1}{k}
  \,|x|^{2m} \nabla \chi(x/k)\,, 
\]
from which we infer
\begin{align*}
  \int_{\R^n} \phi \,|\nabla p_k| |\nabla \phi| \dd x \,&\xrightarrow[~k
  \to \infty~]{}\, 2m \int_{\R^n} |x|^{2m-1}\phi |\nabla \phi| \dd x\,, \\
  \int_{\R^n} (x\cdot\nabla p_k) \phi^2 \dd x  \,&\xrightarrow[~k
  \to \infty~]{}\, 2m \int_{\R^n} |x|^{2m} \phi^2 \dd x\,.
\end{align*}
Thus taking the limit $k \to \infty$ in \eqref{phi6} and using
the monotone convergence theorem, we conclude that
\[
  \lambda_1 \int_{\R^n} |x|^{2m} |\nabla \phi|^2 \dd x \,\le\,
  2m\lambda_2 \int_{\R^n} |x|^{2m-1}\phi |\nabla \phi| \dd x
  + \Bigl(\frac{n}{4} - \frac{m}{2}\Bigr)\int_{\R^n} |x|^{2m}
  \phi^2 \dd x \,<\, \infty\,.
\]
This shows that $|x|^m |\nabla \phi| \in L^2(\R^n)$ for all $m \in \N$,
hence $|\nabla \phi| \in L^2(m)$ for all $m \in \N$. 

We next prove the second assertion in Proposition~\ref{gradphi}. We
know from \eqref{phi4} that $\phi$ satisfies the elliptic equation
\eqref{uell} with $f(x) = \frac12 x \cdot \nabla \phi(x) + \frac{n}{2}
\phi(x) = \frac12\div(x \phi)$. We thus have the representation
\eqref{urep}, which is valid even in the two-dimensional case because
$\int f(x) \dd x = 0$. Differentiating both sides of \eqref{urep} we
obtain
\begin{equation}\label{nablaphi}
  \nabla \phi(x) \,=\, \int_{\R^n} \nabla_x G(x,y)f(y)\dd y \,=\, 
  \int_{\R^n} \nabla_x G(x,y) \bigl({\TS \frac12} y \cdot \nabla
  \phi(y) + {\TS \frac{n}{2}} \phi(y)\bigr)\dd y\,,     
\end{equation}
for (almost) all $x \in \R^n$. This relation allows us to estimate
$\nabla \phi$ in $L^p(\R^n)$ for some $p > 2$ using the following lemma,
which is proved below. 

\begin{lem}\label{lemme_troiscas}
Let $p\in \big(1,\frac n{2-\beta}\big)$ where $\beta\in(0,1)$ is as in
\eqref{Gbd2}. If $f\in L^p(\R^n)$, the function $g$ defined by 
$g(x)=\int_{\R^n} \nabla_x G(x,y) f(y) \dd y$ belongs to $L^q(\R^n)$ with
$q$ such that $\frac 1q + \frac 1n = \frac 1p$.
\end{lem}

Let $p_* = n/(2{-}\beta)$ and $q_* = n/(1{-}\beta)$. We first assume
that $p_* \le 2$, which means that either $n = 2$, or $n = 3$ and
$\beta\le 1/2$. We know from \eqref{phibounds} and from the previous
step that $f \in L^2(m)$ for all $m \in \N$, hence $f \in L^p(\R^n)$
for all $p \in [1,2]$. We can thus apply Lemma~\ref{lemme_troiscas} to
\eqref{nablaphi} for any $p\in (1,p_*)$, and we obtain that $\nabla \phi$
belongs to $L^q(\R^n)$ for any $q \in (2,q_*)$, which gives the desired
conclusion.

We next consider the case where $p_* > 2$, which requires a bootstrap
argument.  For any $j \in \N$ with $j < n/2$, we denote $p_j =
2n/(n{-}2j)$, and we observe that $1/p_j = 1/n + 1/p_{j+1}$. As
before, we start with the knowledge that $f\in L^p(\R^n)$ for all $p
\in [1,2] \equiv [1,p_0]$, and a first application of
Lemma~\ref{lemme_troiscas} to \eqref{nablaphi} shows that $\nabla
\phi$ belongs to $L^q(\R^n)$ for all $q \in (2,p_1)$. Since we also
know that $|x|^m \nabla\phi \in L^2(\R^n)$ for any $m \in \N$, we
obtain by interpolation that $y \cdot \nabla \phi\in L^{r}(\R^n)$ for
any $r \in (2,q)$; in particular, we have shown that $f \in L^p(\R^n)$
for all $p \in [1,p_1)$. Repeating the same argument if needed, we
prove inductively that $f \in L^p(\R^n)$ for all $p \in [1,p_j)$
($j = 1,2,\dots$), until we reach the smallest $j \in \N^*$ such that
$p_j \ge p_*$.  At this point we know that $f \in L^p(\R^n)$ for all
$p \in [1,p_*)$, and Lemma~\ref{lemme_troiscas} implies that $\nabla
\phi \in L^q(\R^n)$ for all $q \in (2,q_*)$.
\end{proof}

\noindent{\bf Proof of Lemma~\ref{lemme_troiscas}.} In view of \eqref{Gbd2}, we
have $|g(x)| \le C\bigl(\psi_1(x) + \psi_2(x)\bigr)$ where
\[
  \psi_1(x) \,=\, \int_{\R^n} \frac{1}{|x-y|^{n-1}}\,|f(y)|\dd y\,, \qquad
  \psi_2(x) \,=\, \frac{1}{|x|^{1-\beta}}\int_{\R^n} \frac{1}{
  |x-y|^{n-2+\beta}}\,|f(y)|\dd y\,.
\]
The Hardy-Littlewood-Sobolev inequality directly yields $\psi_1\in
L^q(\R^n)$, see e.g. \cite{LL}.  To control $\psi_2$, we apply
Proposition~\ref{propkernel} with $k(x,y) = |x|^{\beta-1}
|x-y|^{2-n-\beta}$, which is a homogeneous kernel of degree $-d
=-(n-1)$. As $1 + 1/q = 1/p + d/n$ and $p<n/(2-\beta)<n=n/(n-d)$, we
only need to check the condition \eqref{kprop3}, namely
\[
  \int_{\R^n} \left(\frac{1}{|x-y|^{n-2+\beta}\,|y|^{n/q}}\right)^{n/(n-1)}
  \dd y \,<\, \infty\,, \qquad \hbox{for some } x \in \S^{n-1}\,.
\]
Our assumptions on $p$ are equivalent to $\frac n{n-1} < q <\frac n{1-\beta}$,
and these inequalities ensure that the integral above converges for small
$|y|$ and for large $|y|$, respectively. Moreover, the singularity at $y = x$
is integrable because $\beta < 1$, so that $\psi_2\in L^q(\R^n)$ by 
Proposition \ref{propkernel}. \QED

\begin{rem}\label{phirem}
Since the coefficient $A_\infty$ of the operator $L$ is Lipschitz
outside the origin, the classical regularity theory for second order
elliptic equations \cite{GT} implies that any eigenfunction of $L$, in
particular the principal eigenfunction $\phi$, is necessarily of class
$C^{1,\alpha}$ on $\R^n\setminus\{0\}$ for some $\alpha > 0$. However,
the example studied in Section~\ref{ssec32} shows that $\nabla \phi$
may have a singularity at the origin, as it is the case for the
function $\psi_2$ in the above proof. This indicates that estimate
\eqref{Gbd2} for the Green function cannot be substantially improved
in general. 
\end{rem}

\begin{rem}\label{symmetrization}
In the constant coefficient case, the relation \eqref{Lconj1} shows
that the operator $L$ is formally conjugated to a selfadjoint
operator. Such a property is not known to hold in general, but
the following observation can be made. If $\Phi : \R^n \to \R$ has
bounded second-order derivatives, a direct calculation shows that
\begin{equation}\label{Lconj2}
  e^{\Phi} L \bigl(e^{-\Phi} u\bigr) \,=\, \div(A_\infty \nabla u)
  + \frac{n}{2}\,u + V_\Phi \cdot \nabla u + W_\Phi u\,,
\end{equation}
for all $u \in C^2_c(\R^n)$, where the functions $V_\Phi$ and
$W_\Phi$ are given by
\[
  V_\Phi \,=\, \frac{x}{2} - 2 A_\infty \nabla \Phi\,, \qquad
  W_\Phi \,=\, \bigl(A_\infty\nabla \Phi,\nabla\Phi\bigr)
  - \div(A_\infty \nabla \Phi) - \frac{x}{2}\cdot\nabla \Phi\,.
\]
The conjugated operator \eqref{Lconj2} is symmetric in $L^2(\R^n)$ if
$V_\Phi = 0$, namely if $A_\infty \nabla \Phi = x/4$. For a general
matrix $A_\infty(x)$ satisfying the assumptions listed in
Section~\ref{sec2}, there is no function $\Phi$ with that property.
However, if we assume that $A_\infty(x)x = x$ for all $x \in \R^n$,
which is the case for the Meyers-Serrin matrix \eqref{MSexample}, 
we can take $\Phi(x) = |x|^2/8$ and we obtain, in close analogy
with \eqref{Lconj1}, 
\[
  e^{|x|^2/8} L \bigl(e^{-|x|^2/8} u\bigr) \,=\, \div(A_\infty \nabla u)
  - \frac{|x|^2}{16}\,u + \frac{n}{4}\,u\,.
\]
Note that, in that situation, we also have $L \phi = 0$ where $\phi$
is given by \eqref{phi1}. 
\end{rem}

\begin{rem}\label{principal}
It is interesting to note that, in general, the principal eigenfunction
of the operator $L$  is not given by the explicit expression
\eqref{phi1}. In fact, let $B : \R^n \to \cM_n(\R)$ be a matrix valued
function that is homogeneous of degree zero, smooth outside the origin,
symmetric and uniformly elliptic in the sense of \eqref{Aelliptic}.
We want to determine under which additional conditions the function
$\phi : \R^n \to \R$ defined by
\begin{equation}\label{phi7}
  \phi(x) \,=\, \exp\Bigl(-\frac{1}{4}
  \bigl(B(x)x,x\bigr)\Bigr)\,, \qquad x \in \R^n\,,
\end{equation}
is (up to normalization) the principal eigenfunction of the operator
$L$ for some appropriate choice of the diffusion matrix $A_\infty$.
This is certainly the case if we can construct $A_\infty$ in such
a way that $A_\infty(x)\nabla\phi(x) + \frac{x}{2}\phi(x) = 0$
for all $x \in \R^n$, because the desired property $L\phi = 0$
then follows by taking the divergence with respect to the variable
$x$. In view of \eqref{phi7}, the condition on $A_\infty$ becomes
\begin{equation}\label{phi8}
  A_\infty(x) B(x) x + \frac12 A_\infty(x) \bigl(\nabla B(x)x,x\bigr)
  \,=\, x\,,\qquad x \in \R^n\,, 
\end{equation}
where $\bigl(\nabla B(x)x,x\bigr) \in \R^n$ denotes the vector with
components $\bigl(\partial_j B(x)x,x\bigr)$ for $j = 1,\dots,n$.
Consider the matrix $M(x)$ and the vectors $\zeta(x), \xi(x)$
defined as follows\:
\[
  M(x) \,=\, B(x)^{1/2} A_\infty(x) B(x)^{1/2}\,, \quad
  \zeta(x) \,=\, B(x)^{1/2}x\,, \quad \xi(x) \,=\,
  \frac12 B(x)^{-1/2}\bigl(\nabla B(x)x,x\bigr)\,.
\]
Then our condition \eqref{phi8} can be written in the equivalent
form
\begin{equation}\label{phi9}
  M(x) \zeta(x) + M(x) \xi(x) \,=\, \zeta(x)\,, \qquad x \in \R^n\,.
\end{equation}
Moreover, we observe that
\[
  2\bigl(\zeta(x),\xi(x)\bigr) \,=\, \sum_{j=1}^n x_j
  \bigl((\partial_jB)x,x\bigr) \,=\, \Bigl(\Bigl(\sum_{j=1}^n
  x_j \partial_j B\Bigr)x,x\Bigr) \,=\, 0\,,
\]
because $B(x)$ is homogeneous of degree zero; we deduce that $\zeta(x)
\bot \xi(x)$ for all $x \in \R^n$. Now, if $M(x)$ is the symmetric
matrix with components $M_{ij}(x)$ defined by
\[
  M_{ij}(x) \,=\, \frac{|\xi(x)|^2}{|\zeta(x)|^4}\,\zeta_i(x)\zeta_j(x)
  - \frac{1}{|\zeta(x)|^2}\,\bigl(\zeta_i(x)\xi_j(x) + \xi_i(x)\zeta_j(x)
  \bigr) + \delta_{ij}\,, \qquad x \in \R^n\setminus\{0\}\,,
\]
it is straightforward to verify that \eqref{phi9} hold for all $x \in
\R^n$, and that the map $x \mapsto M(x)$ is homogeneous of degree
zero. Moreover, the matrix $M(x)$ is positive definite if we assume
that $|\xi(x)| \le \kappa |\zeta(x)|$ for some $\kappa < 1$, which
is the case if $\nabla B$ is sufficiently small compared to $B$
on the unit sphere $\S^{n-1}$. Under that assumption, if we set
$A_\infty(x) = B(x)^{-1/2} M(x) B(x)^{-1/2}$, we conclude that $A_\infty$
satisfies the assumptions listed in Section~\ref{sec2} and that the
operator $L$ defined by \eqref{Ldef} has the property that $L\phi = 0$,
where $\phi$ is defined by \eqref{phi7}.
\end{rem}

%%%%%%%%%%%%%%%%%%%%%%%%%%%%%%%%%%%%%%%%%%%%%%%%%%%%%%%%%%%%%%%%%%%%%

\section{Long-time asymptotics in the linear case}\label{sec4}

This section is devoted to the proof of Theorem~\ref{main1}.
We start from the rescaled equation \eqref{veq} with $\cN = 0$,
namely
\begin{equation}\label{veqlin}
  \partial_\tau v \,=\, \div\Bigl(A\bigl(ye^{\tau/2}\bigr)\nabla v\Bigr)
  + \frac{1}{2}\,y\cdot\nabla v + \frac{n}{2}\,v\,,
  \qquad y \in \R^n\,,\quad \tau > 0\,,
\end{equation}
and we consider it as an evolution equation in the weighted
space $L^2(m)$ defined in \eqref{L2mdef}. 

\begin{lem}\label{wellposed}
For any $m \ge 0$, the Cauchy problem for Eq.~\eqref{veqlin} is
globally well-posed in $L^2(m)$. 
\end{lem}

\begin{proof}
That statement, as well as all subsequent claims regarding existence
and regularity of solutions to \eqref{veqlin}, can be justified by
the following standard arguments. If we undo the change of variables
\eqref{self-similar}, we obtain the linear diffusion equation
\eqref{diffeq} with $N=0$, namely
\begin{equation}\label{uplain}
  \partial_t u(x,t) \,=\, \div \bigl(A(x)\nabla u(x,t)\bigr)\,,
  \qquad x \in \R^n\,,\quad t > 0\,,
\end{equation}
which is known to define an analytic evolution semigroup in the
Hilbert space $L^2(\R^n)$, see Section~\ref{sec2} for a similar
analysis. We set $u(x,t) = p(x) \tilde u(x,t)$, where
$p(x) = (1{+}|x|^2)^{-m/2}$. The new function $\tilde u$ then
satisfies the modified evolution equation
\begin{equation}\label{tildeu}
  \partial_t \tilde u \,=\, \div \bigl(A(x)\nabla \tilde u\bigr)
  + \frac{2}{p}\,\bigl(\nabla p\,,\, A(x)\nabla \tilde u\bigr)
  + \frac{1}{p}\,\div \bigl(A(x)\nabla p\bigr)\tilde u\,,
\end{equation}
which differs from \eqref{uplain} by a relatively compact
perturbation, in the sense of operator theory. It follows
\cite[Section~3.2]{Pa} that \eqref{tildeu} defines an analytic
semigroup in $L^2(\R^n)$, which amounts to saying that \eqref{uplain}
defines an analytic semigroup in $L^2(m)$. In particular,
given initial data $u_0 \in L^2(m)$, Eq.~\eqref{uplain} has a
unique solution $u \in C^0([0,+\infty),L^2(m)) \cap C^1((0,+\infty),
L^2(m))$ such that $u(0) = u_0$. Moreover $\nabla u \in C^0((0,+\infty),
L^2(m)^n) \cap L^2((0,T),L^2(m)^n)$ for any ${T > 0}$. Applying
now the change of variables \eqref{self-similar}, which leaves
the space $L^2(m)$ invariant, we conclude in particular that,
given initial data $v_0 \in L^2(m)$, equation \eqref{veqlin} has a
unique global solution $v \in C^0([0,+\infty),L^2(m))$ such
that $v(0) = v_0$. 
\end{proof}

\subsection{Spectral decomposition of the solution}\label{ssec41}

We assume from now on that $m > n/2$, so that $L^2(m) \hookrightarrow
L^1(\R^n)$. If $v \in C^0([0,+\infty),L^2(m))$ is a solution of
\eqref{veqlin} with initial data $v_0 \in L^2(m)$, we observe that
\begin{equation}\label{intconserved}
  \int_{\R^n} v(y,\tau)\dd y \,=\, \int_{\R^n} v_0(y)\dd y\,,
  \qquad \hbox{for all }\tau \ge 0\,.
\end{equation}
Indeed, if $u \in C^0([0,+\infty),L^2(m)) \cap C^1((0,+\infty),L^2(m))$
is the corresponding solution of \eqref{uplain}, we have
\[
  \frac{\D}{\D t}\int_{\R^n} u(x,t) \dd x \,=\,
  \int_{\R^n} \div \bigl(A(x)\nabla u(x,t)\bigr)\dd x \,=\, 0\,,
  \qquad \hbox{for all } t  > 0\,,
\]
where the last equality follows from Lemma~\ref{intlem} since $\div
\bigl(A\nabla u\bigr) \in L^2(m)$ and $A\nabla u \in L^2(m)^n$ for any
$t > 0$.  It follows that the integral of $u(\cdot,t)$ does not depend
on time, and the same property holds for the rescaled function
$v(\cdot,\tau)$ in view of \eqref{self-similar}. This gives
\eqref{intconserved}.

We also recall that, in view of \eqref{Alimits} and \eqref{Aunif}, the diffusion
matrix $A$ can be decomposed as
\begin{equation}\label{ABdecomp}
  A(x) \,=\, A_\infty(x) + B(x)\,, \qquad x \in \R^n\,,
\end{equation}
where $A_\infty$ is homogeneous of degree zero and the remainder $B$ satisfies
\begin{equation}\label{Bestimate}
  \sup_{x \in \R^n} (1+|x|)^\nu \,\|B(x)\| \,<\, \infty\,, \qquad \hbox{for some }
  \nu > 0\,.
\end{equation}
Let $L$ be the limiting operator \eqref{Ldef}, and $\phi \in L^2(m)$ be the
principal eigenfunction of $L$ given by Proposition~\ref{phiprop}. We
decompose the solution of \eqref{veqlin} in the following way\:
\begin{equation}\label{vdecomp}
  v(y,\tau) \,=\, \alpha \phi(y) + w(y,\tau)\,, \qquad
  \hbox{where} \qquad \alpha \,=\, \int_{\R^n} v(y,\tau)\dd y\,.
\end{equation}  
Since $\phi$ is normalized so that $\int_{\R^n}\phi \dd y = 1$, it
follows from \eqref{vdecomp} that $\int_{\R^n} w(y,\tau) \dd y = 0$
for all $\tau \ge 0$. Moreover, in view of \eqref{veqlin} and
\eqref{phieq}, the evolution equation satisfied by $w$ is
\begin{equation}\label{weqlin}
  \partial_\tau w \,=\, \div\Bigl(A\bigl(ye^{\tau/2}\bigr)\nabla w\Bigr)
  + \frac{1}{2}\,y\cdot\nabla w + \frac{n}{2}\,w + r_1\,,
  \qquad y \in \R^n\,,\quad \tau > 0\,,
\end{equation}
where
\begin{equation}\label{r1def}
  r_1(y,\tau) \,=\, \alpha \div\bigl(B(y e^{\tau/2})\nabla \phi(y)\bigr)\,.
\end{equation}

\begin{rem}\label{remdecomp}
As simple as it may seem, the decomposition \eqref{vdecomp} is an
essential step in the proof of Theorem~\ref{main1}. To understand its
meaning, let us assume for the moment that the solutions of \eqref{veqlin}
are well approximated, for large times, by those of the limiting equation
\eqref{veqlimit}; this is certainly expected in view of \eqref{ABdecomp},
\eqref{Bestimate}. So our task is to understand the long-time behavior
of the semigroup $e^{\tau L}$ generated by the limiting operator \eqref{Ldef}.
In the weighted space $L^2(m)$ with $m > n/2$, we claim that $0$ is a simple
eigenvalue of $L$, and that the rest of the spectrum is contained in the 
half-plane $\{z \in \C\,|\, \Re(z) \le -\mu\}$ for some $\mu > 0$. This
is in fact what Theorem~\ref{main1} asserts in the particular case
where $A = A_\infty$. As is easily verified, the spectral projection $P$ onto
the kernel of $L$ is the map $v \mapsto Pv$ defined by
\[
  (P v)(y) \,=\, \phi(y) \int_{\R^n} v(y)\dd y\,, \qquad y \in \R^n\,.
\]
With this notation, the decomposition \eqref{vdecomp} simply reads $v
= Pv + w$ where $w = (1-P)v$. To prove Theorem~\ref{main1}, our
strategy is to show that the solutions of \eqref{weqlin} in the invariant
subspace $L^2_0(m) \equiv (1-P)L^2(m)$ decay exponentially to zero as
$\tau \to +\infty$, even though the equation for $w$ involves the
time-dependent matrix $A(y e^{\tau/2})$ instead of the limiting matrix
$A_\infty(y)$. As we shall see in the rest of this section, the
exponential decay of $w$ can be established using appropriate energy
estimates.
\end{rem}

\subsection{Weighted estimates for the perturbation}\label{ssec42}

Given any solution $w$ of \eqref{weqlin} in $L^2(m)$, we consider the
energy functional 
\begin{equation}\label{emdef}
  e_{m,\delta}(\tau) \,=\, \frac 12 \int_{\R^n} (\delta+|y|^2)^m w(y,\tau)^2
  \dd y\,, \qquad \tau \ge 0\,,
\end{equation}
where $\delta > 0$ is a parameter that will be fixed later on. This quantity
is differentiable for $\tau > 0$, and using \eqref{weqlin} we find
\begin{align}\nonumber
   \partial_\tau e_{m,\delta}(\tau) \,&=\, \int (\delta+|y|^2)^m w
   \Bigl[\div(A(ye^{\tau/2})\nabla w) + \frac12 (y\cdot\nabla)w +
     \frac{n}{2}w + r_1\Bigr] \dd y \\ \label{em1}
   \,&=\, - \int \Langle \nabla\big((\delta+|y|^2)^m w
   \big),A(ye^{\tau/2})\nabla w\Rangle \dd y + \frac14 \int
   (\delta+|y|^2)^m y\cdot\nabla(w^2) \dd y \\ &\hskip 12pt \nonumber
   \,+\, \frac{n}{2} \int(\delta+|y|^2)^m |w|^2 \dd y -\alpha \int \Langle
   \nabla\big((\delta+|y|^2)^m w \big),B(ye^{\tau/2})\nabla \phi\Rangle\dd y\,,
\end{align}
where the second equality is obtained after integrating by parts and using
the definition \eqref{r1def} of the quantity $r$. Here and in what follows,
it is understood that all integrals are taken over the whole space $\R^n$. 
In view of the elementary identities
\begin{align*}
  \nabla \bigl((\delta+|y|^2)^m\bigr) \,&=\, 2my\,(\delta+|y|^2)^{m-1}\,, \\ 
  \div \bigl(y(\delta+|y|^2)^m\bigr) \,&=\, (n+2m)(\delta+|y|^2)^m -
  2m\delta(\delta+|y|^2)^{m-1}\,,
\end{align*}
we can write \eqref{em1} in the equivalent form
\begin{align}\nonumber
  \partial_\tau e_{m,\delta}(\tau) \,&=\, -\int (\delta+|y|^2)^m \Langle\nabla w,
  A(ye^{\tau/2})\nabla w\Rangle \dd y - 2m \int (\delta+|y|^2)^{m-1} w \Langle
  y,A(ye^{\tau/2})\nabla w\Rangle \dd y \\ \label{em2}
  &\hskip 12pt \,+\, \frac {n{-}2m}{4} \int (\delta+|y|^2)^m |w|^2\dd y
  + \frac{m\delta}2 \int(\delta+|y|^2)^{m-1} |w|^2 \dd y\\ \nonumber 
  \,&-\,\alpha \int (\delta+|y|^2)^m \Langle \nabla w ,B(ye^{\tau/2})
  \nabla \phi\Rangle \dd y - 2\alpha m \int (\delta+|y|^2)^{m-1} w \Langle y ,
  B(ye^{\tau/2})\nabla \phi\rangle \dd y\,.
\end{align}
The last term in the first line of the right-hand side has no obvious
sign, but applying H\"older's inequality we can estimate it as follows\:
\begin{align*}
  2m\Bigl|&\int (\delta+|y|^2)^{m-1} w \Langle y, A(ye^{\tau/2})\nabla
  w\Rangle \dd y\Bigr| \\
  \,&\le\, \frac14 \int (\delta+|y|^2)^m \Langle \nabla w, A(ye^{\tau/2})
  \nabla w\Rangle \dd y + 4m^2 \int (\delta+|y|^2)^{m-2} |w|^2 \Langle y,
  A(ye^{\tau/2}) y \Rangle \dd y \\
  \,&\le\, \frac14 \int (\delta+|y|^2)^m \Langle \nabla w, A(ye^{\tau/2})
  \nabla w\Rangle \dd y + C m^2 \int(\delta+|y|^2)^{m-1} |w|^2 \dd y\,,
\end{align*}
where in the last line we used the obvious fact that $(\delta+|y|^2)^{m-2}|y|^2 \le
(\delta+|y|^2)^{m-1}$. Here and below, we denote by $C$ any positive constant
depending only on the properties of the matrix $A$. We proceed in a similar
way to bound both terms in the last line of \eqref{em2}, and this leads to 
the inequality
\begin{equation}\label{em3}
  \begin{split}
  \partial_\tau e_{m,\delta}(\tau) \,&\le\,  -\frac12 \int
  (\delta+|y|^2)^m \Langle\nabla w,A(ye^{\tau/2})\nabla w\Rangle \dd y + 
  \frac {n{-}2m}{2}\,e_{m,\delta}(\tau)\\
  \,&+\, \bigl(m\delta + C_1 m^2\bigr)\,e_{m-1,\delta}(\tau) + C_2\alpha^2
  \int (\delta+|y|^2)^m \|B(ye^{\tau/2})\|^2|\nabla \phi|^2\dd y\,,
  \end{split}
\end{equation}
for some positive constants $C_1, C_2$. 

\begin{rem}\label{emrem}
If we forget for the moment the last term in \eqref{em3}, assuming thus
that $B\equiv 0$, we have shown that
\begin{equation}\label{em4}
  \partial_\tau e_{m,\delta}(\tau) \,\le\, \frac{n{-}2m}{2}\,e_{m,\delta}(\tau)
  + \bigl(m\delta + C_1 m^2\bigr)\,e_{m-1,\delta}(\tau)\,,
  \qquad \tau > 0\,.
\end{equation}
If $m = 0$, so that $e_{0,\delta}(\tau) = \frac12\|w(\cdot,\tau)\|_{L^2}^2$, the
last term in \eqref{em4} disappears, and we are left with the differential
inequality $\partial_\tau e_{0,\delta} \le (n/2)\,e_{0,\delta}$ which allows for
an exponential growth in time. This is compatible with the spectral
picture in Figure~\ref{fig1}, where the essential spectrum of the operator
$L$ fills the half-plane $\{\Re(z) \le n/4\}$ if $m = 0$. Now, if we assume that
$m > n/2$, the coefficient in front of $e_{m,\delta}$ in the right-hand side of
\eqref{em4} becomes negative, but then we also have the ``lower order term''
proportional to $e_{m-1,\delta}$ which makes it impossible to prove exponential
decay using only \eqref{em4}. The obstacle we hit here is in the nature of
things\: we cannot prove exponential decay in time of the solution of
\eqref{weqlin} if we do not use the crucial fact that $\int_{\R^n} w \dd y = 0$.
\end{rem}

\subsection{Evolution equation for the antiderivative}\label{ssec43}

If we want to study evolutionary PDEs using just $L^2$ energy estimates,
it is not straightforward to exploit the information, if applicable,
that the solutions under consideration have zero mean. In the
one-dimensional case, the following elementary observation was made in
\cite{GR} and applied to the analysis of parabolic or damped hyperbolic
equations\: if $u : \R \to \R$ belongs to $L^2(m)$ for some $m \ge 1$
and has zero mean, the primitive function $U(x) = \int_{-\infty}^x u(y)
\dd y$ is square integrable and satisfies $\|U\|_{L^2} \le 2\|xu\|_{L^2}$
(this is a variant of Hardy's inequality). The idea is then to control
the evolution of the primitive $U$ using $L^2$ energy estimates, and it
turns out that this procedure takes into account the information that
the original function $u$ has zero mean. 

In the same spirit, we propose here an approach that works in
dimensions two and three, and can be extended to cover the
higher-dimensional cases as well (see Section~\ref{ssec45} below).
If $m > n/2$ and $w \in L^2_0(m)$, so that $\int_{\R^n}w(y)\dd y = 0$, the
idea is to define the ``antiderivative'' $W$ of $w$ as the
solution of the elliptic equation
\begin{equation}\label{Welliptic}
  -\div \bigl(A_\infty(y)\nabla W(y)\bigr) \,=\, w(y)\,, \qquad
  y \in \R^n\,.
\end{equation}
More precisely we set $W = K[w]$, where $K$ denotes the integral
operator \eqref{defK} whose kernel is the Green function $G(x,y)$ of
the differential operator in \eqref{Welliptic}, see Section~\ref{ssec25}.
We recall that, if $m \in (n/2,n/2 + \beta)$ where $\beta \in (0,1)$
is defined in Proposition~\ref{GHolder}, then $K$ is a bounded linear
operator from $L^2_0(m)$ to $L^2(m{-}2)$. Moreover, as is shown in
Proposition~\ref{Kernelprop2}, the operator $K$ can be extended so as
to act on first order distributions of the form $w = \div g$, where
$g \in L^2(m{-}1)^n$.

The definition \eqref{Welliptic} has the property that the
antiderivative $W$ satisfies a nice equation if $w$ evolves according
to \eqref{weqlin}.

\begin{lem}\label{Weqlem}
Assume that $m \in (n/2,n/2 + \beta)$, and that $w \in C^0([0,+\infty),
L^2_0(m))$ is a solution of Eq.~\eqref{weqlin}. If we define
$W(\cdot,\tau) = K[w(\cdot,\tau)]$ for $\tau \ge 0$, then $W \in
C^0([0,+\infty),L^2(m{-}2))$ is a solution of the evolution equation
\begin{equation}\label{Weqlin}
  \partial_\tau W \,=\, \div\bigl(A_\infty(y)\nabla W\bigr) + \frac12
  \,y\cdot\nabla W + \frac{n{-}2}{2}\,W + R_1\,,
\end{equation}
where the remainder term $R(y,\tau)$ is given by
\begin{equation}\label{R1def}
  R_1(\cdot,\tau) \,=\, K\left[\div \Bigl(B(\cdot\,e^{\tau/2})(\alpha \nabla
  \phi + \nabla w)\Bigr)\right]\,, \qquad \tau \ge 0\,.
\end{equation}
\end{lem}

\begin{proof}
We rewrite the evolution equation \eqref{weqlin} in the
equivalent form
\begin{equation}\label{weqlin2}
  \partial_\tau w \,=\, \div\Bigl(A_\infty(y)\nabla w\Bigr)
  + \frac{1}{2}\,\div\bigl(y w\bigr) + \tilde r_1\,,
\end{equation}
where $\tilde r_1(y,\tau) = \div \big[B(y e^{\tau/2})\bigl(\alpha
\nabla\phi(y) + \nabla w(y,\tau)\bigr)\bigr]$, and we apply the
linear operator $K$ to both sides of \eqref{weqlin2}. Since $W = K[w]$
and $R_1 = K[\tilde r_1]$ by definition, it remains to treat the first
two terms in the right-hand side, which are in divergence form so
that we can apply Corollary~\ref{2Kcor}. We assume here that
$\nabla w(\cdot,\tau)\in L^2(m)^n$, which is the case as soon as
$\tau > 0$. We make the following observations\:

\smallskip\noindent{\bf 1.}
Let $F = -(K \circ \div)\bigl[A_\infty \nabla w\bigr]$, where
$w \in L^2_0(m)$ and $\nabla w \in L^2(m)^n$. By \eqref{Kcircdef}, we have
\[
  F(x) \,=\, \int_{\R^n} \Bigl(\nabla_y G(x,y) \cdot \bigl(A_\infty(y) 
  \nabla w(y)\bigr)\Bigr) \dd y \,=\, \int_{\R^n} \bigl(A_\infty(y)
  \nabla_y G(x,y)\,, \nabla w(y)\bigr) \dd y\,, 
\]
for (almost) all $x \in \R^n$. If $w \in C^\infty_c(\R^n)$, the right-hand
side is equal to $w(x)$ by \eqref{GDirac}, and using a density argument
we deduce that $F = w$ in the general case. As $w = -\div \bigl(A_\infty
\nabla W\bigr)$ by Remark~\ref{rightinverse}, this gives the elegant
relation $(K\circ \div)\bigl[A_\infty \nabla w\bigr] = \div
\bigl(A_\infty\nabla W\bigr)$. 

\smallskip\noindent{\bf 2.}
As the matrix $A_\infty$ is homogeneous of degree zero, the Green
function $G$ has the following property\: there exists a constant
$c_0 \in \R$ such that, for all $x,y \in \R^n$ with $x \neq y$, 
\begin{equation}\label{GEuler}
  (n-2)G(x,y) + x\cdot\nabla_x G(x,y) + y\cdot\nabla_y G(x,y)
  \,=\, -c_0\,. 
\end{equation}
Indeed, if $n \ge 3$, we have $\lambda^{n-2}G(\lambda x,\lambda y) =
G(x,y)$ for any $\lambda > 0$, and this implies the Euler relation
\eqref{GEuler} with $c_0 = 0$; when $n = 2$, we deduce \eqref{GEuler}
directly from \eqref{Gdef3}. If $w \in L^2_0(m)$ and $g(y) = y w(y)$,
then $g \in L^2(m{-}1)^n$ and, in view of \eqref{Kcircdef} and 
Proposition~\ref{Kernelprop2}, we have
\begin{align*}
  \bigl[K\circ \div g\bigr](x) &=
  -\int \bigl(y\cdot\nabla_y G(x,y) \bigr)w(y)\dd y 
  = \int \bigl(x\cdot\nabla_x G(x,y)+(n-2)G(x,y)+c_0 \bigr)w(y)\dd y\\
  &= \int \Bigl({\div}_x \bigl(x\,G(x,y)\bigr) -2 G(x,y)\Bigr)w(y)\dd y
  = \Bigl(\div\bigl(x K[w]\bigr) - 2 K[w]\Bigr)(x)\,, 
\end{align*}
where we used \eqref{GEuler} and the fact that $\int w \dd y = 0$. After
changing $x$ into $y$, the relation above becomes $(K\circ\div)[yw] =
\div (y W) - 2 W = y \cdot \nabla W + (n{-}2)W$. 

\smallskip\noindent
Summarizing, if apply the operator $K$ to all terms in \eqref{weqlin2}
and use the steps {\bf 1} and {\bf 2} above, we arrive at \eqref{Weqlin}. 
\end{proof}

Notice that Equation~\eqref{Weqlin} is very similar to \eqref{weqlin}, with the
important difference that the ``amplification factor'' $n/2$ in the
right-hand side of \eqref{weqlin} is reduced to $(n{-}2)/2$ in
\eqref{Weqlin}. This makes it possible to control the evolution of
the antiderivative $W$ using energy estimates if $n \le 3$. To this 
end, we introduce the following additional energy functional\:
\begin{equation}\label{Emdef}
  E_{m-2,\delta}(\tau) \,=\, \frac 12 \int_{\R^n} (\delta+|y|^2)^{m-2}
  |W(y,\tau)|^2 \dd y\,, \qquad \tau \ge 0\,.
\end{equation}
Repeating the same calculations as in Section~\ref{ssec42}, we obtain
in analogy with \eqref{em3}\:
\begin{equation}\label{Em1}
  \begin{split}
  \partial_\tau E_{m-2,\delta}(\tau) \,&\le\,  -\frac12 \int (\delta+|y|^2)^{m-2}
  \Langle\nabla W, A_\infty \nabla W\Rangle\dd y \,+\, \frac{n{-}2m}{2}\,
  E_{m-2,\delta}(\tau)\\ 
  \,&+\,\Bigl((m{-}2)\delta + C_1(m{-}2)^2\Bigr) E_{m-3,\delta}(\tau)
  \,+\, \int (\delta+|y|^2)^{m-2} \,W R_1 \dd y\,.
  \end{split}
\end{equation}

\begin{rem}\label{stupidrem}
In the derivation of \eqref{Em1}, the coefficient in front of
$E_{m-2,\delta}(\tau)$ in the right-hand side is obtained through the
elementary calculation
\[
  \frac{n-2m}{4} \,=\, \frac{n-2}{2} - \frac{n + 2(m-2)}{4}\,,
\]
where we observe that the smaller ``amplification factor'' $(n{-}2)/2$
in \eqref{Weqlin} is exactly compensated by the fact that we estimate
the antiderivative $W$ in $L^2(m{-}2)$ instead of $L^2(m)$.  As a
result, we obtain exactly the same coefficient $(n{-}2m)/2$ in both
estimates \eqref{em3} and \eqref{Em1}.
\end{rem}

\subsection{Exponential decay of the perturbation in low 
dimensions}\label{ssec44}

In this section, we assume that $n=2$ or $n=3$, and we combine estimates
\eqref{em3}, \eqref{Em1} to prove that the solutions of \eqref{weqlin}
in $L^2_0(m)$ converge exponentially to zero as $\tau \to +\infty$.
For the moment, we assume that $m \in (n/2,n/2+\beta)$, so that we can
apply Proposition~\ref{Kernelprop1} to control the antiderivative $W$,
and for convenience we also suppose that $m \le 2$ (note, however,
that all upper bounds on $m$ will be relaxed later). The crucial
observation is that the coefficient of $E_{m-3,\delta}$ in \eqref{Em1}
vanishes if $m = 2$, and becomes negative if $m < 2$ provided that the
parameter $\delta > 0$ is chosen large enough. Therefore, we
assume that
\begin{equation}\label{lambdadef}
  m \,=\, \frac{n}{2} + \lambda\,, \quad \hbox{where}\quad
  0 < \lambda < \beta\,, \quad \lambda \,\le\, 2 - \frac{n}{2}\,, 
  \qquad \hbox{and}\quad \delta \,\ge\, 2 C_1(2-m)\,. 
\end{equation}
Under these hypotheses, inequalities \eqref{Em1}, \eqref{em3} become
\begin{align}\nonumber
   \partial_\tau E_{m-2,\delta}(\tau) \,&\le\, -\frac{\lambda_1}{2} \int
   (\delta+|y|^2)^{m-2} |\nabla W|^2 \dd y - \lambda E_{m-2,\delta}(\tau)
   + \int (\delta+|y|^2)^{m-2} W R_1 \dd y\,, \\ \label{Eem}
  \partial_\tau e_{m,\delta}(\tau) \,&\le\, - \frac{\lambda_1}{2}\int
  (\delta+|y|^2)^m |\nabla w|^2 \dd y \,-\, \lambda e_{m,\delta}(\tau)
  +\, C_3\,e_{m-1,\delta}(\tau)\\ \nonumber
  & \hskip 13pt \,+\, C_2\,\alpha^2\int (\delta+|y|^2)^m \|B(ye^{\tau/2})\|^2
  |\nabla \phi|^2 \dd y\,,
\end{align}
where $C_3 = m\delta + C_1 m^2$ and $\lambda_1 > 0$ is as in \eqref{Aelliptic}. 

The next step is a simple interpolation argument which allows us to
control the undesirable quantity $C_3\,e_{m-1,\delta}$ in \eqref{Eem}
using the negative terms involving $\nabla w$ and $\nabla W$.
In view of \eqref{Welliptic}, we have
\begin{align*}
  2 e_{m-1,\delta} \,&=\, \int(\delta+|y|^2)^{m-1}|w|^2\dd y \,=\,
  -\int(\delta+|y|^2)^{m-1} w \div\bigl(A_\infty \nabla W\bigr) \dd y\\
  \,&=\, \int(\delta+|y|^2)^{m-1}\Langle \nabla w,A_\infty\nabla W\Rangle\dd y
  + 2(m{-}1)\int(\delta+|y|^2)^{m-2} w\Langle y ,A_\infty\nabla W\Rangle \dd y\\
  \,&\le\, \epsilon_0 \int (\delta+|y|^2)^m |\nabla w|^2 \dd y + C_{\epsilon_0,m}
  \int (\delta+|y|^2)^{m-2} |\nabla W|^2\dd y + e_{m-1,\delta}\,,
\end{align*}
where the parameter $\epsilon_0 > 0$ can be taken arbitrarily small. In
the last line, we used again the obvious inequality $(\delta+|y|^2)^{m-2}|y|^2
\le (\delta+|y|^2)^{m-1}$. Assuming that $C_3 \epsilon_0 \le \lambda_1/4$,
we thus obtain
\begin{equation}\label{interpol}
  C_3\,e_{m-1,\delta} \,\le\, \frac{\lambda_1}{4} \int (\delta+|y|^2)^m
  |\nabla w|^2 \dd y + C_4  \int (\delta+|y|^2)^{m-2} |\nabla W|^2
  \dd y\,,
\end{equation}
for some positive constant $C_4$.

We now choose a constant $\kappa > 0$ large enough so that $\kappa\lambda_1
\ge 2 C_4$, and we consider the combined energy functional
\begin{equation}\label{cEdef}
  \cE_{m,\delta}(\tau) \,=\, e_{m,\delta}(\tau) + \kappa E_{m-2,\delta}(\tau)\,,
  \qquad \tau \,\ge\, 0\,.
\end{equation}
By Proposition~\ref{Kernelprop1}, we have $e_{m,\delta}(\tau) \le
\cE_{m,\delta}(\tau) \le C_5\,e_{m,\delta}(\tau)$ for some $C_5 > 0$. Moreover,
it follows from \eqref{Eem} and from our choice of $\kappa$ that
$\cE_{m,\delta}(\tau)$ satisfies the differential inequality
\begin{equation}\label{cE1}
  \partial_\tau \cE_{m,\delta}(\tau) \,\le\, -\frac{\lambda_1}{4} \int
  (\delta+|y|^2)^m |\nabla w|^2 \dd y \,-\, \lambda
  \cE_{m,\delta}(\tau) + \kappa \cF_1(\tau) + C_2 \cF_2(\tau)\,,
\end{equation}
where
\[
  \cF_1(\tau) \,=\, \int (\delta+|y|^2)^{m-2}\,W R_1 \dd y\,, \qquad
  \cF_2(\tau) \,=\, \alpha^2\int (\delta+|y|^2)^m \|B(ye^{\tau/2})\|^2\,
  |\nabla \phi|^2\dd y\,.
\]

Our final task is to estimate the remainder terms $\cF_1, \cF_2$ in
\eqref{cE1}, which involve the matrix $B(x) = A(x) - A_\infty(x)$,
either explicitly or through the definition \eqref{R1def} of $R_1$. We
recall that $B$ satisfies the bound \eqref{Bestimate} for some $\nu >
0$. We start with the term $\cF_1$, which can be bounded using Young's
inequality and Proposition~\ref{Kernelprop2}. For $\epsilon > 0$
arbitrarily small, we thus obtain
\begin{align*}
  \cF_1(\tau) \,&\le\, \epsilon E_{m-2,\delta}(\tau) + 
  C_\epsilon \int (\delta+|y|^2)^{m-2} |R_1(y,\tau)|^2 \dd y\\
  \,&\le\, \epsilon E_{m-2,\delta}(\tau) + C_\epsilon \int |y|^{2m-2}
  \|B(ye^{\tau/2})\|^2 \bigl(\alpha^2|\nabla \phi|^2 + |\nabla w|^2\bigr)\dd y\,,
\end{align*}
where in the second line we used the fact that $(\delta + |y|^2)^{m-2}\le
|y|^{2m-4}$ because $m \le 2$, and we applied estimate \eqref{uweight2}
with $u = R_1$ and $g = B(\cdot\,e^{\tau/2})(\alpha \nabla\phi + \nabla w)$. 
To bound the last integral, we take $\gamma =\min(\nu,m-1) > 0$ and we
observe that
\[
  |y|^{2m-2}\|B(ye^{\tau/2})\|^2 \,\le\, C |y|^{2\gamma}\|B(ye^{\tau/2})\|^2
  (\delta+|y|^2)^m \,\le\, C\,e^{-\gamma \tau} (\delta+|y|^2)^m\,,
\]
because $\sup_{x\in \R^n}|x|^\gamma \|B(x)\| < \infty$. Using in
addition Proposition~\ref{gradphi}, we arrive at
\begin{equation}\label{cF1}
  \cF_1(\tau) \,\le\, \epsilon E_{m-2,\delta}(\tau) + C_\epsilon
  \,e^{-\gamma\tau}\Bigl(\alpha^2 + \int (\delta+|y|^2)^m |\nabla w|^2\dd y
  \Bigr)\,.
\end{equation}

To control $\cF_2$, we use the bound $(\delta+|y|^2)^m \le 2^{m-1}(\delta^m +
|y|^{2m})$, and we treat the term involving $|y|^{2m}$ exactly as before.
When no power of $|y|$ is available, this argument does not work,
but taking $0 < \epsilon < \gamma$ we can apply H\"older's inequality
with conjugate exponents
\[
  q \,=\, \frac{n}{2(\gamma{-}\epsilon)}\,, \qquad p \,=\,
  \frac{n}{n-2(\gamma{-}\epsilon)}\,, \qquad \hbox{so that}\quad
  1 \,<\, p \,<\, \frac{n}{2(1-\beta)}\,.
\]
We know that $\nabla \phi \in L^{2p}(\R^n)$ by Proposition~\ref{gradphi}, and
that $x \mapsto B(x) \in L^{2q}(\R^n)$ in view of \eqref{Bestimate} because
$2q = n/(\gamma{-}\epsilon) > n/\nu$. It follows that
\[
  \int \|B(ye^{\tau/2})\|^2|\nabla \phi|^2 \dd y \le
  \left(\int \|B(ye^{\tau/2})\|^{2q}\dd y\right)^{\!1/q}\! 
  \left(\int| \nabla\phi|^{2p} \dd y\right)^{\!1/p} \le
  C_\epsilon \,e^{-(\gamma-\epsilon) \tau} \|\nabla \phi\|_{L^{2p}}^2\,,
\]
hence
\begin{equation}\label{cF2}
  \cF_2(\tau) \,\le\, C_\epsilon \,\alpha^2 \,e^{-(\gamma-\epsilon)\tau}
  \Bigl(\|\nabla \phi\|_{L^{2p}}^2 + \|(1+|y|)^m \nabla \phi\|_{L^2}^2
  \Bigr)\,.
\end{equation}

To summarize, it follows from \eqref{cE1}, \eqref{cF1}, \eqref{cF2} that
\begin{equation}\label{cE2}
  \partial_\tau \cE_{m,\delta}(\tau) \,\le\, -(\lambda-\epsilon) \cE_{m,\delta}(\tau)
  + \Bigl(C_\epsilon\,e^{-\gamma\tau} - \lambda_1\Bigr)\cD_{m,\delta}(\tau) + 
  C_\epsilon' \,\alpha^2 \,e^{-(\gamma-\epsilon) \tau}\,, \quad \tau > 0\,,
\end{equation}
where
\begin{equation}\label{cDdef}
  \cD_{m,\delta}(\tau) \,=\, \frac14 \int (\delta+|y|^2)^m |\nabla w(y,\tau)|^2
  \dd y\,.
\end{equation}
Here the parameter $\epsilon>0$ can be taken arbitrarily small, and
the constants $C_\epsilon, C_\epsilon' > 0$ depend only on $\epsilon$
and on the properties of the matrix $A$. If $\tau > 0$ is large enough,
the coefficient of $\cD_{m,\delta}(\tau)$ in the right-hand side of
\eqref{cE2} becomes negative, and we obtain a differential inequality for
the combined energy which implies that $\cE_{m,\delta}(\tau)$ decays
exponentially as $\tau \to +\infty$. More precisely, using
inequalities \eqref{em3} and \eqref{cE2}, we obtain\:

\begin{prop}\label{propdecay}
Assume that $n = 2$ or $3$, $m \in (n/2,n/2+\beta)$, and $m \le 2$. For any
real number $\mu$ satisfying \eqref{mudef1}, there exists a positive
constant $C$ such that, for any $\alpha \in \R$ and any initial data
$w_0 \in L^2_0(m)$, the solution $w \in C^0([0,+\infty),L^2_0(m))$
of \eqref{weqlin} satisfies 
\begin{equation}\label{alltimes}
  \|w(\tau)\|_{L^2(m)} \,\le\, C \bigl(\|w_0\|_{L^2(m)} +
  |\alpha|\bigr)\,e^{-\mu \tau}\,, \qquad \tau \ge 0\,.
\end{equation}
\end{prop}

\begin{proof}
Given $\mu$ satisfying \eqref{mudef1}, we choose $\epsilon > 0$
small enough so that $2\mu < \min(\lambda,\gamma) - \epsilon$, where
(as above) $\lambda = m-n/2$ and $\gamma = \min(\nu,m{-}1)$. 
If $\tau_* > 0$ is large enough so that $\lambda_1\,e^{\gamma \tau_*} \ge
C_\epsilon$, we can omit the term involving $\cD_{m,\delta}(\tau)$ in
the right-hand side of \eqref{cE2}, and integrating the resulting differential
inequality we obtain $\cE_{m,\delta}(\tau) \le C\bigl(\cE_{m,\delta}(\tau_*) +
\alpha^2\bigr)e^{-2\mu(\tau-\tau_*)}$ for $\tau \ge \tau_*$. Since the
combined energy $\cE_{m,\delta}(\tau)$ is equivalent to $\|w(\tau)\|_{L^2(m)}^2$,
this gives the large time estimate
\begin{equation}\label{largetime}
  \|w(\tau)\|_{L^2(m)}^2 \,\le\, C \bigl(\|w(\tau_*)\|_{L^2(m)}^2 +
  \alpha^2\bigr)\,e^{-2\mu (\tau-\tau_*)}\,, \qquad \tau \ge \tau_*\,.
\end{equation}

To control the solution for intermediate times, we use the differential 
inequality \eqref{em3} with $\delta = 1$, which is in fact valid regardless
of the value of the parameter $m$. If we bound the last term in the right-hand
side using \eqref{cF2}, we obtain the useful inequality
\begin{equation}\label{em5}
  \partial_\tau \|w(\tau)\|_{L^2(m')}^2 \,\le\,  \frac{n{-}2m'}{2}\,
  \|w(\tau)\|_{L^2(m')}^2   + \bigl(m'\delta + C_1 m'^2\bigr)
  \|w(\tau)\|_{L^2(m'-1)}^2 + C_2\,\alpha^2 \,e^{-\gamma'\tau}\,,
\end{equation}
which holds for any $m'\ge 0$ and any $\gamma' \in [0,m']$ with
$\gamma' < \nu$. In particular, if $m' = 0$ and $\gamma' = 0$, we have
$\partial_\tau \|w(\tau)\|_{L^2}^2 \le (n/2)\|w(\tau)\|_{L^2}^2 +
C_2\alpha^2$, so that $\|w(\tau)\|_{L^2}^2 \le \bigl(\|w(0)\|_{L^2}^2
+ C\alpha^2\bigr)e^{n\tau/2}$ for all $\tau \ge 0$. Then, taking
successively $m' = 1$, $m' = 2$, \dots we obtain in a finite number of
steps the rough estimate
\begin{equation}\label{smalltime}
  \|w(\tau)\|_{L^2(m)}^2 \,\le\, C\bigl(\|w_0\|_{L^2(m)}^2 +  \alpha^2\bigr)
  \,e^{n\tau/2}\,, \qquad \tau \ge 0\,.
\end{equation}
Combining \eqref{smalltime} for $\tau \le \tau_*$ and \eqref{largetime}
for $\tau \ge \tau_*$ , we easily obtain \eqref{alltimes}. 
\end{proof}

\subsection{Higher-order antiderivatives}\label{ssec45}

Proposition~\ref{propdecay} is the main ingredient in the proof of
Theorem~\ref{main1} in low space dimensions. It is obtained, however,
under the (unfortunate) assumption that $m \le 2$, which implies first
that $n \le 3$, and also that the convergence rate $\mu$ cannot exceed
the value $1/4$ if $n = 3$, even if the parameters $\beta,\nu$ are
larger than $1/2$. To remove these artificial restrictions, we need to
introduce higher-order antiderivatives, as we now explain. The reader
who is satisfied with the assumptions of Proposition~\ref{propdecay}
can skip what follows and jump directly to Section~\ref{ssec46}.

We first recall that most of our analysis so far, including the
weighted estimates in Section~\ref{ssec25}, is valid in arbitrary
space dimension $n \ge 2$. In Section~\ref{sec4}, the differential
inequality \eqref{em3} for the energy functional $e_{m,\delta}(\tau)$
holds for all $n \ge 2$ and any $m \ge 0$, but is not sufficient by
itself to prove exponential decay of the solutions. This was precisely
the reason for introducing the additional functional $E_{m-2,\delta}(\tau)$, 
which involves the antiderivative $W = K[w]$. The assumption that $m \le 2$ 
is needed to eliminate the undesirable term involving $E_{m-3,\delta}(\tau)$ 
in the right-hand side of \eqref{Em1}, so as to obtain exponential decay
by combining \eqref{em3} and \eqref{Em1}. 

We now consider the situation where $m \in (n/2,n/2+\beta)$ and $2 < m
\le 4$, which is possible if $n = 3$ and $\beta>1/2$, or if $4 \le n
\le 7$. In that case, keeping in mind the conclusions of
Propositions~\ref{Kernelprop1} and \ref{Kernelprop2}, which show that
each application of the linear operator $K$ decreases by two units the
power $m$ in the weight $(\delta + |y|^2)^m$, we introduce the ``second
antiderivative'' $\WW = K[W] = K^2[w]$. We know from
Remark~\ref{rem_p} that $W \in L^2(m{-}2)$, and our current
assumptions on $m$ imply that $0 < m-2 < n/2$. Thus we can apply
Proposition~\ref{Kernelprop0} which asserts that $\WW \in L^2(m{-}4)$
with $\|\WW\|_{L^2(m-4)} \le C\|W\|_{L^2(m-2)} \le C\|w\|_{L^2(m)}$. Moreover,
proceeding as in Section~\ref{ssec43}, it is straightforward to verify
that $\WW(y,\tau)$ satisfies the evolution equation
\begin{equation}\label{WWeqlin}
  \partial_\tau \WW \,=\, \div\bigl(A_\infty(y)\nabla \WW\bigr) + \frac12
  \,y\cdot\nabla \WW + \frac{n{-}4}{2}\,\WW + K[R_1]\,,
\end{equation}
where $R_1$ is as in \eqref{R1def}. Note that the factor $(n{-}2)/2$
in \eqref{Weqlin} becomes $(n{-}4)/2$ in \eqref{WWeqlin}. 

The natural energy functional for the new variable $\WW$ is
\begin{equation}\label{EEmdef}
  \EE_{m-4,\delta}(\tau) \,=\, \frac 12 \int_{\R^n} (\delta+|y|^2)^{m-4}
  |\WW(y,\tau)|^2 \dd y\,, \qquad \tau \ge 0\,.
\end{equation}
In analogy with \eqref{Em1} we find
\begin{equation}\label{EEm}
  \begin{split}
  \partial_\tau \EE_{m-4,\delta}(\tau) \,&\le\,  -\frac12 \int (\delta+|y|^2)^{m-4}
  \Langle\nabla \WW, A_\infty \nabla \WW\Rangle\dd y \,+\, \frac{n{-}2m}{2}\,
  \EE_{m-2,\delta}(\tau)\\ 
  \,&+\,\Bigl((m{-}4)\delta + C_1(m{-}4)^2\Bigr) \EE_{m-5,\delta}(\tau)
  \,+\, \int (\delta+|y|^2)^{m-4} \,\WW K[R_1] \dd y\,.
  \end{split}
\end{equation}
As in Section~\ref{ssec44}, since $m \le 4$, the coefficient of $\EE_{m-5,\delta}$
in \eqref{EEm} can be made non-positive by an appropriate choice of
$\delta$. Moreover the negative term involving $\nabla \WW$ can be used
to control the undesirable quantity $\bigl((m{-}2)\delta + C_1(m{-}2)^2\bigr)
E_{m-3,\delta}(\tau)$ in \eqref{Em1}, in view of the interpolation inequality
\[
  E_{m-3,\delta} \,\le\, \varepsilon \int (\delta+|y|^2)^{m-2}
  |\nabla W|^2 \dd y + C_\varepsilon \int (\delta+|y|^2)^{m-4} |\nabla \WW|^2
  \dd y\,,
\]
which is established exactly as in \eqref{interpol}. Finally, the remainder term
in \eqref{EEm} can be estimated just as the quantity $\cF_1$ in \eqref{cE1}. 
Indeed, since $m-4 \le 0$, Proposition~\ref{Kernelprop0} yields
\[
  \int (\delta+|y|^2)^{m-4} \,|K[R_1]|^2 \dd y \,\le\,  \int |y|^{2(m-4)} \,|K[R_1]|^2
  \dd y \,\le\, C \int |y|^{2(m-2)} \,|R_1|^2 \dd y \,.
\]

The arguments above allow us to control the solution of \eqref{weqlin} using the
new functional
\[
  \cE^{(2)}_{m,\delta}(\tau) \,=\, e_{m,\delta}(\tau) + \kappa_1 E_{m-2,\delta}(\tau)
  + \kappa_2 \EE_{m-4,\delta}(\tau)\,, \qquad \tau \ge 0\,,
\]
where $\kappa_1, \kappa_2$ are positive constants satisfying
$\kappa_2 \gg \kappa_1 \gg 1$. Combining the differential
inequalities \eqref{em3}, \eqref{Em1}, \eqref{EEm} and proceeding as
in Section~\ref{ssec44}, it is straightforward to prove the
exponential decay of the energy $\cE^{(2)}_{m,\delta}(\tau)$ as
$\tau \to +\infty$.

In yet higher space dimensions, namely when $m \in (n/2,n/2+\beta)$
and $m > 4$, the strategy is similar but it becomes necessary to use
the iterated antiderivatives $W^{(\ell)} = K^\ell[w]$ for larger
values of $\ell \in \N$. To give a flavor, let $\ell$ be the smallest
integer such that $m - 2\ell \le 0$. The energy functional
$E_{m-2\ell,\delta}^{(\ell)}(\tau)$ is defined in close analogy with
\eqref{EEmdef}, and satisfies a differential inequality similar to
\eqref{EEm} where the coefficient $(m{-}2\ell)\delta + C_1(m{-}2\ell)^2$ 
in front of $E_{m-2\ell-1,\delta}^{(\ell)}$ is either zero or can be
made negative by an appropriate choice of $\delta$. Moreover, the
negative term involving $|\nabla W^{(\ell)}|^2$ can be used to control
an undesirable quantity in the evolution equation for the next functional 
in the hierarchy, which is $E_{m-2(\ell-1),\delta}^{(\ell-1)}$. Exponential
decay can thus be established using a combined functional of the 
form
\[
  \cE^{(\ell)}_{m,\delta}(\tau) \,=\, e_{m,\delta}(\tau) + \kappa_1 E_{m-2,\delta}(\tau)
  + \kappa_2 E_{m-4,\delta}^{(2)}(\tau) + \dots + \kappa_\ell 
  E_{m-2\ell,\delta}^{(\ell)}(\tau)\,,
\]
for some suitable constants $\kappa_1, \dots, \kappa_\ell$. The details are 
left to the reader.

Taking the above arguments for granted, we thus obtain\:

\begin{cor}\label{propdecaybis}
The conclusion of Proposition~\ref{propdecay} holds for all $n \ge 2$
if $m \in (n/2,n/2+\beta)$. 
\end{cor}

\subsection{End of the proof of Theorem~\ref{main1}}\label{ssec46}

We conclude here the proof of Theorem~\ref{main1} assuming the
validity of Corollary~\ref{propdecaybis}, which was carefully 
established at least in low dimensions, see Proposition~\ref{propdecay}. 
What remains to be done is essentially to remove the upper bound
$n/2+\beta$ on the parameter $m$. This will not increase the
convergence rate $\mu$, as can be seen from \eqref{mudef1}, but
estimate \eqref{mainconv1} will nevertheless be obtained in a stronger
norm.  To do that, our strategy is to introduce an auxiliary parameter
$\bar m \le m$ such that $\bar m \in (n/2,n/2+\beta)$. Estimate 
\eqref{alltimes} allows us to control the solution in the larger space
$L^2(\bar m)$, and a simple interpolation gives convergence in 
$L^2(m)$ too.

We now provide the details. Assume that $n \ge 2$ and take initial
data $v_0 \in L^2(m)$ for some $m > n/2$. We decompose
$v_0 = \alpha \phi + w_0$, where $\alpha = \int v_0(y)\dd y$, and we
consider the unique solution $w \in C^0([0,+\infty), L^2_0(m))$ of
equation \eqref{weqlin} such that $w(0) = w_0$. Given $\mu$ satisfying
\eqref{mudef1}, we choose $\bar m \le m$ such that $\bar m \in 
(n/2+2\mu,n/2+\beta)$.  We start from estimate \eqref{em5} with
$m' = m$ and $\gamma' = 2\mu$, which gives
\[
  \partial_\tau \|w(\tau)\|_{L^2(m)}^2 \,\le\, \frac{n{-}2m}{2}\,
  \|w(\tau)\|_{L^2(m)}^2 + \bigl(m\delta + C_1 m^2\bigr) \|w(\tau)\|_{L^2(m-1)}^2
  + C_2\,\alpha^2 \,e^{-2\mu\tau}\,.
\]
We next use the elementary bound
\[
  \|w(\tau)\|_{L^2(m-1)}^2 \,\le\, \epsilon \|w(\tau)\|_{L^2(m)}^2 + 
  C_\epsilon \|w(\tau)\|_{L^2(\bar m)}^2\,,
\]
which is obtained by interpolation if $\bar m < m-1 <m$, and is completely
obvious if $m-1\le \bar m \le m$. Taking any $\lambda$ such that $2\mu < 
\lambda < (n{-}2m)/2$ and choosing $\epsilon > 0$ small enough, we thus obtain
\[
  \partial_\tau \|w(\tau)\|_{L^2(m)}^2 \,\le\,  -\lambda \,\|w(\tau)\|_{L^2(m)}^2
   + C'_\epsilon \|w(\tau)\|_{L^2(\bar m)}^2 + C_2\,\alpha^2 \,e^{-2\mu\tau}\,.
\]
The second term in the right-hand side is controlled using estimate
\eqref{alltimes} in the space $L^2(\bar m)$, and taking into account the 
fact that $\bar m \in (n/2+2\mu,n/2+\beta)$. This gives
\[
  \partial_\tau \|w(\tau)\|_{L^2(m)}^2 \,\le\, -\lambda\,\|w(\tau)\|_{L^2(m)}^2 +
  C''_\epsilon  \big(\|w_0\|_{L^2(\bar m)}^2 + \alpha^2\big)e^{- 2\mu \tau}
  + C_2\,\alpha^2 \,e^{-2\mu\tau}\,.
\]
As $\|w_0\|_{L^2(\bar m)} \le \|w_0\|_{L^2(m)}$ and $\lambda > 2\mu$, a final application
of Gr\"onwall's lemma gives the desired estimate
\[
  \|w(\tau)\|_{L^2(m)} \,\le\, C \bigl(\|w_0\|_{L^2(m)} +
  |\alpha|\bigr)\,e^{-\mu \tau}\,, \qquad \tau \ge 0\,, 
\]
where the constant $C$ depends on $n$, $m$, $\mu$, and on the properties
of the matrix $A$. \QED

%%%%%%%%%%%%%%%%%%%%%%%%%%%%%%%%%%%%%%%%%%%%%%%%%%%%%%%%%%%%%%%%%%%%%

\section{Long-time asymptotics in the semilinear case}\label{sec5}

In this section we study the long-time behavior of small solutions to
the full equation \eqref{veq}, where the nonlinearity $\cN(\tau,v)$ is
given by \eqref{cNdef}. As before, we concentrate on the low space
dimensions $n = 2$ and $n = 3$, but using the ideas introduced in 
Section~\ref{ssec45} it is possible to treat the higher-dimensional
case as well. We recall that the function $N$ in \eqref{cNdef}
satisfies \eqref{Nprop}, and we suppose without loss of generality
that the exponent $\sigma$ in \eqref{Nprop} lies in the range
\begin{equation}\label{sigmarange}
 1 + \frac{2}{n} \,<\, \sigma \,\le\, 1 + \frac{3}{n}\,.  
\end{equation}
This means that the quantity $\eta$ defined in \eqref{mudef2}
satisfies $0 < \eta \le 1/2$. Clearly, a larger value of $\sigma$,
hence of $\eta$, would not increase the convergence exponent $\mu$
in \eqref{mudef2}, since $\beta < 1$. 

In view of \eqref{cNdef}, \eqref{Nprop}, the nonlinearity $\cN$ in
\eqref{veq} satisfies
\begin{equation}\label{cNprop}
  \bigl|\cN(\tau,v)\bigr| \,\le\, C_0\,e^{-\eta \tau}|v|^\sigma\,,
  \qquad \hbox{and}\qquad \bigl|\cN(v) - \cN(\tilde v)\bigr|
  \,\le\, C_0\,e^\tau |v-\tilde v|\,,
\end{equation}
for all $v, \tilde v \in \R$ and all $\tau \ge 0$, where $C_0$ is some
positive constant. In particular, since $\cN(\tau,v)$ is a globally
Lipschitz function of $v$, uniformly in $\tau$ on compact intervals,
it is straightforward to verify, as in Lemma~\ref{wellposed}, that the
Cauchy problem for Eq.~\eqref{veq} is globally well-posed in the space
$L^2(m)$ for any $m \ge 0$. In other words, given any initial data
$v_0 \in L^2(m)$, there exists a unique global solution $v \in
C^0([0,+\infty),L^2(m))$ of \eqref{veq} such that $v(0) = v_0$. Our
goal here is to compute the long-time behavior of that solution when
the initial data are sufficiently small.

We assume henceforth that $m > n/2$, so that $L^2(m) \hookrightarrow
L^1(\R^n)$. We decompose the solution as in \eqref{vdecomp}, with
the important difference that the integral of $v$ is no longer
a conserved quantity. Instead we have
\begin{equation}\label{alphadot}
  \alpha(\tau) \,=\, \int_{\R^n} v(y,\tau)\dd y\,, \qquad \hbox{and}
  \qquad \alpha'(\tau) \,=\, \int_{\R^n} \cN\bigl(\tau,v(y,\tau)\bigr)
  \dd y\,.
\end{equation}
The equation satisfied by the perturbation $w(y,\tau) = v(y,\tau) -
\alpha(\tau)\phi(y)$ is of the form \eqref{weqlin}, except that 
the remainder term $r_1$ given by \eqref{r1def} is replaced by 
$r_1+ r_2$, where
\begin{equation}\label{r2def}
  r_2(y,\tau) \,=\, \cN\bigl(\tau,\alpha(\tau)\phi(y) + w(y,\tau)\bigr) -
  \alpha'(\tau)\phi(y)\,.
\end{equation}
Similarly, the antiderivative $W(y,\tau)$ defined by \eqref{Welliptic}
satisfies equation \eqref{Weqlin}, except that the remainder term $R_1$
defined by \eqref{R1def} is replaced by $R_1 + R_2$, where $R_2 = K[r_2]$. 

As in the previous section, our strategy is to control the solution
of \eqref{weqlin} or \eqref{Weqlin} using weighted energy estimates,
where the weight is a power of $\rho(y) := (\delta + |y|^2)^{1/2}$.
To treat the nonlinear terms, the following auxiliary results
will be useful.

\begin{lem}\label{Nlem1}
If $w \in L^2(m)$ and $\nabla w \in L^2(m)^n$, we have, for all 
$\tau \ge 0$,
\begin{equation}\label{lem1ineq}
  \int_{\R^n} \!\rho^{2m} |w|\,\bigl|\cN(\tau,\alpha\phi{+}w)\bigr|\dd y
  \,\le\, C\,e^{-\eta\tau}\Bigl(|\alpha|^\sigma \|w\|_{L^2(m)} +
  \|w\|_{L^2(m)}^{\sigma+1} + \|\nabla w\|_{L^2(m)}^{\eta + 1}\,
  \|w\|_{L^2(m)}^{\sigma - \eta}\Bigr),
\end{equation}
where $\eta > 0$ is as in \eqref{mudef2}. 
\end{lem}

\begin{proof}
In view of \eqref{cNprop}, we have $|\cN(\tau,\alpha\phi{+}w)|
\le C\,e^{-\eta\tau} \bigl(|\alpha|^\sigma \phi^\sigma + |w|^\sigma\bigr)$,
hence
\[
  \int_{\R^n} \rho^{2m} |w|\,\bigl|\cN(\tau,\alpha\phi{+}w)\bigr|\dd y
  \,\le\, C\,e^{-\eta\tau}\biggl(|\alpha|^\sigma \|w\|_{L^2(m)} + 
  \int_{\R^n} \rho^{2m}|w|^{\sigma+1}\dd y\biggr)\,,
\]
where we used the Cauchy-Schwarz inequality and the fact that
$\phi^\sigma \in L^2(m)$, see \eqref{phibounds}. To bound the
last integral, we observe that $\rho^{2m}|w|^{\sigma+1} \le
|\omega|^{\sigma+1}$ where $\omega = \rho^m w$, and we use the
interpolation inequality
\begin{equation}\label{interpsigma}
  \int_{\R^n} |\omega|^{\sigma+1}\dd y \,\le\, C\,
  \|\nabla \omega\|_{L^2(\R^n)}^{\frac{n}{2}(\sigma-1)}\,
  \|\omega\|_{L^2(\R^n)}^{\sigma + 1 - \frac{n}{2}(\sigma-1)}\,,
\end{equation}
which is valid because $(\sigma+1)(n-2) \le 2n$. Since $\|\nabla
\omega\|_{L^2(\R^n)} \le C\bigl(\|\nabla w\|_{L^2(m)} + \|w\|_{L^2(m)}\bigr)$
and $(n/2)(\sigma-1) = 1 + \eta$, we obtain \eqref{lem1ineq}. 
\end{proof}

\begin{rem}\label{Nrem1}
We can simplify somehow estimate \eqref{lem1ineq} by applying 
Young's inequality to the various terms in the right-hand side.
The appropriate pairs of conjugate exponents are $p = p' = 2$ for
the first two terms, and $q = 2/(1{+}\eta)$, $q' = 2/(1{-}\eta)$
for the last one. We observe that $q' > 2$ and $q'(\sigma - \eta)
> 2\sigma$, hence assuming that $\|w\|_{L^2(m)} \le 1$  we obtain, for 
any $\epsilon > 0$, 
\[
  \int_{\R^n} \rho^{2m} |w|\,\bigl|\cN(\tau,\alpha\phi{+}w)\bigr|\dd y
  \,\le\, \epsilon \Bigl(\|w\|_{L^2(m)}^2 + \|\nabla
  w\|_{L^2(m)}^2\Bigr) +  C_\epsilon\,e^{-2\eta\tau}\Bigl(|\alpha|^{2\sigma} +
  \|w\|_{L^2(m)}^{2\sigma} \Bigr)\,.
\]
\end{rem}

\begin{lem}\label{Nlem2}
If $w \in L^2(m)$ and $\nabla w \in L^2(m)^n$, we have, for all 
$\tau \ge 0$, 
\begin{equation}\label{lem2ineq}
  \int_{\R^n} \rho^{m-n/2} \bigl|\cN(\tau,\alpha\phi{+}w)\bigr|\dd y
  \,\le\, C\,e^{-\eta\tau}\Bigl(|\alpha|^\sigma + \|w\|_{L^2(m)}^\sigma
  + \zeta_n \|\nabla w\|_{L^2(m)}^{\sigma-2} \,\|w\|_{L^2(m)}^2\Bigr)\,,
\end{equation}
where $\zeta_n = 0$ if $n \ge 3$ and $\zeta_n = 1$ if $n = 2$. 
\end{lem}

\begin{proof}
In view of \eqref{cNprop} and \eqref{phibounds}, we have as before
\[
  \int_{\R^n} \rho^{m-n/2} \bigl|\cN(\tau,\alpha\phi{+}w)\bigr|\dd y
  \,\le\, C\,e^{-\eta\tau}\biggl(|\alpha|^\sigma + 
  \int_{\R^n} \rho^{m-n/2}|w|^\sigma \dd y\biggr)\,.
\]
If $n \ge 3$, then $\sigma \in (1,2]$ by \eqref{sigmarange}, and
a simple application of H\"older's inequality yields
\[
  \int_{\R^n} \rho^{m-n/2}|w|^\sigma \dd y \,\le\, 
  \|\rho^{-m(\sigma-1)-n/2}\|_{L^{2/(2-\sigma)}(\R^n)} \|w\|_{L^2(m)}^\sigma 
  \,\le\, C \|w\|_{L^2(m)}^\sigma\,.
  \]
We thus obtain \eqref{lem2ineq} with $\zeta_n = 0$, for any $w \in
L^2(m)$. If $n = 2$, then $\sigma > 2$ by \eqref{sigmarange}, and we
can control the term involving $|w|^\sigma$ as in the proof of
Lemma~\ref{Nlem1}. Setting $\omega = \rho^m w$ and using the
interpolation inequality \eqref{interpsigma} with $\sigma$ replaced by
$\sigma-1$, we arrive at \eqref{lem2ineq} with $\zeta_n = 1$.
\end{proof}

Our main goal is to prove that the quantities $|\alpha'(\tau)|$ and
$\|w(\cdot,\tau)\|_{L^2(m)}$ decay exponentially to zero as $\tau \to
+\infty$, if we assume a priori that $|\alpha(\tau)| +
\|w(\cdot,\tau)\|_{L^2(m)} \le 1$ for all $\tau \ge 0$. As we shall
see, that condition will be fulfilled if we take sufficiently small
initial data. Proceeding as in Section~\ref{sec4}, our strategy is to
use the differential inequalities satisfied by the energy functionals
$e_{m,\delta}(\tau)$ and $\cE_{m-2,\delta}(\tau)$ defined by
\eqref{emdef}, \eqref{cEdef}, respectively. In what follows, we fix
some $\delta \ge 1$ and we denote $\rho(y) = (\delta+|y|^2)^{1/2}$.

We first control the evolution of the scalar quantity $\alpha$.
The derivative $\alpha'(\tau)$ given by \eqref{alphadot} can
be estimated using Lemma~\ref{Nlem2}, if we disregard the factor
$\rho^{m-n/2} \ge 1$ in the left-hand side of \eqref{lem2ineq}. 
We thus find
\begin{equation}\label{alphadotest}
  |\alpha'(\tau)| \,\le\, C\,e^{-\eta\tau}\Bigl(|\alpha|^\sigma +
  \|w\|_{L^2(m)}^\sigma + \zeta_n \|\nabla w\|_{L^2(m)}^{\sigma-2}
  \,\|w\|_{L^2(m)}^2\Bigr)\,.
\end{equation}
Next, since the function $w$ satisfies \eqref{weqlin} with remainder
term $r_1 + r_2$, we obtain as in \eqref{Eem}\:
\begin{equation}\label{NLe1}
  \partial_\tau e_{m,\delta}(\tau) \,\le\, - 2\lambda_1 \cD_{m,\delta}(\tau)
  - \lambda e_{m,\delta}(\tau) + C_3\,e_{m-1,\delta}(\tau)
  + C_2 \cF_2(\tau) + \cF_3(\tau)\,,
\end{equation}
where $\lambda = m - n/2$, $C_3 = m\delta +C_1 m^2$, $\cD_{m,\delta}(\tau)$
is defined in \eqref{cDdef}, and $\cF_3(\tau) = \int \rho^{2m} w\,r_2
\dd y$. In view of \eqref{cF2}, we have $\cF_2(\tau) \le C_\epsilon\,\alpha^2
\,e^{-(\gamma-\epsilon)\tau}$ for any small $\epsilon > 0$, where
$\gamma = \min(\nu,m-1)$. Moreover, the definition \eqref{r2def}
of $r_2$ implies that
\[
  \cF_3(\tau) \,\le\, \int \rho^{2m}\,|w|\,\bigl|\cN(\tau,\alpha\phi + w)
  \bigr| \dd y \,+\, |\alpha'(\tau)| \int \rho^{2m}\,|w|\,\phi\dd y\,.
\]
The first term in the right-hand side is estimated using Lemma~\ref{Nlem1}
and Remark~\ref{Nrem1}, whereas for the second term we use \eqref{alphadotest},
the Cauchy-Schwarz inequality, and Young's inequality. We thus find
\begin{equation}\label{cF3}
  \cF_3(\tau) \,\le\, \epsilon \Bigl(\|w\|_{L^2(m)}^2 + \|\nabla w\|_{L^2(m)}^2
  \Bigr) + C_\epsilon\,e^{-2\eta\tau}\Bigl(|\alpha|^{2\sigma} +
  \|w\|_{L^2(m)}^{2\sigma} \Bigr)\,,
\end{equation}
where $\epsilon > 0$ is arbitrarily small. Replacing these estimates
into \eqref{NLe1}, we arrive at
\begin{equation}\label{NLe2}
\begin{split}
  \partial_\tau e_{m,\delta}(\tau) \,\le\, &- (2\lambda_1{-}\epsilon)
  \cD_{m,\delta}(\tau) - (\lambda{-}\epsilon) e_{m,\delta}(\tau) +
  C_3\,e_{m-1,\delta}(\tau) \\ &+ C_\epsilon\,\alpha^2
  \,e^{-(\gamma-\epsilon)\tau} + C_\epsilon\,e^{-2\eta\tau}\bigl(|\alpha|^{2\sigma} +
  e_{m,\delta}(\tau)^\sigma \bigr)\,,
\end{split}
\end{equation}
for some sufficiently small $\epsilon > 0$.

The second important quantity we want to control is the combined
energy functional \eqref{cEdef}, which involves both $w$ and the
antiderivative $W$. At this point, we have to assume as in
Proposition~\ref{propdecay} that $m \in (n/2,n/2+\beta)$ and
$m \le 2$. We also suppose that $\delta$ satisfies \eqref{lambdadef}.
Due to the additional nonlinear terms in the evolution equations for
$w$ and $W$, we obtain instead of \eqref{cE1}\:
\begin{equation}\label{NLe3}
  \partial_\tau \cE_{m,\delta}(\tau) \,\le\, -\lambda_1 \cD_{m,\delta}(\tau)
  - \lambda\cE_{m,\delta}(\tau) + \kappa\bigl(\cF_1(\tau) + \cF_4(\tau)\bigr)
  + C_2 \cF_2(\tau) + \cF_3(\tau)\,,
\end{equation}
where $\cF_1(\tau)$ satisfies \eqref{cF1} and $\cF_4(\tau) = \int
\rho^{2m-4} W R_2 \dd y = \int \rho^{2m-4} W K[r_2] \dd y$. To
estimate the new term, we first apply the Cauchy-Schwarz inequality,
and then Proposition~\ref{Kernelprop1} with $p = 1$ and $s = n/2$.
We thus obtain
\[
  |\cF_4(\tau)| \,\le\, C \|W\|_{L^2(m{-}2)}
  \int_{\R^n} \rho^{m-n/2}\Bigl|\cN \bigl(\tau,\alpha\phi
  + w\bigr) - \alpha'(\tau)\phi\Bigr| \dd y\,,
\]
where the integral in the right-hand side can be controlled using
Lemma~\ref{Nlem2} and estimate \eqref{alphadotest}. Using in
addition Young's inequality when $n = 2$ (in which case $\zeta_n = 1$),
we obtain
\begin{equation}\label{cF4}
  \cF_4(\tau) \,\le\, \epsilon\Bigl(\|W\|_{L^2(m{-}2)}^2 +
  \zeta_n \|\nabla w\|_{L^2(m)}^2\Bigr) + C_\epsilon\,e^{-2\eta\tau}\Bigl(
  |\alpha|^{2\sigma} + \|w\|_{L^2(m)}^{2\sigma} \Bigr)\,.
\end{equation}
If we bound $\cF_1(\tau)$ by \eqref{cF1}, $\cF_4(\tau)$ by \eqref{cF4},
and $\cF_2(\tau), \cF_3(\tau)$ as in \eqref{NLe2}, we can write
\eqref{NLe3} in the form
\begin{equation}\label{NLe4}
\begin{split}
  \partial_\tau \cE_{m,\delta}(\tau) \,\le\, &-(\lambda-\epsilon)
  \cE_{m,\delta}(\tau) + \Bigl(C_\epsilon'\,e^{-\gamma\tau} - \lambda_1\Bigr)
  \cD_{m,\delta}(\tau) \\ &+  C_\epsilon'' \,\alpha^2
  \,e^{-(\gamma-\epsilon) \tau} + C_\epsilon''\,e^{-2\eta\tau}\Bigl(
  |\alpha|^{2\sigma} + e_{m,\delta}(\tau)^\sigma \Bigr)\,,
\end{split}
\end{equation}
where $\epsilon > 0$ is small enough. Both inequalities \eqref{NLe2},
\eqref{NLe4} are valid as long as $\|w(\tau)\|_{L^2(m)} \le 1$,
and the constants $C_\epsilon, C_\epsilon', C_\epsilon''$ therein 
depend only on $\epsilon$ and on the properties of the matrix $A$.

\medskip\noindent{\bf End of the proof of Theorem~\ref{main2}}.  We
now show how to deduce the conclusion of Theorem~\ref{main2} from
estimates \eqref{alphadotest}, \eqref{NLe2}, and \eqref{NLe4}, assuming
for simplicity that either $n = 2$ or $n = 3$ and $\mu < 1/4$. The
arguments here are pretty standard, and we only indicate the main
steps. Throughout the proof, we assume that $v$ is the solution of
\eqref{veq} with initial data $v_0 \in L^2(m)$ satisfying
$\|v_0\|_{L^2(m)} \le \epsilon_0$, for some sufficiently small
$\epsilon_0 > 0$. We decompose this solution as $v(y,\tau) = \alpha(\tau) 
\phi(y) + w(y,\tau)$ where $\alpha(\tau)$ is defined by \eqref{alphadot}.

\smallskip\noindent{\bf Step 1.} ({\em Short-time estimate}) We claim
that there exist constants $C_9 > 1$ and $\theta > n/2$ such that
\begin{equation}\label{short1}
  \alpha(\tau)^2 + e_{m,\delta}(\tau) \,\le\, C_9\,e^{\theta\tau}
  \Bigl(\alpha(0)^2 + e_{m,\delta}(0)\Bigr)\,, \qquad \tau \ge 0\,,
\end{equation}
as long as the right-hand side is smaller than or equal to $1$.
To prove \eqref{short1}, we start from the differential inequality
\eqref{NLe2}, which is valid for any $m > n/2$. Using the rough
estimate $e_{m-1,\delta}(\tau) \le e_{m,\delta}(\tau)$ and assuming
that $\alpha(\tau)^2 + e_{m,\delta}(\tau) \le 1$, we deduce from
\eqref{NLe2} that 
\begin{equation}\label{short2}
  \partial_\tau e_{m,\delta}(\tau) \,\le\, -c\cD_{m,\delta}(\tau)
  \,+\, \theta e_{m,\delta}(\tau) + C\,e^{-2\mu \tau} \Bigl(\alpha(\tau)^2
  \,+\, e_{m,\delta}(\tau)\Bigr)\,,
\end{equation}
where $c = 2\lambda_1 - \epsilon$ and $\theta = C_3 - \lambda + \epsilon$.  
Under the same assumptions, it follows from \eqref{alphadotest} and 
Young's inequality that
\begin{equation}\label{short3}
  2 \alpha(\tau)\alpha'(\tau) \,\le\, \epsilon \Bigl(\alpha(\tau)^2 +
  \zeta_n D_{m,\delta}(\tau)\Bigr) + C_\epsilon\,e^{-2\eta\tau}
  \Bigl(\alpha(\tau)^2 \,+\, e_{m,\delta}(\tau)\Bigr)\,,
\end{equation}
where $\epsilon > 0$ is arbitrarily small. Combining \eqref{short2},
\eqref{short3} we obtain a differential inequality for the
quadratic quantity $\alpha(\tau)^2 + e_{m,\delta}(\tau)$, which can
be integrated to give \eqref{short1}. 

\medskip\noindent{\bf Step 2.} ({\em Exponential decay for large times})
We assume for the time being that $m \le 2$ and $m \in (n/2 + 2\mu,n/2+\beta)$,
so that estimate \eqref{NLe4} is valid. We take $\tau_1 > 0$ large
enough so that $C_\epsilon'\, e^{-\gamma\tau_1} \le \lambda_1/2$, where
$C_\epsilon'$ is as in \eqref{NLe4}, and we assume that $\epsilon_1^2
:= \alpha(\tau_1)^2 + e_{m,\delta}(\tau_1) \ll 1$. In view of
\eqref{short1}, this condition is fulfilled if the initial
data are sufficiently small. For $\tau \ge \tau_1$, as long as
the solution satisfies $\alpha(\tau)^2 + e_{m,\delta}(\tau) \le
M^2 \epsilon_1^2 \le 1$, for some fixed constant $M > 1$, we can
integrate the differential inequality \eqref{NLe4} to obtain
\begin{equation}\label{long1}
  \cE_{m,\delta}(\tau) + \frac{\lambda_1}{2}\int_{\tau_1}^\tau e^{-\lambda'(\tau-s)}
  \cD_{m,\delta}(s)\dd s \,\le\, e^{-\lambda'(\tau-\tau_1)}\,\cE_{m,\delta}(\tau_1)
  + C M^2 \epsilon_1^2 \,e^{-2\mu\tau}\,,
\end{equation}
where $\lambda' = \lambda-\epsilon$. Under the same assumptions, 
integrating \eqref{alphadotest}, we obtain for $\tau_1 \le \tau_2 \le \tau$\:
\begin{equation}\label{long2}
  \bigl|\alpha(\tau) - \alpha(\tau_2)\bigr| \,\le\, \int_{\tau_2}^\tau
  \bigl|\alpha'(s)\bigr|\dd s \,\le\, C M^\sigma \epsilon_1^\sigma
  \,e^{-\eta\tau_2}\,.
\end{equation}
Estimate \eqref{long2} is straightforward to obtain when $n \ge 3$, but
in the two-dimensional case we must use the integral term in the
right-hand side of \eqref{long1} to control the quantity involving
$\|\nabla w\|_{L^2(m)}$ in the expression \eqref{alphadotest} of
$\alpha'(\tau)$. In any case, it follows from \eqref{long1},
\eqref{long2} that
\begin{equation}\label{long3}
  \alpha(\tau)^2 + e_{m,\delta}(\tau) \,\le\, C_{10}\Bigl(
  \epsilon_1^2 + M^2 \epsilon_1^2 \,e^{-2\mu\tau_1}\Bigr)\,,
  \qquad \tau \ge \tau_1\,,
\end{equation}
as long as $\alpha(\tau)^2 + e_{m,\delta}(\tau) \le M^2 \epsilon_1^2$.
Here the constant $C_{10}$ does not depend on $M$ nor on $\tau_1$. 
Thus we can choose $M$ large enough so that $M^2 > 2 C_{10}$,
and then $\tau_1$ large enough so that $e^{2\mu\tau_1} \ge M^2$. Under
these assumptions, we deduce from \eqref{long3} that $\alpha(\tau)^2
+ e_{m,\delta}(\tau) \le M^2 \epsilon_1^2 \le 1$ for all $\tau \ge \tau_1$,
and this in turn implies that estimates \eqref{long1}, \eqref{long2}
hold for all $\tau \ge 0$. In particular, we have $e_{m,\delta}(\tau)
\le \cE_{m,\delta}(\tau) \le C \epsilon_1^2\,e^{-2\mu\tau}$, and there
exists $\alpha_* \in \R$ such that $|\alpha(\tau) - \alpha_*| \le
C \epsilon_1 \,e^{-\eta\tau}$ for all $\tau \ge \tau_1$. Together with
the short time estimate \eqref{short1}, this proves \eqref{mainconv2}
in the case where $m \in (n/2+2\mu,n/2+\beta)$ and $m \le 2$. 

The final step consists in proving the exponential decay for
large times in the general case where $m > n/2$. This can be
done using the previous result and a simple interpolation
argument as in the proof of Theorem~\ref{main1}. We omit the
details. \QED

\begin{rem}\label{finalrem}
It is possible to relax considerably our assumptions \eqref{Nprop}
on the nonlinearity $N$ and to strengthen our convergence result
\eqref{mainconv2} by using additional functionals that control
derivatives of the solution $v(y,\tau)$. In view of \eqref{em3},
it is natural to consider the functional
\begin{equation}\label{Dmdef}
  D_{m,\delta}(\tau) \,=\, \frac12 \int (\delta+|y|^2)^m \Langle\nabla
  w(y,\tau),A(ye^{\tau/2})\nabla w(y,\tau)\Rangle \dd y\,,
\end{equation}
which is equivalent to $\cD_{m,\delta}(\tau)$ in \eqref{cDdef}. However,
controlling the time evolution of $D_{m,\delta}(\tau)$ requires
the additional hypothesis that the matrix $A(x)$ in \eqref{diffeq}
satisfies $x \cdot \nabla A \in L^\infty(\R^n)$. Such an assumption
is quite natural in our problem, but is not required in
Theorems~\ref{main1} and \ref{main2}. 
\end{rem}

%%%%%%%%%%%%%%%%%%%%%%%%%%%%%%%%%%%%%%%%%%%%%%%%%%%%%%%%%%%%%%%%%%%%%

\section{Appendix}\label{sec6}

\subsection{A generalized Young inequality}
\label{ssec61}

In this section, following \cite{LPSW}, we give a short proof of
Proposition~\ref{propkernel}. 

\begin{lem}\label{kappalem}
Under the assumptions of Proposition~\ref{propkernel}, we define, 
for any $y \in \S^{n-1} \subset \R^n$, 
\begin{equation}\label{defkappa2}
  \kappa_2 \,=\, \int_{\R^n} |k(x,y)|^{n/d} \,|x|^{-n^2/(dp')}\dd x\,,
  \qquad \hbox{where}\quad \frac1p + \frac{1}{p'} \,=\, 1\,.
\end{equation}
Then $\kappa_2 = \kappa_1$, where $\kappa_1$ is given by \eqref{kprop3}. 
\end{lem}

\begin{proof}
Proceeding as in \cite[Lemma~1]{LPSW}, we write $x = r \sigma$ and
$y = \rho \theta$, where $r = |x|$, $\rho = |y|$, and $\sigma,
\theta \in \S^{n-1}$. By rotation invariance, the definition
\eqref{kprop3} does not depend on the choice of $x \equiv \sigma
\in\S^{n-1}$. Thus, averaging over $\sigma$, we obtain 
\[
  \kappa_1 \,=\, \frac{1}{s_n} \int_{\S^{n-1}} \int_{\S^{n-1}}
  \int_0^\infty |k(\sigma,\rho \theta)|^{n/d} \,\rho^{n-1-n^2/(dq)} \dd \rho
  \dd \theta \dd \sigma\,,
\]
where $s_n = 2\pi^{n/2}\,\Gamma(n/2)^{-1}$ is the measure of
$\S^{n-1}$. We perform the change of variable $\rho = 1/r$ in the
inner integral, and use the fact that $|k(x,y)|^{n/d}$ is homogeneous
of degree $-n$. This gives
\begin{equation}\label{aux0}
  \kappa_1 \,=\, \frac{1}{s_n} \int_{\S^{n-1}} \int_{\S^{n-1}}
  \int_0^\infty |k(r\sigma,\theta)|^{n/d} \,r^{n-1-n^2/(dp')} \dd \rho
  \dd \sigma \dd \theta\,,
\end{equation}
because $n^2/(dq) = n - n^2/(dp')$ in view of \eqref{kprop3}.
The right-hand side of \eqref{aux0} is the average over $\theta \in
\S^{n-1}$ of the quantity \eqref{defkappa2}, which does not depend
on the choice of $y \equiv \theta \in \S^{n-1}$. This yields the
desired equality $\kappa_1 = \kappa_2$. 
\end{proof}

\noindent{\bf Proof of Proposition~\ref{propkernel}.}
We assume for definiteness that $1 < p < q < \infty$, which is the
most interesting situation. The other cases, where some of the
inequalities above are not strict, can be established by similar
(or simpler) arguments. If $f \in L^p(\R^n)$ and $g = \cK[f]$, we have
\[
  |g(x)| \,\le\, \int_{\R^n} \Bigl(|k(x,y)|^a |y|^{-b}\Bigr)\,
  \Bigl(|k(x,y)|^{1-a} |y|^b |f(y)|^{p/q}\Bigr)\, |f(y)|^{1-p/q}\dd y\,,
  \quad x \in \R^n\,,
\]
where $a = n/(dp')$ and $b = n^2/(dqp')$. We apply to the right-hand
side the trilinear H\"older inequality with exponents $p'$, $q$, and
$r := pq/(q-p)$, which satisfy $1/p' + 1/q + 1/r = 1$. This gives
\begin{equation}\label{aux1}
  |g(x)|^q \,\le\, I(x) \left(\int_{\R^n} |k(x,y)|^{(1-a)q} |y|^{bq}
  |f(y)|^p \dd y\right)\,\|f\|_{L^p(\R^n)}^{q-p}\,,
\end{equation}
where
\[
  I(x)^{p'/q} \,=\, \int_{\R^n} |k(x,y)|^{ap'} |y|^{-bp'} \dd y \,=\,
  \int_{\R^n} |k(x,y)|^{n/d} |y|^{-n^2/(dq)} \dd y\,.
\]
Applying the change of variables $y = |x| z$ in the last integral
and using the assumption that the expression $|k(x,y)|^{n/d}$ is
homogeneous of degree $-n$, we obtain
\[
  I(x)^{p'/q} \,=\, |x|^{n - n^2/dq} \int_{\R^n} \bigl|k(x,|x|z)\bigr|^{n/d}
  \,|z|^{-n^2/(dq)} \dd z \,=\, \kappa_1 \,|x|^{-n^2/(dq)}\,.
\]
We now replace this expression into \eqref{aux1} and integrate
over $x \in \R^n$, using Fubini's theorem to exchange the
integrals. Since $(1-a)q = n/d$ and $bq = n^2/(dp')$, this
gives
\begin{equation}\label{aux2}
  \|g\|_{L^q(\R^n)}^q \,\le\, \kappa_1^{q/p'} \,\|f\|_{L^p(\R^n)}^{q-p}
  \int_{\R^n} J(y) |y|^{n^2/(dp')} |f(y)|^p \dd y\,,
\end{equation}
where
\[
  J(y) \,=\, \int_{\R^n} |k(x,y)|^{n/d} |x|^{-n^2/(dp')} \dd x \,.
\]
As for the computation of $I$, we use the homogeneity of $k$ and the change of variable $z=x/|y|$ to obtain
\[
J(y)\,=\, \int_{\R^n} \big|k(z,y/|y|)\big|^{n/d} |z|^{-n^2/(dp')} |y|^{-n^2/(dp')} \dd z  \,=\, \kappa_2 \,|y|^{-n^2/(dp')}\,.
\]
Using Lemma~\ref{kappalem}, we conclude that $\|g\|_{L^q(\R^n)} \le
\kappa_1^{1/p'} \kappa_2^{1/q} \|f\|_{L^p(\R^n)} = \kappa_1^{d/n}
\|f\|_{L^p(\R^n)}$.
\QED

\subsection{On the divergence of localized vector fields}\label{ssec62}

Let $\chi : \R^n \to [0,1]$ be any smooth, compactly supported
function such that $\chi(x) = 1$ for $|x| \le 1$ and $\chi(x) = 0$
for $|x| \ge 2$. Given any $k \in \N^*$, we denote $\chi_k(x) =
\chi(x/k)$. 

\begin{lem}\label{intlem}
Assume that $g \in L^p(\R^n)^n$ for some $p \in [1,\infty)$ such that
$(n-1)p < n$. Then we have $\langle \div g\,,\,\chi_k \rangle \to 0$ as
$k \to +\infty$. In particular, if $\div g \in L^1(\R^n)$, then $\int_{\R^n}
\div g \dd x = 0$. 
\end{lem}

\begin{proof}
For any $k \ge 1$, we have
\begin{equation}\label{divg}
  \langle \div g\,,\,\chi_k \rangle \,=\, -\langle g\,,\,\nabla \chi_k
  \rangle \,=\, - \frac{1}{k}\int_{\R^n} g(x) \cdot \nabla\chi(x/k)
  \dd x\,.
\end{equation}
The integral in the right-hand side is easily estimated using H\"older's
inequality\:
\[
  \Bigl|\int_{\R^n} g(x) \cdot \nabla\chi(x/k)\dd x\Bigr| \,\le\,
  C \int_{|k| \le |x| \le 2|k|} |g(x)|\dd x \,\le\, C \|g\|_{L^p}
  (k^n)^{1-\frac1p}\,.
\]
Our assumption on $p$ ensures that $n(1-1/p) < 1$, hence the
last member of \eqref{divg} converges to zero as $k \to \infty$.
Finally, if $\div g \in L^1(\R^n)$, the first member of \eqref{divg}
converges to $\int_{\R^n}\div g\dd x$ by Lebesgue's dominated convergence
theorem. 
\end{proof}

\begin{lem}\label{divlem}
Let $n \ge 2$, $m \in (n/2,n/2+1)$, and assume that $f \in L^2(m)$
satisfies $\int_{\R^n}f \dd x = 0$. Then there exists $g \in L^2(m{-}1)^n$
such that $\div g = f$ and $\|g\|_{L^2(m-1)} \le C \|f\|_{L^2(m)}$.
\end{lem}

\begin{proof}
Under our assumptions on $f$, it is known that the elliptic equation
$\Delta u = f$ has a unique solution $u : \R^n \to \R$ that decays to
zero at infinity \cite{GT}. We take $g = \nabla u$. Using the explicit
form of the fundamental solution of the Laplace equation in $\R^n$,
we obtain the representations
\begin{equation}\label{ker1}
  g(x) \,=\, \frac{1}{s_n}\int_{\R^n}\frac{x-y}{|x-y|^n}\,f(y)\dd y
  \,=\, \frac{1}{s_n}\int_{\R^n}\biggl(\frac{x-y}{|x-y|^n}
  - \frac{x}{|x|^n}\biggr)f(y)\dd y\,,
\end{equation}
where $s_n$ is again the measure of the unit sphere $\S^{n-1}$. Since
$f \in L^2(\R^n)$, we can apply the Hardy-Littlewood-Sobolev
inequality to the first expression of $g$ in \eqref{ker1}, and we
deduce that $g \in L^p(\R^n)$ for $p = 2n/(n-2)$ when $n \ge 3$. If
$n \ge 2$, using the fact that $L^2(m) \hookrightarrow L^q(\R^2)$ for
$q \in (1,2)$, we obtain that $g \in L^p(\R^2)$ for $p \in
(2,\infty)$. In particular, we have in all cases
\begin{equation}\label{kerfirst}
  \int_{|x|\le1}|g(x)|^2 \dd x \,\le\, C \|f\|_{L^2(m)}^2\,.
\end{equation}

We next exploit the second expression of $g$ in \eqref{ker1}. We claim
that
\begin{equation}\label{ker2}
  \biggl|\frac{x-y}{|x-y|^n} - \frac{x}{|x|^n}\biggr| \,\le\,
  \frac{C|y|}{|x| |x-y|}\biggl(\frac{1}{|x-y|^{n-2}} + \frac{1}{|x|^{n-2}}
  \biggr)\,,
\end{equation}
for all $x,y \in \R^n$ with $x \neq 0$ and $x \neq y$. Equivalently,
\begin{equation}\label{ker3}
  \Bigl| |x|^n (x-y) - x\,|x-y|^n\Bigr| \,\le\, C |x|\,|y|\,|x-y|
  \,\Bigl(|x-y|^{n-2} + |x|^{n-2}\Bigr)\,,
\end{equation}
for all $x,y \in \R^n$. To establish \eqref{ker3}, we decompose
\begin{equation}\label{ker4}
  |x|^n (x{-}y) - x\,|x{-}y|^n \,=\, |x|^{n-1}\Bigl(|x|(x{-}y) - x\,|x{-}y|\Bigr)
  + x\,|x{-}y| \Bigl(|x|^{n-1} - |x{-}y|^{n-1}\Bigr)\,,
\end{equation}
and we use the following two elementary observations\:  

\smallskip\noindent
1. For any $x,z \in \R^n$ we have $\bigl||x|z - x|z|\bigr| \le
2 |z| |x-z|$. This can be proved by taking the square of both sides
and considering two cases according to whether $|x| \le 4 |z|$ or
$|x| \ge 4 |z|$.

\smallskip\noindent
2. For any $x,z \in \R^n$, we have
\[
  \Bigl| |x|^{n-1} - |z|^{n-1} \Bigr| \,\le\, \frac{n{-}1}{2}\,
  \bigl| |x| - |z| \bigr|\,\Bigl(|x|^{n-2} + |z|^{n-2}\Bigr)\,.
\]
Indeed the map $t \mapsto h(t) = (n{-}1)t^{n-2}$ is convex on $\R_+$, so
that for all $b \ge a \ge 0$ we have $\int_a^b h(t)\dd t \le (b-a)
(h(a)+h(b))/2$, which gives the result if $a = \min(|x|,|z|)$,
$b = \max(|x|,|z|)$. 

\smallskip\noindent
Applying these elementary estimates with $z = x-y$, we can bound both
terms in the right-hand side of \eqref{ker4}, and we arrive at
\eqref{ker3}. 

Now, in view of \eqref{ker1}, \eqref{ker2}, we have $|x|^{m-1} |g(x)|
\le C\int_{\R^n} k(x,y) |y|^m |f(y)|\dd y$, where 
\[
  k(x,y) \,=\, \frac{|x|^{m-2}}{|x-y| \,|y|^{m-1}}\biggl(\frac{1}{|x-y|^{n-2}}
  + \frac{1}{|x|^{n-2}}\biggr)\,.
\]
The kernel $k(x,y)$ is homogeneous of degree $-n$ and invariant under
rotations in $\R^n$. Moreover, if $|x| = 1$, the assumption that
$m \in (n/2,n/2 + 1)$ ensures that $\int_{\R^n}k(x,y) |y|^{-n/2} \dd y
< \infty$. Applying Proposition~\ref{propkernel} with $d = n$ and
$p = q = 2$, we deduce that
\[
  \int_{\R^n} |x|^{2m-2} |g(x)|^2 \dd x \,\le\, C
  \int_{\R^n} |y|^{2m} |f(y)|^2 \dd y\,,
\]
and combining this estimate with \eqref{kerfirst} we conclude that
$\|g\|_{L^2(m-1)} \le C \|f\|_{L^2(m)}$.
\end{proof}

\subsection{On the optimality of Proposition~\ref{Kernelprop1}}
\label{ssec63}

We show here using an explicit example that the assumption $m < n/2+\beta$
in Proposition~\ref{Kernelprop1} cannot be relaxed. Given $a,b > 0$, we
consider the functions $u,f : \R^n \to \R$ defined by
\begin{equation}\label{ufdef}
  u(x) \,=\, \frac{x_1}{(1+|x|^2)^a}\,, \qquad 
  f(x) \,=\, -\div\bigl(A_b(x)\nabla u(x)\bigr)\,,
\end{equation}
where $A_b$ is the Meyers-Serrin matrix \eqref{Abdef}. We have
\[
  \nabla u(x) \,=\, \frac{e_1}{(1+|x|^2)^a} - \frac{2ax_1x}{(1+|x|^2)^{a+1}}\,,
  \qquad \hbox{where} \quad e_1 \,=\, (1,0,\dots,0)\,,
\]
and since $A_b(x)x = x$ we find
\[
  A_b(x)\nabla u(x) \,=\, \frac{1}{(1+|x|^2)^a} \Bigl(b e_1 +
  (1-b) \frac{x_1 x}{|x|^2}\Bigr) - \frac{2ax_1x}{(1+|x|^2)^{a+1}}\,.
\]
Taking the divergence with respect to $x$, we arrive at
\[
  f(x) \,=\, -\frac{(1{-}b)(n{-}1)}{(1+|x|^2)^a}\,\frac{x_1}{|x|^2}
  + \frac{2a(n{+}2)x_1}{(1+|x|^2)^{a+1}} - \frac{4a(a{+}1)x_1|x|^2}{
  (1+|x|^2)^{a+2}}\,, \qquad x \in \R^n\,.
\]

For simplicity, we assume henceforth that $n \ge 3$, so that $f \in
L^2_\loc(\R^n)$. As $|x| \to \infty$, we have
\begin{equation}\label{fasym}
  f(x) \,=\, x_1 \Bigl(\frac{c}{|x|^{2a+2}} + \cO\Bigl(\frac{1}{|x|^{2a+4}}
  \Bigr)\Bigr)\,, \qquad \hbox{as} \quad |x| \to +\infty\,,
\end{equation}
where $c = -(1-b)(n-1) + 2an - 4a^2$. The idea is now to choose the
parameters $a,b$ so that $c = 0$, in order to maximize the decay of $f$.
For instance, we can take
\begin{equation}\label{achoice}
  a \,=\, \frac14 \Bigl(n + \sqrt{n^2 - 4(1{-}b)(n{-}1)}\Bigr) \,=\,
  \frac14 \Bigl(n + \sqrt{(n{-}2)^2 + 4b(n{-}1)}\Bigr)\,.
\end{equation}
With this choice, given $m > n/2$, it follows from \eqref{ufdef},
\eqref{fasym} that
\begin{equation}\label{ufcond}
  \begin{split}
  |x|^m f \in L^2(\R^n) \quad &\hbox{if and only if} \quad 2a > n/2 + m - 3\,, \\
  |x|^{m-2} u \in L^2(\R^n) \quad &\hbox{if and only if} \quad 2a > n/2 + m - 1\,.
  \end{split}
\end{equation}
Under the first condition in \eqref{ufcond}, we also have $\int_{\R^n} f(x)
\dd x = 0$ since $f$ is odd, hence $f \in L^2_0(m)$. 

According to \eqref{ufcond}, the pair $(u,f)$ violates inequality
\eqref{uweight} with $p = 2$, $s = 0$ provided $m > n/2$ and
\begin{equation}\label{acond}
  n/2 + m - 3 \,<\, 2a \,<\, n/2 + m - 1\,.
\end{equation}
For instance, if $n = 3$ and $m = 2$, we have $1/2 < 2a < 5/2$ by
\eqref{achoice} if $b > 0$ is sufficiently small, and it follows that
$f \in L^2(m)$, $\int_{\R^3} f(x)\dd x = 0$, and yet $u \notin
L^2(\R^3)$. The explanation is that the H\"older exponent $\beta$ in
Proposition~\ref{GHolder} tends to zero as $b \to 0$ in the case of
the Meyers-Serrin operator, see Remark~\ref{betarem}, and that the
value $m = 2$ is not allowed in Proposition~\ref{Kernelprop1}
if $n = 3$ and $\beta < 1/2$. More generally, if $n \ge 3$ and
$n/2 < m < n/2 + 1$, we can choose $b > 0$ small enough so that
inequalities \eqref{acond} hold, which implies the failure of estimate
\eqref{uweight} with $p = 2$, $s = 0$; but it follows from \eqref{achoice}
and \eqref{betaupper} that $2a \ge n-1 + \beta$, hence the second
inequality in \eqref{acond} implies that $m > n/2 + \beta$. This shows
that the assumption $m < n/2+\beta$ in Proposition~\ref{Kernelprop1} is
sharp in the case of the Meyers-Serrin matrix \eqref{Abdef}, at least
if the quantity $\beta$ is understood as given by the right-hand side
of \eqref{betaupper}. 

%%%%%%%%%%%%%%%%%%%%%%%%%%%%%%%%%%%%%%%%%%%%%%%%%%%%%%%%%%%%%%%%%%%%%

\vspace{0.5cm}\noindent
{\bf Thierry Gallay\:}\\
Institut Fourier, Universit\'e Grenoble Alpes et CNRS, 100 rue des Maths, 
38610 Gi\`eres, France\\
{\tt Thierry.Gallay@univ-grenoble-alpes.fr}

\vspace{0.3cm}\noindent
{\bf Romain Joly\:}\\
Institut Fourier, Universit\'e Grenoble Alpes et CNRS, 100 rue des Maths, 
38610 Gi\`eres, France\\
{\tt Romain.Joly@univ-grenoble-alpes.fr}

\vspace{0.3cm}\noindent
{\bf Genevi\`eve Raugel\:} ($\dagger$ May 10, 2019)\\
CNRS et Universit\'e Paris-Saclay, D\'epartement de Math\'ematiques, 
91405 Orsay, France

\end{document}